\documentclass[reqno]{amsart}
\usepackage{hyperref}
\usepackage{graphics}
\usepackage{amssymb}
\usepackage{enumerate}
\usepackage{color}
\usepackage[left=3.00cm, right=3.00cm, top=2.00cm, bottom=2.00cm]{geometry}

\allowdisplaybreaks
\begin{document} 

\title[Decay character and semi-linear structurally damped evolution equations]{New definitions of decay indicators and critical exponent for fractional semi-linear structurally damped evolution equations on the Heisenberg group}
\author[P.D. An, C.T. Anh, P.T. Duong]{Phan Duc An, Cung The Anh$^{\natural}$}

\address{Phan Duc An\hfill\break
Department of Mathematics, Hanoi National University of Education \hfill\break
136 Xuan Thuy, Cau Giay, Hanoi, Vietnam}
\email{anan26042001@gmail.com}

\address{Cung The Anh \hfill\break
Department of Mathematics, Hanoi National University of Education \hfill\break
136 Xuan Thuy, Cau Giay, Hanoi, Vietnam}
\email{anhctmath@hnue.edu.vn}

\keywords{Heisenberg group, Decay indicator, Decay character, Semi-linear structurally damped evolution equations, Global existence of small data solution, Blow-up, Critical, exponent}

\begin{abstract}
In this paper, we introduce new definitions for the lower and upper decay indicators, along with the decay character on the Heisenberg group. Based on these definitions, we analyze the decay rates of solutions to the linear fractional diffusion equation and evaluate the decay character within specific function spaces. Leveraging these insights, we further investigate the Cauchy problem for the linear structurally damped evolution equation on the Heisenberg group 
$$
\partial_t^2 u(t, \eta)+\left(-\Delta_{\mathrm{H}}\right)^{\delta_1} u(t, \eta)+\left(-\Delta_{\mathrm{H}}\right)^{\delta_2} \partial_t u(t, \eta)=0, \quad \text{ with } \delta_1 \in \left[0, \tfrac{\delta_2}{2}\right].
$$
Furthermore, we examine the decay rates of solutions in homogeneous fractional Sobolev spaces and their derivatives, which are expressed in terms of the decay character of the initial data $u_0(\eta)=u(0, \eta)$ and $u_1(\eta)=\partial_t u(0, \eta)$, to unify existing results and include new ones not previously discussed. We also focus on the existence and decay rates of global solutions (in time) for the corresponding semi-linear problem with the nonlinear power type term $\left|u(t, \eta)\right|^p$, including the derivation of conditions ensuring the existence of solutions for exponents $p>p_{\mathrm{Fuji}}\left(\frac{Q-2 \omega \delta_2}{\omega \delta_1}\right):= 1+\frac{2 \omega \delta_1}{Q-2 \omega \delta_2}$, with $\omega:=\frac{Q}{Q+\min \left\{ r_\mathrm{H}\left(u_0\right), r_\mathrm{H}\left(u_1\right)-2\delta_2 \right\}+2\delta_2}$, as well as the study of the critical case $p:=p_{\mathrm{Fuji}}\left(\frac{Q-2 \omega \delta_2}{\omega \delta_1}\right)$. Furthermore, we will construct novel test functions to handle the fractional Laplace operators $\left(-\Delta_{\mathrm{H}}\right)^{\delta_1}$ and $\left(-\Delta_{\mathrm{H}}\right)^{\delta_2}$, which are well-known nonlocal operators on the Heisenberg group. Based on this approach, we establish blow-up results for solutions to the semi-linear problem to determine the critical exponent $p:=p_{\mathrm{Fuji}}\left(\tfrac{Q+m\gamma-2 m \delta_2}{m \delta_1}\right):=1+\tfrac{2 m \delta_1}{Q+m\gamma-2 m \delta_2}$ in the case where the initial data belong to the space $\dot{H}_m^{-\gamma}\left(\mathbf{H}_n\right)$, with $m \in (1, 2]$ and $\gamma \in\left[0, Q-\frac{Q}{m}\right)$.
\end{abstract}

\maketitle
\numberwithin{equation}{section}
\newtheorem{theorem}{Theorem}[section]
\newtheorem{remark}{Remark}[section]
\newtheorem{definition}{Definition}[section]
\newtheorem{lemma}[theorem]{Lemma}
\newtheorem{corollary}[theorem]{Corollary}
\newtheorem{proposition}[theorem]{Proposition}
\newtheorem{example}[theorem]{Example}

\section{introduction}
\subsection{Background of the damped wave equations and decay character}
To explore the classical Euclidean framework, consider the semi-linear damped wave equation on $\mathbb{R}^n$ with a power-type nonlinearity
\begin{equation} \label{eq:Matsumura1976}
\begin{cases}
\partial_t^2 u(t, x) -\Delta u(t, x)+\partial_t u(t, x)=\left|u(t, x)\right|^p, & (t, x) \in[0, \infty) \times \mathbb{R}^n, \quad n \in \mathbb{N}^*, \quad p>1, \\ u(0, x)=u_0(x), \quad \partial_t u(0, x)=u_1(x), & x \in \mathbb{R}^n, \quad n \in \mathbb{N}^*,
\end{cases}
\end{equation}
where $\Delta$ denotes the Laplacian operator in $\mathbb{R}^n$. When the initial conditions also lie in the $L^1\left(\mathbb{R}^n\right)$ space, the global existence or finite time blow-up of solutions to equation \eqref{eq:Matsumura1976}, depending on the value of the critical exponent, has been extensively investigated in \cite{Ikehata2005, Matsumura1976, Todorova2001, Zhang2001} and related works. The critical exponent determines the dividing line for the exponent $p$ that governs whether solutions persist globally in time or blow up in finite time, particularly for Sobolev solutions with small initial data. This critical exponent, known as the Fujita exponent, is defined as $p_{\mathrm{Fuji}}(n) := 1 + \frac{2}{n}$ (see \cite{Ikehata2005, Matsumura1976, Todorova2001, Zhang2001} for details). More specifically:
\begin{itemize}
\item For dimensions $n = 1, 2$, Matsumura, in his pioneering work \cite{Matsumura1976}, demonstrated that solutions with small initial data exist globally when $p > p_{\mathrm{Fuji}}(n)$.
\item For all $n \geq 1$, Todorova and Yordanov \cite{Todorova2001} established global existence for $p > p_{\mathrm{Fuji}}(n)$ (assuming the initial data is compactly supported) and showed finite-time blow-up in the subcritical case where $1 < p < p_{\mathrm{Fuji}}(n)$.
\item In the critical case $p = p_{\mathrm{Fuji}}(n)$, Zhang \cite{Zhang2001} proved the occurrence of blow-up.
\end{itemize}

In the following years, researchers have increasingly focused on identifying new critical exponents for classical semi-linear damped wave equations on $\mathbb{R}^n$ in varying contexts. For example, when the Cauchy problem \eqref{eq:Matsumura1976} is considered with initial data additionally belonging to $L^m$ spaces, where $m \in [1,2)$, the critical exponent shifts, leading to a modified Fujita exponent defined by: $p_{\mathrm{Fuji}}\left(\frac{n}{m}\right):=1+\frac{2 m}{n}$.

In contrast to the $L^1\left(\mathbb{R}^n\right)$ case, the presence of $L^m\left(\mathbb{R}^n\right)$-regular initial data guarantees the unique global (in time) solution precisely at the critical point $p_{\mathrm{Fuji}}\left(\frac{n}{m}\right) = 1 + \frac{2m}{n}$. This distinction highlights the fundamental difference between $L^1\left(\mathbb{R}^n\right)$ and $L^m\left(\mathbb{R}^n\right)$ regular data. For a comprehensive analysis of the critical exponent $p_{\mathrm{Fuji}}\left(\frac{n}{m}\right)$ concerning solutions to semi-linear wave equations with $L^m$-regular data, see \cite{Ikeda2019, Ikehata2002, Nakao1993} and related works.

A number of recent publications (see, e.g., \cite{DAbbicco2017, Dao2019}) have explored the most typical scenario of power-type nonlinearities $\left|u(t, x)\right|^p$ in the context of semi-linear structurally damped $\sigma$-evolution equations, focusing in particular on the following Cauchy problem:
\begin{equation} \label{eq:evolution}
\begin{cases}
\partial_t^2 u(t, x)+(-\Delta)^\sigma u(t, x)+ (-\Delta)^\delta \partial_t u(t, x)=\left|u(t, x)\right|^p, & (t, x) \in [0, \infty) \times \mathbb{R}^n, \\ u(0, x)=u_0(x), \quad \partial_t u(0, x)=u_1(x), & x \in \mathbb{R}^n, \quad n \in \mathbb{N}^*,
\end{cases}
\end{equation}
where $\sigma \geq 1$, $\delta \in \left[0, \frac{\sigma}{2}\right]$, and $p>1$. In their work, D'Abbicco and Ebert \cite{DAbbicco2017} addressed a broader class of the semilinear problem \eqref{eq:evolution} involving integer parameters $\sigma \in \mathbb{N} \backslash\{0\}$ and $\delta \in \mathbb{N}$, where $\delta \in \left(0, \frac{\sigma}{2}\right]$. They succeeded in establishing the global (in time) existence of small data solutions provided that $p>p_{\mathrm{Fuji}}\left(\frac{n-2 m \delta}{m \sigma}\right):= 1+\frac{2 m \sigma}{n-2 m \delta}$, under the assumption that the initial data belong to $\left(L^m\left(\mathbb{R}^n\right) \cap H^\sigma\left(\mathbb{R}^n\right)\right) \times\left(L^m\left(\mathbb{R}^n\right) \cap L^2\left(\mathbb{R}^n\right)\right)$, with $m \in(1,2]$. Later, Dao and Reissig \cite{Dao2019} extended this result to cover all fractional values of $\sigma \geq 1$ and $\delta \in[0, \sigma)$. In addition, in the same work \cite{Dao2021}, they also established a blow-up results and identified the corresponding critical exponent in the case where $u_0=0$ and $u_1 \in L^1\left(\mathbb{R}^n\right)$ satisfies $\int_{\mathbb{R}^n} u_1 \mathrm{d} x>\epsilon$, where $\epsilon$ is a suitable nonnegative constant. Moreover, \cite{DAbbicco2017} also provided a nonexistence result in the subcritical range $p<p_{\mathrm{Fuji}}\left(\frac{n-2 m \delta}{m \sigma}\right)$ by applying a modified test function method. This was achieved by selecting $u_0=0$ and taking $u_1 \geq$ $\varepsilon(1+|x|)^{-\mu}$, where the exponent $\mu=\mu(p)$ satisfies $\mu>\frac{n}{m}$ (see also the remark in \cite[Section 2.5]{DAbbicco2017}).

The study of semi-linear damped wave equations has been further developed within the non-Euclidean framework. Over the last few decades, numerous studies have examined nonlinear partial differential equations (PDEs) in non-Euclidean settings. For example, the semi-linear wave equation, both with and without damping, has been analyzed on the Heisenberg group in \cite{Georgiev2020, Nachman1982, Ruzhansky2018}. For graded groups, we refer to recent works in \cite{Georgiev2021, Ruzhansky2021, Ruzhansky2018, Taranto2018}. For wave equations on compact Lie groups, consult \cite{Bhardwaj2023, Dasgupta2023, Garetto2015, Palmieriblowup2021, PalmieriSemilinear2021, Palmieri2022} and for Riemannian symmetric spaces of non-compact type, refer to \cite{Anker2024, Anker2014, Kassymov2024, Zhang2021}.

In particular, the semi-linear damped wave equation on the Heisenberg group $\mathbf{H}_n$ is given by
$$
\begin{cases}
\partial_t^2 u(t, \eta)-\Delta_{\mathrm{H}} u(t, \eta)+\partial_t u(t, \eta)=\left|u(t, \eta)\right|^p, & (t, \eta) \in[0, \infty) \times \mathbf{H}_n, \quad n \in \mathbb{N}^*, \quad p>1, \\ u(0, \eta)=u_0(\eta), \quad \partial_t u(0, \eta)=u_1(\eta), & \eta \in \mathbf{H}_n, \quad n \in \mathbb{N}^*,
\end{cases}
$$
where $\Delta_{\mathrm{H}}$ is the sub-Laplacian on the Heisenberg group $\mathbf{H}_n$. It has been shown in \cite{Georgiev2020} that the critical exponent is determined by
\begin{equation} \label{Fujita}
p_{\mathrm{Fuji}}(Q) := 1+\tfrac{2}{Q},
\end{equation}
where $Q := 2n + 2$ represents the homogeneous dimension of $\mathbf{H}_n$. Interestingly, \eqref{Fujita} is also the Fujita exponent for the semi-linear heat equation
$$
\begin{cases}
\partial_t u(t, \eta)-\Delta_{\mathrm{H}} u(t, \eta) =\left|u(t, \eta)\right|^p, & (t, \eta) \in[0, \infty) \times \mathbf{H}_n, \quad n \in \mathbb{N}^*, \quad p>1, \\ u(0, \eta) =u_0(\eta), & \eta \in \mathbf{H}_n, \quad n \in \mathbb{N}^*.
\end{cases}
$$
This topic has been discussed in \cite{Georgiev2020, Georgiev2021, Ruzhansky2022}. However, in the case of a compact Lie group $\mathbb{G}$, it has been shown in \cite{ Palmieriblowup2021} that $p_{\mathrm{Fuji}}(0) := \infty$ is the critical exponent for the solution to the semi-linear damped wave equation on $\mathbb{G}$. We cite \cite{Ruzhansky2018} for a global existence result with small data in the broader context of graded Lie groups for the semi-linear damped wave equation involving a Rockland operator with an additional mass term.

To examine the decay properties of parabolic conservation laws and the Navier-Stokes equations, Schonbek introduced the Fourier Splitting Method through a series of studies \cite{Schonbek1980, Schonbek1985, Schonbek1986}. Bjorland and Schonbek later advanced this method in \cite{Bjorland2009}, who introduced the idea of a decay character. This character represented as $r^*=r^*\left(u_0\right)$ for an initial function $u_0 \in L^2\left(\mathbb{R}^n\right)$, offers a more accurate decay rate for solutions of the heat equation with the given initial data $u_0$. Further research on decay characters was carried out by Niche and Schonbek \cite{Niche2015} and Brandolese \cite{Brandolese2016}, yielding more refined decay estimates for solutions to dissipative evolution Cauchy problems with only $L^2\left(\mathbb{R}^n\right)$ initial data. Moreover, decay characters can be assessed for various function spaces, and Brandolese \cite{Brandolese2016} established criteria for finite decay characters in the context of Besov spaces.

Recent studies have utilized the concept of decay characters to analyze the semi-linear structurally damped evolution equations
\begin{equation} \label{eq:semi_linear}
\begin{cases}
\partial_t^2 u(t, x)+(-\Delta)^\sigma u(t, x)+ b(t)(-\Delta)^\delta \partial_t u(t, x)=\left|u(t, x)\right|^p, & (t, x) \in [0, \infty) \times \mathbb{R}^n, \\ u(0, x)=u_0(x), \quad \partial_t u(0, x)=u_1(x), & x \in \mathbb{R}^n, \quad n \in \mathbb{N}^*,
\end{cases}
\end{equation}
where $\sigma \geq 1, \delta \in [0, \sigma], p>1$ and $n \in \mathbb{N}^*$, along with the corresponding linear Cauchy problem
\begin{equation} \label{eq:linear}
\begin{cases}
\partial_t^2 u(t, x)+(-\Delta)^\sigma u(t, x)+ b(t)(-\Delta)^\delta \partial_t u(t, x)=0, & (t, x) \in [0, \infty) \times \mathbb{R}^n, \\ u(0, x)=u_0(x), \quad \partial_t u(0, x)=u_1(x), & x \in \mathbb{R}^n, \quad n \in \mathbb{N}^*,
\end{cases}
\end{equation}
Researchers such as Armando and Niche \cite{Niche2022} (for $\sigma=1, \delta=0, b(t)=1$), Anh, Duong, Loc \cite{Loc2023} (for $\sigma \geq 1, \delta \in[0, \sigma], b(t)=1$), and Anh, An, Duong \cite{AAD} (for $\sigma \geq 1, \delta \in\left[0, \frac{\sigma}{2}\right], b(t)$ is effective) have examined the asymptotic behavior of solutions to \eqref{eq:linear} using decay characters and established the global existence of small data solutions for the semi-linear equations \eqref{eq:semi_linear}.
\subsection{Main purpose of this paper}
In this paper, we introduce and develop a new concept of the lower and upper decay indicators, along with the decay character, on the Heisenberg group, a Lie group with a distinct geometric and algebraic structure, widely used in the theory of partial differential equations and harmonic analysis. We also focus on investigating the decay rates of solutions to the linear fractional diffusion equation based on these decay indicators and the decay character. Additionally, we estimate the decay character when the initial conditions are restricted to specific function spaces. Building on these insights, we extend this framework to explore the Cauchy problem for the fractional semi-linear structurally damped evolution equation on the Heisenberg group, expressed as follows:
\begin{equation} \label{eq:1.0}
\begin{cases}
\partial_t^2 u(t, \eta)+\left(-\Delta_{\mathrm{H}}\right)^{\delta_1} u(t, \eta)+\left(-\Delta_{\mathrm{H}}\right)^{\delta_2} \partial_t u(t, \eta)=\left|u(t, \eta)\right|^p, & (t, \eta) \in[0, \infty) \times \mathbf{H}_n, \quad n \in \mathbb{N}^*, \\ u(0, \eta)=u_0(\eta), \quad \partial_t u(0, \eta)=u_1(\eta), & \eta \in \mathbf{H}_n, \quad n \in \mathbb{N}^*,
\end{cases}
\end{equation}
where $\delta_2 \in \left[0, \frac{\delta_1}{2}\right]$ and $p>1$. Accordingly, the corresponding linear Cauchy problem for \eqref{eq:1.0} is given by
\begin{equation} \label{eq:1.1}
\begin{cases}
\partial_t^2 u(t, \eta)+\left(-\Delta_{\mathrm{H}}\right)^{\delta_1}u(t, \eta)+\left(-\Delta_{\mathrm{H}}\right)^{\delta_2}\partial_t u(t, \eta)=0, & (t, \eta) \in[0, \infty) \times \mathbf{H}_n, \quad n \in \mathbb{N}^*, \\ u(0, \eta)=u_0(\eta), \quad \partial_t u(0, \eta)=u_1(\eta), & \eta \in \mathbf{H}_n, \quad n \in \mathbb{N}^*.
\end{cases}
\end{equation}
As far as we are aware, the utilization of the decay character to analyze the subelliptic structurally damped evolution equation \eqref{eq:1.0} on the Heisenberg group, with initial data confined solely to the Sobolev space $H^\alpha\left(\mathbf{H}_n\right)$, where $\alpha \geq 0$ and without additional restrictions on other spaces, has not been explored in existing research, not even for the linear Cauchy problem \eqref{eq:1.1}. This presents an interesting and viable opportunity to employ the decay character as a unified framework for studying essential qualitative properties, including global-in-time well-posedness, blow-up criteria, decay rates, and the asymptotic behavior of solutions to the subelliptic structurally damped evolution equation on $\mathbf{H}_n$. By merely restricting the initial data to the Sobolev space $H^\alpha\left(\mathbf{H}_n\right)$, where $\alpha \geq 0$, without imposing additional constraints on other spaces, this approach unifies the derivation of known results with new findings, eliminating the need for separate analyses based on additional intersections with specific functional spaces.

Another aim of this paper is to leverage the decay character to explore and identify critical exponents associated with the Cauchy problem for the fractional semi-linear structurally damped evolution equation \eqref{eq:1.0}, where the initial data is restricted exclusively to the Sobolev spaces $H^{\delta_1}\left(\mathbf{H}_n\right)$ and $L^2\left(\mathbf{H}_n\right)$. In particular, under additional assumptions on the initial data $\left(u_0, u_1\right) \in H^{\delta_1}\left(\mathbf{H}_n\right) \times L^2\left(\mathbf{H}_n\right)$, with $r_{\mathrm{H}}\left(u_0\right)$ and $r_{\mathrm{H}}\left(u_1\right)$ chosen such that $-\frac{Q}{2} \leq \min \left\{ r_\mathrm{H}\left(u_0\right), r_\mathrm{H}\left(u_1\right)-2\delta_2 \right\} \leq -2\delta_2$ and $P_{r_{\mathrm{H}}}\left(u_0\right)_{+}, P_{r_{\mathrm{H}}}\left(u_1\right)_{+}<\infty$, we derive a new critical exponent for equation \eqref{eq:1.0}, given by
$$
p_{\mathrm{Fuji}}\left(\tfrac{Q-2 \omega \delta_2}{\omega \delta_1}\right):= 1+\tfrac{2 \omega \delta_1}{Q-2 \omega \delta_2}, \quad \text{ with } \omega:=\tfrac{Q}{Q+\min \left\{ r_\mathrm{H}\left(u_0\right), r_\mathrm{H}\left(u_1\right)-2\delta_2 \right\}+2\delta_2}.
$$
This critical exponent is especially important because it also applies to cases where the initial data satisfy
\begin{align*}
&\left(u_0, u_1\right) \in\left(H^{\delta_1}\left(\mathbf{H}_n\right) \cap L^m\left(\mathbf{H}_n\right)\right) \times\left(L^2\left(\mathbf{H}_n\right) \cap L^m\left(\mathbf{H}_n\right)\right), \text{ with } m \in[1,2]; \\
&\left(u_0, u_1\right) \in\left(H^{\delta_1}\left(\mathbf{H}_n\right) \cap \dot{H}^{-\gamma}\left(\mathbf{H}_n\right)\right) \times\left(L^2\left(\mathbf{H}_n\right) \cap \dot{H}^{-\gamma}\left(\mathbf{H}_n\right)\right), \text{ with } \gamma \in\left[0, \tfrac{Q}{2}\right]; \\
&\left(u_0, u_1\right) \in H_{\sigma \psi(t, \cdot)}^{\delta_1}\left(\mathbf{H}_n\right) \times L_{\sigma \psi(t, \cdot)}^2\left(\mathbf{H}_n\right), \text{ with } \sigma>0 \text{ and }
\psi(t, \eta) := \tfrac{|x|^2+|y|^2+4|\tau|}{8(1+t)}; \\
&\left(u_0, u_1\right) \in\left(H^{\delta_1}\left(\mathbf{H}_n\right) \cap \mathcal{Y}^q\left(\mathbf{H}_n\right)\right) \times\left(L^2\left(\mathbf{H}_n\right) \cap \mathcal{Y}^q\left(\mathbf{H}_n\right)\right), \text{ with } q \in \left[0, \tfrac{Q}{2}\right] ; \\
&\left(u_0, u_1\right) \in\left(H^{\delta_1}\left(\mathbf{H}_n\right) \cap \dot{H}_m^{-\gamma}\left(\mathbf{H}_n\right)\right) \times\left(L^2\left(\mathbf{H}_n\right) \cap \dot{H}_m^{-\gamma}\left(\mathbf{H}_n\right)\right), \text{ with } m \in (1, 2] \text{ and } \gamma \in\left[0, Q-\tfrac{Q}{m}\right)
\end{align*}
and this exponent also serves as the generalized critical exponent when the initial conditions belong to other specific function spaces. Moreover, we also investigate the critical case $p:=p_{\mathrm{Fuji}}\left(\tfrac{Q-2 \omega \delta_2}{\omega \delta_1}\right)$ and demonstrate that when $\min \left\{ r_\mathrm{H}\left(u_0\right), r_\mathrm{H}\left(u_1\right)-2\delta_2 \right\}<-2\delta_2$ the critical exponent falls within the range of global existence results. In contrast, when $\min \left\{ r_\mathrm{H}\left(u_0\right), r_\mathrm{H}\left(u_1\right)-2\delta_2 \right\}=-2\delta_2$ the critical exponent corresponds to blow-up phenomena. This means that for $p:=p_{\mathrm{Fuji}}\left(\frac{Q-2 \omega \delta_2}{\omega \delta_1}\right):= 1+\frac{2 \omega \delta_1}{Q-2 \omega \delta_2}$ the critical exponent leads to global existence, whereas for $p:=p_{\mathrm{Fuji}}\left(\frac{Q-2 \delta_2}{\delta_1}\right):= 1+\frac{2 \delta_1}{Q-2 \delta_2}$ the critical exponent results in blow-up. To identify the critical exponent, we develop new test functions designed to effectively handle the fractional Laplace operators $\left(-\Delta_{\mathrm{H}}\right)^{\delta_1}$ and $\left(-\Delta_{\mathrm{H}}\right)^{\delta_2}$, which are prominent nonlocal operators on the Heisenberg group. This strategy enables us to derive blow-up results for the associated semi-linear problem and, in turn, determine the critical threshold $p:=p_{\mathrm{Fuji}}\left(\tfrac{Q+m\gamma-2 m \delta_2}{m \delta_1}\right):=1+\tfrac{2 m \delta_1}{Q+m\gamma-2 m \delta_2}$ in the case where the initial data $\left(u_0, u_1\right) \in\dot{H}_m^{-\gamma}\left(\mathbf{H}_n\right) \times\dot{H}_m^{-\gamma}\left(\mathbf{H}_n\right)$, with $m \in (1, 2]$ and $\gamma \in\left[0, Q-\frac{Q}{m}\right)$.

\subsection{Notations}
Let $P_{r_\mathrm{H}}\left(u_0\right)_{-}=P_{r_\mathrm{H}}^0\left(u_0\right)_{-}, P_{r_\mathrm{H}}\left(u_0\right)_{+}=P_{r_\mathrm{H}}^0\left(u_0\right)_{+}$ be the lower and upper decay indicators of the function $u_0 \in L^2\left(\mathbf{H}_n\right)$. Definitions of the lower and upper decay indicators will be given in details in Definition \ref{definition_3.7}. Roughly speaking, The lower and upper decay indicators $P_{r_\mathrm{H}}\left(u_0\right)_{-}$ and $P_{r_\mathrm{H}}\left(u_0\right)_{+}$ describe a tool for measuring the decay of the function $u_0$ on the Heisenberg group through its Fourier transform. Let $r_\mathrm{H}^*\left(u_0\right)$ be the decay character of the function $u_0$ (see Definition \ref{definition_3.9}). For a given parameter $r_\mathrm{H}\left(u_0\right)\in \left[-\frac{Q}{2}, \infty\right]$ and $\kappa \in \mathbb{R}$, the quantities $\|h\|_{P, \kappa}$ and $\|h\|_{P, 0}$ denote the following norms of a function $h$ in $H^\kappa\left(\mathbf{H}_n\right)$ and $L^2\left(\mathbf{H}_n\right)$, respectively: 
$$ 
\|h\|_{P, \kappa}:=\|h\|_{\dot{H}^\kappa\left(\mathbf{H}_n\right)}+\|h\|_{L^2\left(\mathbf{H}_n\right)}+\left(P_{r_\mathrm{H}\left(h\right)}\left(h\right)\right)^{\frac{1}{2}}, \quad
\|h\|_{P, 0}:=\|h\|_{L^2\left(\mathbf{H}_n\right)}+\left(P_{r_\mathrm{H}\left(h\right)}\left(h\right)\right)^{\frac{1}{2}}.
$$

Throughout the paper, for nonnegative functions $f(t), g(t)$ the notation $f(t) \lesssim g(t)$ is used to denote the inequalities $f(t) \leq C g(t)$ that are satisfied uniformly for all $t>0$, with a positive constant $C$.

\subsection{Structure of the paper} In Section \ref{section_2}, we review some essential concepts related to the Heisenberg group, specifically the definition of the Heisenberg group (Section \ref{section_2.1}), the group Fourier transform on the Heisenberg group (Section \ref{section_2.2}), Hermite functions and their properties (Section \ref{section_2.3}) and some auxiliary inequalities (Section \ref{section_2.4}). In Section \ref{section_3}, we introduce a new definition of the decay character on the Heisenberg group (Section \ref{section_3.2}), investigate the decay characterization of solutions to the linear fractional diffusion equation on this group (Section \ref{section_3.3}) and compute the decay character in several function spaces on the Heisenberg group (Section \ref{section_3.4}). Section \ref{section_4} provides detailed proof of Theorem \ref{theorem:1.2} in turn, to derive linear decay estimates with corresponding rates depending on $r_\mathrm{H}$ for the problem \eqref{eq:1.1}. In Section \ref{section_5}, we prove the existence of global solutions with small initial data for the semi-linear problem \eqref{eq:1.0} as stated in Theorems \ref{theorem:1.3} and \ref{theorem:1.4}. Finally, Section \ref{section_6} presents newly constructed test functions to rigorously establish the blow-up results for \eqref{eq:1.0}.

\section{preliminaries} \label{section_2}
\subsection{The Heisenberg group} \label{section_2.1}
The Heisenberg group is the Lie group $\mathbf{H}_n=\mathbb{R}^{2 n+1}$ equipped with the multiplication rule $(x, y, \tau) \circ\left(x^{\prime}, y^{\prime}, \tau^{\prime}\right)=\left(x+x^{\prime}, y+y^{\prime}, \tau+\tau^{\prime}+\frac{1}{2}\left(x \cdot y^{\prime}-x^{\prime} \cdot y\right)\right)$, where $\cdot$ denotes the standard scalar product in $\mathbb{R}^n$
and dilations $\delta_\gamma(\eta)$ are given by $\gamma \odot \eta := \delta_\gamma(\eta) := \left(\gamma x, \gamma y, \gamma^2 \tau\right)$, respectively. It can be checked that the identity element is the origin 0, the inverse $\eta^{-1}=-\eta$. A system of left-invariant vector fields that span the Lie algebra $\mathfrak{h}_n$ is given by $X_j:= \partial_{x_j}-\tfrac{y_j}{2} \partial_\tau, Y_j:= \partial_{y_j}+\tfrac{x_j}{2} \partial_\tau, T:= \partial_\tau$, where $1 \leq j \leq n$. This system satisfies the commutation relations $\left[X_j, Y_k\right]=\delta_{j k} \partial_\tau \quad \text { for } 1 \leq j, k \leq n$. Therefore, $\mathfrak{h}_n$ admits the stratification $\mathfrak{h}_n=V_1 \oplus V_2$, where $V_1 := \operatorname{span}\left\{X_j, Y_j\right\}_{1 \leq j \leq n}$ and $V_2 := \operatorname{span}\left\{\partial_\tau\right\}$. Hence, $\mathbf{H}_n$ is a 2 step stratified Lie group, whose homogeneous dimension is $Q=2n+2$.

We next recall some related definitions that will be useful for the subsequent sections.
\begin{definition} \cite[Section 2.1]{Palatucci2022}. \label{definition:2.1}
A homogeneous norm on $\mathbf{H}_n$ is a continuous function (with respect to the Euclidean topology) $\mathrm{d}_{\mathrm{\circ}}: \mathbf{H}_n \rightarrow[0,+\infty)$ such that:
\begin{enumerate}[i)]
\item $\mathrm{d}_{\mathrm{\circ}}\left(\delta_\gamma(\eta)\right)=\gamma \mathrm{d}_{\mathrm{\circ}}(\eta)$, for every $\gamma>0$ and every $\eta \in \mathbf{H}_n$.
\item $\mathrm{d}_{\circ}(\eta)=0$ if and only if $\eta=0$.
\end{enumerate}
Moreover, we say that the homogeneous norm $\mathrm{d}_{\mathrm{\circ}}$ is symmetric if $\mathrm{d}_{\mathrm{\circ}}\left(\eta^{-1}\right)=\mathrm{d}_{\mathrm{\circ}}(\eta)$, for any $\eta \in \mathbf{H}_n$.
\end{definition}
Fixed an homogeneous norm $\mathrm{d}_{\mathrm{\circ}}$ on $\mathbf{H}_n$, the function $\Psi$ defined on the set of all pairs of elements of $\mathbf{H}_n$ by $\Psi(\eta, \zeta) := \mathrm{d}_{\circ}\left(\zeta^{-1} \circ \eta\right)$ is a pseudometric on $\mathbf{H}_n$. Homogenous norms are not, in general, proper norms on $\mathbf{H}_n$. However, for any homogeneous norm $\mathrm{d}_{\mathrm{\circ}}$ on $\mathbf{H}_n$ we have that there exists a constant $\Lambda>0$ such that $\Lambda^{-1}|\eta|_{\mathbf{H}_n} \leq \mathrm{d}_{\circ}(\eta) \leq \Lambda|\eta|_{\mathbf{H}_n} \text{ for all } \eta \in \mathbf{H}_n, |\eta|_{\mathbf{H}_n}:= \left(\left(|x|^2+|y|^2\right)^2+\tau^2\right)^{\frac{1}{4}}$. The function $|\cdot|_{\mathbf{H}_n}$ in the display above is the \text{Kor\'anyi} distance in the Heisenberg group, and it is actually a norm; for the proof we refer to \cite{Cygan1981} and to Example 5.1 in \cite{Balogh2017}. For any fixed $\eta_0 \in \mathbf{H}_n$ and $R>0$, we denote by $B_R\left(\eta_0\right)$ the ball with center $\eta_0$ and radius $R$, given by $B_R\left(\eta_0\right) := \left\{\eta \in \mathbf{H}_n:\left|\eta_0^{-1} \circ \eta\right|_{\mathbf{H}_n}<R\right\}$.
\begin{definition} \cite[Section 2.1]{Palatucci2022}.
Let $u \in \mathcal{C}^{\infty}\left(\mathbf{H}_n; \mathbb{R}\right)$. Then, for any $m \in \mathbb{N} \cup\{0\}$ there exists a unique polynomial $P$ being $\Phi_\lambda$-homogeneous of degree at most $m$ such that
$$
\left(X_1, \ldots, X_{2 n}, T\right)^\beta P(0)=\left(X_1, \ldots, X_{2 n}, T\right)^\beta u(0)
$$
for any multi-index $\beta=\left(\beta_1, \ldots, \beta_{2 n+1}\right)$ with $|\beta|_{\mathbf{H}_n}=\beta_1+\cdots+\beta_{2 n}+2 \beta_{2 n+1} \leq m$. We say that $P := P_m(u, 0)(\eta)$ is the MacLaurin polynomial of $\Phi_\lambda$-degree $m$ associated to $u$.
\end{definition}
In the case of the Heisenberg group one can explicitly write the MacLaurin polynomial of $\Phi_\lambda$-degree 2; we have $P_2(u, 0)\left(x_1, \ldots, x_{2 n}, t\right)=u(0)+\nabla_{\mathbf{H}_n} u(0) \cdot z+\partial_t u(0) \cdot t+\frac{1}{2}\left\langle x, D_{\mathbf{H}_n}^{2, *} u(0) \cdot x\right\rangle$, where the subgradient $\nabla_{\mathbf{H}_n} u$ is given by $\nabla_{\mathbf{H}_n} u(\eta) := \left(X_1 u(\eta), \ldots, X_{2 n} u(\eta)\right)$, and $D_{\mathbf{H}_n}^{2, *}$ is the symmetrized horizontal Hessian matrix; that is, $D_{\mathbf{H}_n}^{2, *} u(\eta) := \left(\tfrac{1}{2}\left(X_i X_j u(\eta)+X_j X_i u(\eta)\right)\right)_{i, j=1, \ldots, 2 n}$.
\begin{definition} \cite[Section 2.1]{Palatucci2022}.
Let $u \in \mathcal{C}^{\infty}\left(\mathbf{H}_n ; \mathbb{R}\right), \eta_0, \eta \in \mathbf{H}_n$, and $m \in \mathbb{N} \cup\{0\}$. Let us consider the Maclaurin polynomial $P_m\left(u\left(\eta_0 \circ \cdot\right), 0\right)$ of the function $\eta \longmapsto u\left(\eta_0 \circ \eta\right)$ The polynomial
$$
P_m\left(u, \eta_0\right)(\eta) := P_m\left(u\left(\eta_0 \circ \cdot\right), 0\right)\left(\eta_0^{-1} \circ \eta\right)
$$
is the Taylor polynomial of $\mathbf{H}_n$-degree $m$ centered at $\eta_0$ associated to $u$.
\end{definition}
One can prove the following
\begin{proposition} \label{proposition:2.1}
\cite[Corollary 20.3.5]{Bonfiglioli2007}. For every $u \in \mathcal{C}^{m+1}\left(\mathbf{H}_n ; \mathbb{R}\right), \eta_0, \eta \in$ $\mathbf{H}_n$ and $m \in \mathbb{N} \cup\{0\}$, we have that $u(\eta)=P_m\left(u, \eta_0\right)(\eta)+o\left(\left|\eta_0^{-1} \circ \eta\right|_{\mathbf{H}_n}^{m+1}\right)$.
\end{proposition}

\subsection{Group fourier transform, sub-Laplacian and Sobolev spaces on the Heisenberg group} \label{section_2.2}
In this section, we recall the properties of the group Fourier transform on $\mathbf{H}_n$ that we will use to prove Theorem \ref{theorem:1.2}. In the following, we follow the definitions and the notations of \cite[Chapters 1 and 6]{Fischer2016}. For the general definition of group Fourier transform on a locally compact group and its properties, we refer to the classical works \cite{Dixmier1977, Dixmier1981, Taylor1986, Corwin1990, Folland1995, Kirillov2004} and references therein.

Let us recall the following equivalent realization of Schr$\ddot{\text{o}}$dinger representations $\left\{\pi_\lambda\right\}_{\lambda \in \mathbb{R}^*}$ of $\mathbf{H}_n$ on $L^2\left(\mathbb{R}^n\right)$, here $\mathbb{R}^* := \mathbb{R} \backslash\{0\}$. For any $\lambda \in \mathbb{R}^*$ the mapping $\pi_\lambda$ is a strongly continuous unitary representation defined by $\pi_\lambda(x, y, \tau) \varphi(w) := \mathrm{e}^{i \lambda\left(\tau+\frac{1}{2} x \cdot y\right)} \mathrm{e}^{i \operatorname{sign}(\lambda) \sqrt{|\lambda|} y \cdot w} \varphi(w+\sqrt{|\lambda|} x)$ for any $(x, y, \tau) \in \mathbf{H}_n, \varphi \in L^2\left(\mathbb{R}^n\right)$ and $w \in \mathbb{R}^n$.
If $f \in L^1\left(\mathbf{H}_n\right)$, the group Fourier transform of $f$ is the bounded operator on $L^2\left(\mathbb{R}^n\right)$ defined as $\pi_\lambda(f) \equiv \widehat{f}(\lambda) := \int_{\mathbf{H}_n} f(\eta) \pi_\lambda(\eta)^* \mathrm{d} \eta$ for any $\lambda \in \mathbb{R}^*$, that is, $\left(\widehat{f}(\lambda) \varphi_1, \varphi_2\right)_{L^2\left(\mathbb{R}^n\right)}=\int_{\mathbf{H}_n} f(\eta)\left(\pi_\lambda(\eta)^* \varphi_1, \varphi_2\right)_{L^2\left(\mathbb{R}^n\right)} \mathrm{d} \eta$ for any $\varphi_1, \varphi_2 \in L^2\left(\mathbb{R}^n\right)$, where $\pi_\lambda(\eta)^*=\pi_\lambda(\eta)^{-1}$ denotes the adjoint operators of the unitary operator $\pi_\lambda(\eta)$. For $f \in \mathcal{S}\left(\mathbf{H}_n\right)$, the space of all Schwartz class functions on $\mathbf{H}_n$, the Fourier inversion formula takes the form $f(\eta)=\int_{\lambda \in \mathbb{R}^*} \operatorname{Tr}\left[\widehat{f}(\lambda) \pi_\lambda(\eta)\right]|\lambda|^n \mathrm{d} \lambda$, where $\mathrm{Tr}$ is the trace operator. As Schr$\ddot{\text{o}}$dinger representations are unitary, it follows immediately by the definition the inequality
\begin{equation} \label{ineq:2.1}
\|\widehat{f}(\lambda)\|_{\mathcal{L}\left(L^2\left(\mathbb{R}^n\right) \rightarrow L^2\left(\mathbb{R}^n\right)\right)} \leq\|f\|_{L^1\left(\mathbf{H}_n\right)}
\end{equation}
for any $\lambda \in \mathbb{R}^*$.
If $f \in L^2\left(\mathbf{H}_n\right)$, then, the following Plancherel formula holds
\begin{equation} \label{eq:2.2}
\|f\|_{L^2\left(\mathbf{H}_n\right)}^2=c_n \int_{\mathbb{R}^*}\|\widehat{f}(\lambda)\|_{\operatorname{HS}\left[L^2\left(\mathbb{R}^n\right)\right]}^2|\lambda|^n \mathrm{d} \lambda,
\end{equation}
where $c_n :=(2 \pi)^{-(3 n+1)}$ and $\|\widehat{f}(\lambda)\|_{\mathrm{HS}\left[L^2\left(\mathbb{R}^n\right)\right]}$ denotes the Hilbert-Schmidt norm of $\widehat{f}(\lambda)$ (cf. \cite[Proposition 6.2.7]{Fischer2016}). Therefore, for an arbitrary maximal orthonormal system $\left\{\varphi_k\right\}_{k \in \mathbb{N}}$ of $L^2\left(\mathbb{R}^n\right)$ the definition of Hilbert-Schmidt norm
\begin{align*}
\|\widehat{f}(\lambda)\|_{\mathrm{HS}\left[L^2\left(\mathbb{R}^n\right)\right]}^2 := \operatorname{Tr}\left(\left(\pi_\lambda(f)\right)^* \pi_\lambda(f)\right) =\sum_{k \in \mathbb{N}}\|\widehat{f}(\lambda) \varphi_k\|_{L^2\left(\mathbb{R}^n\right)}^2=\sum_{k, \ell \in \mathbb{N}}|\left(\widehat{f}(\lambda) \varphi_k, \varphi_{\ell}\right)_{L^2\left(\mathbb{R}^n\right)}|^2,
\end{align*}
allows us to write \eqref{eq:2.2} in more operative way, as follows:
$$
\|f\|_{L^2\left(\mathbf{H}_n\right)}^2=c_n \int_{\mathbb{R}^*} \sum_{k \in \mathbb{N}}\|\widehat{f}(\lambda) \varphi_k\|_{L^2\left(\mathbb{R}^n\right)}^2|\lambda|^n \mathrm{d} \lambda=c_n \int_{\mathbb{R}^*} \sum_{k, \ell \in \mathbb{N}}|\left(\widehat{f}(\lambda) \varphi_k, \varphi_{\ell}\right)_{L^2\left(\mathbb{R}^n\right)}|^2|\lambda|^n \mathrm{d} \lambda.
$$
The sub-Laplacian on $\mathbf{H}_n$ is defined as 
$$
\Delta_{\mathrm{H}} := \sum_{j=1}^n\left(X_j^2+Y_j^2\right)=\sum_{j=1}^n\left(\partial_{x_j}^2+\partial_{y_j}^2\right)+\tfrac{1}{4} \sum_{j=1}^n\left(x_j^2+y_j^2\right) \partial_\tau^2+\sum_{j=1}^n\left(x_j \partial_{y_j \tau}^2-y_j \partial_{x_j \tau}^2\right) .
$$
For a function $v: \mathbf{H}_n \rightarrow \mathbb{R}$, the horizontal gradient of $v$ is
$$
\nabla_{\mathrm{H}} v :=\left(X_1 v, \cdots, X_n v, Y_1 v, \cdots, Y_n v\right) \equiv \sum_{j=1}^n\left(\left(X_j v\right) X_j+\left(Y_j v\right) Y_j\right),
$$
where each fiber of the horizontal subbundle $\mathrm{H}\mathbf{H}_n=\sqcup_{\eta \in \mathbf{H}_n} \mathrm{H}_\eta \mathbf{H}_n$ can be endowed with a scalar product $\langle\cdot, \cdot\rangle_\eta$ in such a way that $X_1(\eta), \cdots, X_n(\eta), Y_1(\eta), \cdots, Y_n(\eta)$ are orthonormal in $\left(\mathrm{H}_\eta \mathbf{H}_n,\langle\cdot, \cdot\rangle_\eta\right)$ for any $\eta \in \mathbf{H}_n$. Therefore, if $X=\sum_{j=1}^n\left(\alpha_j X_j+\beta_j Y_j\right)+\gamma \partial_\tau$ is a vector field on $\mathbf{H}_n$ with $\alpha_j, \beta_j, \gamma \in \mathcal{C}^1\left(\mathbf{H}_n\right)$ for any $j=1, \cdots, n$, the divergence of $X$ is the function $\operatorname{div} X := \sum_{j=1}^n\left(X_j \alpha_j+Y_j \beta_j\right)+\partial_\tau \gamma$. In particular, the sub-Laplacian may be expressed also as $\Delta_{\mathrm{H}} v=\operatorname{div}\left(\nabla_{\mathrm{H}} v\right)$. For a function $v \in L^2\left(\mathbf{H}_n\right)$ we say that $X_j v, Y_j v \in L_{\mathrm{loc}}^1\left(\mathbf{H}_n\right)$ exist in the sense of distributions, if the integral relations 
$$
\int_{\mathbf{H}_n}\left(X_j v\right)(\eta) \phi(\eta) \mathrm{d} \eta =\int_{\mathbf{H}_n} v(\eta)\left(X_j^* \phi\right)(\eta) \mathrm{d} \eta \quad \text{ and } \quad \int_{\mathbf{H}_n}\left(Y_j v\right)(\eta) \phi(\eta) \mathrm{d} \eta =\int_{\mathbf{H}_n} v(\eta)\left(Y_j^* \phi\right)(\eta) \mathrm{d} \eta
$$
are fulfilled for any $\phi \in \mathcal{C}_0^{\infty}\left(\mathbf{H}_n\right)$, where $X_j^*=-X_j$ and $Y_j^*=-Y_j$ denote the formal adjoint operators of $X_j$ and $Y_j$, respectively.

On the other hand, the Plancherel measure $\mu$ on $\widehat{\mathbf{H}}_n$ is supported on the equivalence classes of $\pi_\lambda, \lambda \in \mathbb{R}^*$ and $\mathrm{d} \mu\left(\pi_\lambda\right)=c_n|\lambda|^n \mathrm{d} \lambda$.
A further property related to the group Fourier transform, that we will apply extensively, is the action of the infinitesimal representation of $\pi_\lambda$ on the generators of the first layer of the Lie algebra $\mathfrak{h}_n$, namely,
\begin{align}
\label{eq:2.4} \mathrm{d} \pi_\lambda\left(X_j\right) =&\sqrt{|\lambda|} \partial_{w_j} \text { for } j=1, \cdots, n, \\
\label{eq:2.5} \mathrm{d} \pi_\lambda\left(Y_j\right) =&i \operatorname{sign}(\lambda) \sqrt{|\lambda|} w_j \text { for } j=1, \cdots, n.
\end{align}

Since the action of $\mathrm{d} \pi_\lambda$ can be extended to the universal enveloping algebra of $\mathfrak{h}_n$, combining \eqref{eq:2.4} and \eqref{eq:2.5}, it follows
\begin{equation} \label{eq:2.6}
\mathrm{d} \pi_\lambda\left(\Delta_{\mathrm{H}}\right)=|\lambda| \sum_{j=1}^n\left(\partial_{w_j}^2-w_j^2\right)=-|\lambda| \mathrm{H}_w,
\end{equation}
where $\mathrm{H}_w :=-\Delta_w+|w|^2$ is the harmonic oscillator on $\mathbb{R}^n$. Thus, the operator-valued symbol $\sigma_{\Delta_{\mathrm{H}}}(\lambda)$ of $\Delta_{\mathrm{H}}$ acting on $L^2\left(\mathbb{R}^n\right)$ takes the form $\sigma_{\Delta_{\mathrm{H}}}(\lambda)=-|\lambda| \mathrm{H}_w$. Furthermore, for $s \in \mathbb{R}$, using the functional calculus, the symbol of $(-\Delta_{\mathrm{H}})^s$ is $|\lambda|^s H^s$, where the notion of $H^s$ is defined in \eqref{eq:2.11}.

The Sobolev spaces $H^s, s \in \mathbb{R}$, associated to the sublaplacian $\Delta_{\mathrm{H}}$, are defined as
$$
H^s\left(\mathbf{H}_n\right) := \left\{f \in \mathcal{D}^{\prime}\left(\mathbf{H}_n\right):(I-\Delta_{\mathrm{H}})^{\frac{s}{2}} f \in L^2\left(\mathbf{H}_n\right)\right\},
$$
with the norm $\|f\|_{H^s\left(\mathbf{H}_n\right)} := \left\|(I-\Delta_{\mathrm{H}})^{\frac{s}{2}} f\right\|_{L^2\left(\mathbf{H}_n\right)}$. Similarly, we denote by $\dot{H}^s\left(\mathbf{H}_n\right)$, the homogeneous Sobolev defined as the space of all $f \in \mathcal{D}^{\prime}\left(\mathbf{H}_n\right)$ such that $(-\Delta_{\mathrm{H}})^{\frac{s}{2}} f \in L^2\left(\mathbf{H}_n\right)$. More generally, we define $\dot{H}_s^p\left(\mathbf{H}_n\right)$ as the homogeneous Sobolev defined as the space of all $f \in \mathcal{D}^{\prime}\left(\mathbf{H}_n\right)$ such that $(-\Delta_{\mathrm{H}})^{\frac{s}{2}} f \in L^p\left(\mathbf{H}_n\right)$ for $s<\frac{Q}{p}$.
\subsection{Hermite functions and their properties} \label{section_2.3}
In the last section, we recalled some properties of Schr$\ddot{\text{o}}$dinger representations. Because of \eqref{eq:2.6} it will be helpful for what follows to recall the definition of the Hermite functions and to prove some elementary properties of them. These functions constitute a complete system of eigenfunctions for the essentially self adjoint operator $\mathrm{H}_w$. Let us begin with the definition of $m$-th Hermite polynomial $H_m(x) := \mathrm{e}^{x^2}\left(\tfrac{\mathrm{d}}{\mathrm{d} x}\right)^m \mathrm{e}^{-x^2}, x \in \mathbb{R}, m \in \mathbb{N}$. It is well-known that Hermite polynomials satisfy the following recurrence properties
\begin{align}
\label{eq:2.7} H_{m+1}(x) =&2 x H_m(x)-2 m H_{m-1}(x), \\
\label{eq:2.8} H_m^{\prime}(x) =&2 m H_{m-1}(x)
\end{align}
for any $m \in \mathbb{N}$ and $x \in \mathbb{R}$. Furthermore, $\int_{\mathbb{R}} H_m(x) H_{\ell}(x) \mathrm{e}^{-x^2} \mathrm{d} x=\sqrt{\pi} 2^m m!\delta_{m, \ell} \quad \text { for any } m, \ell \in \mathbb{N}$, where $\delta_{m, \ell}$ denotes the Kronecker delta. The $m$-th Hermite function is $\psi_m(x) := a_m \mathrm{e}^{-\frac{x^2}{2}} H_m(x)$ for any $m \in \mathbb{N}$ and $x \in \mathbb{R}$, where $a_m :=\left(\sqrt{\pi} 2^m m!\right)^{-1 / 2}$. By using \eqref{eq:2.7} and \eqref{eq:2.8} it follows easily that $\psi_m$ is an eigenfunction for the 1-d harmonic oscillator $H=-\frac{\mathrm{d}^2}{\mathrm{d} x^2}+x^2$ relative to the eigenvalue $2 m+1$ as
\begin{equation} \label{eq:2.9}
-\psi_m^{\prime \prime}(x)+x^2 \psi_m(x)=(2 m+1) \psi_m(x) .
\end{equation}
The multidimensional version of Hermite functions is given by
$$
e_k(w) := \prod_{j=1}^n \psi_{k_j}\left(w_j\right)=\prod_{j=1}^n a_{k_j} H_{k_j}\left(w_j\right) \mathrm{e}^\frac{-w_j^2}{2}
$$
for any multi-index $k=\left(k_j\right)_{1 \leq j \leq n} \in \mathbb{N}^n$ and any $w=\left(w_j\right)_{1 \leq j \leq n} \in \mathbb{R}^n$. As the elements of $\left\{e_k\right\}_{k \in \mathbb{N}^n}$ are functions with separate variables, by \eqref{eq:2.9} we get that
$$
\mathrm{H}_w e_k(w)=\left(-\Delta_w+|w|^2\right) e_k(w)=\left(\sum_{j=1}^n\left(2 k_j+1\right)\right) e_k(w)=\underbrace{(2|k|+n)}_{:= \mu_k} e_k(w)
$$
for any $k \in \mathbb{N}^n$ and any $w \in \mathbb{R}^n$. Given $f \in L^2\left(\mathbb{R}^n\right)$, we have the Hermite expansion
$$
f=\sum_{k \in \mathbb{N}^n}\left(f, e_k\right) e_k=\sum_{m=0}^{\infty} \sum_{|k|=m}\left(f, e_k\right) e_k=\sum_{m=0}^{\infty} P_m f,
$$
where $P_m$ denotes the orthogonal projection of $L^2\left(\mathbb{R}^n\right)$ onto the eigenspace spanned by $\left\{e_k:|k|=m\right\}$. Then, the spectral decomposition of $H$ on $\mathbb{R}^n$ is given by $H f=\sum_{m=0}^{\infty}(2 m+n) P_m f$. Since 0 is not in the spectrum of $H$, for any $s \in \mathbb{R}$, we can define the fractional powers $H^s$ by means of the spectral theorem, namely
\begin{equation} \label{eq:2.11}
H^s f=\sum_{m=0}^{\infty}(2 m+n)^s P_m f.
\end{equation}
Let us recall the following result:
\begin{proposition} \cite[Theorem 2.2.3]{Nicola2010}. The system $\left\{e_k\right\}_{k \in \mathbb{N}^n}$ is an orthonormal basis of $L^2\left(\mathbb{R}^n\right)$.
\end{proposition}
\subsection{Some auxiliary inequalities} \label{section_2.4}
Next, we recall the following important inequalities:
\begin{lemma} \label{Hardy_Littlewood}
\cite[Hardy-Littlewood-Sobolev inequality]{Dasgupta2024}. Let $\mathbf{H}_n$ be the Heisenberg group with the homogeneous dimension $Q=2n+2$. Let $a \geq 0$ and $1<p \leq q<\infty$ be such that $\frac{a}{Q}=\frac{1}{p}-\frac{1}{q}$. Then we have the following inequality $\|f\|_{\dot{H}^{q}_{-a}\left(\mathbf{H}_n\right)} \lesssim\|f\|_{L^p\left(\mathbf{H}_n\right)}$.
\end{lemma}
\begin{lemma} \label{Gagliardo}
\cite[Gagliardo-Nirenberg inequality]{Ruzhansky2018}. Let $Q$ be the homogeneous dimension on the Heisenberg group $\mathbf{H}_n$. Assume that $s \in(0,1], 1<r<\frac{Q}{s}, \text { and } 2 \leq q \leq \frac{r Q}{Q-s r}$. Then, we have the following Gagliardo-Nirenberg type inequality:
$$
\|u\|_{L^q\left(\mathbf{H}_n\right)} \lesssim\|u\|_{\dot{H}^{r}_s\left(\mathbf{H}_n\right)}^\theta\|u\|_{L^2\left(\mathbf{H}_n\right)}^{1-\theta} \quad \text{ for } \theta=\tfrac{\frac{1}{2}-\frac{1}{q}}{\frac{s}{Q}+\frac{1}{2}-\frac{1}{r}} \in [0,1], \text{ provided } \tfrac{s}{Q}+\tfrac{1}{2} \neq \tfrac{1}{r}
$$
\end{lemma}

\section{decay character on the Heisenberg group} \label{section_3}
\subsection{Definition of the decay character on the Heisenberg group} \label{section_3.2}
Next, we introduce the following new definitions of decay indicators and decay character on the Heisenberg group. Firstly, we introduce the following notions of lower and upper decay indicators.
\begin{definition} \label{definition_3.7}
(The lower and upper decay indicators). Let $u_0 \in L^2\left(\mathbf{H}_n\right)$ and $\beta\geq 0$. We define the lower and upper decay indicator by
\begin{align*}
P_{r_\mathrm{H}}^\beta\left(u_0\right)_{-}=&\liminf _{\rho_\mathrm{H} \rightarrow 0^{+}} \rho_\mathrm{H}^{-2 r_\mathrm{H}-Q} c_n \sum_{k, \ell \in \mathbb{N}^n} \int_{B^*\left(\mu_k^{-1}\rho_\mathrm{H}^2\right)} \mu_k^\beta |\lambda|^{\beta+n} \left|\widehat{u}_0(\lambda)_{k, \ell}\right|^2 \mathrm{d} \lambda, \\
P_{r_\mathrm{H}}^\beta\left(u_0\right)_{+}=&\limsup _{\rho_\mathrm{H} \rightarrow 0^{+}} \rho_\mathrm{H}^{-2r_\mathrm{H}-Q} c_n \sum_{k, \ell \in \mathbb{N}^n} \int_{B^*\left(\mu_k^{-1}\rho_\mathrm{H}^2\right)} \mu_k^\beta |\lambda|^{\beta+n} \left|\widehat{u}_0(\lambda)_{k, \ell}\right|^2 \mathrm{d} \lambda,
\end{align*}
where $c_n := (2 \pi)^{-(3 n+1)}$ and $B^*\left(\mu_k^{-1}\rho_\mathrm{H}^2\right)=\left\{\lambda \in \mathbb{R}^*:|\lambda| \leq \mu_k^{-1}\rho_\mathrm{H}^2(t)\right\}$ is the ball centered at origin with radius $\mu_k^{-1}\rho_\mathrm{H}^2$, with $\mu_k := 2|k|+n$.
\end{definition}
This definition is interesting only for $r_\mathrm{H}\in \left(-\frac{Q}{2}+\beta, \infty\right)$  as the decay indicator is
always zero when $r_\mathrm{H} \leq -\frac{Q}{2}+\beta$. We also note that the lower and upper decay indicators of $u_0 \in L^2\left(\mathbf{H}_n\right)$ always exist in $\overline{\mathbb{R}}$.

Next, we introduce the following new definitions of lower and upper decay indicators of negative order on the Heisenberg group.
\begin{definition}
(The lower and upper decay indicators with negative order). Let $u_0 \in L^2\left(\mathbf{H}_n\right)$ and $\beta\geq 0$. We define the lower and upper decay indicator with negative order by
\begin{align*}
P_{r_\mathrm{H}}^{-\beta}\left(u_0\right)_{-}=&\liminf _{\rho_\mathrm{H} \rightarrow 0^{+}} \rho_\mathrm{H}^{-2 r_\mathrm{H}-Q} c_n \sum_{k, \ell \in \mathbb{N}^n} \int_{B^*\left(\mu_k^{-1}\rho_\mathrm{H}^2\right)} \mu_k^{-\beta} |\lambda|^{-\beta+n} \left|\widehat{u}_0(\lambda)_{k, \ell}\right|^2 \mathrm{d} \lambda, \\
P_{r_\mathrm{H}}^{-\beta}\left(u_0\right)_{+}=&\limsup _{\rho_\mathrm{H} \rightarrow 0^{+}} \rho_\mathrm{H}^{-2r_\mathrm{H}-Q} c_n \sum_{k, \ell \in \mathbb{N}^n} \int_{B^*\left(\mu_k^{-1}\rho_\mathrm{H}^2\right)} \mu_k^{-\beta} |\lambda|^{-\beta+n} \left|\widehat{u}_0(\lambda)_{k, \ell}\right|^2 \mathrm{d} \lambda,
\end{align*}
where $c_n := (2 \pi)^{-(3 n+1)}$ and $B^*\left(\mu_k^{-1}\rho_\mathrm{H}^2\right)=\left\{\lambda \in \mathbb{R}^*:|\lambda| \leq \mu_k^{-1}\rho_\mathrm{H}^2(t)\right\}$ is the ball centered at origin with radius $\mu_k^{-1}\rho_\mathrm{H}^2$, with $\mu_k := 2|k|+n$.
\end{definition}
Next, we define the lower and upper decay character on the Heisenberg group as follows.
\begin{definition} (The lower and upper decay characters).
The lower and upper decay characters of $u_0 \in L^2\left(\mathbf{H}_n\right)$ are respectively defined by
\begin{align*}
& r_{\mathrm{H}, \beta}\left(u_0\right)_{-}=\inf \left\{r_\mathrm{H} \in \mathbb{R}: P_{r_\mathrm{H}}^\beta\left(u_0\right)_{-}>0\right\}, \\
& r_{\mathrm{H}, \beta}\left(u_0\right)_{+}=\sup \left\{r_\mathrm{H} \in \mathbb{R}: P_{r_\mathrm{H}}^\beta\left(u_0\right)_{+}<\infty\right\}.
\end{align*}
\end{definition}
\begin{remark}
\rm{
For any $u_0 \in L^2\left(\mathbf{H}_n\right)$, the upper and lower decay character of $u_0$ are always well defined and admit the inequality $-\frac{Q}{2}+\beta \leq r_{\mathrm{H}, \beta}^*\left(u_0\right)_{+} \leq r_{\mathrm{H}, \beta}^*\left(u_0\right)_{-} \leq+\infty$.
}
\end{remark}
Finally, the decay character on the Heisenberg group is defined by the following manner.
\begin{definition} \label{definition_3.9}
(Decay character). If $u_0 \in L^2\left(\mathbf{H}_n\right)$ is such that there exists $r_{\mathrm{H}, \beta}^* \in \left(-\frac{Q}{2}+\beta, \infty\right)$ such that
$$
r_{\mathrm{H}, \beta}^*=\min \left\{r_\mathrm{H} \in \mathbb{R}: P_{r_\mathrm{H}}^\beta\left(u_0\right)_{-}>0\right\}=\max \left\{r_\mathrm{H} \in \mathbb{R}: P_{r_\mathrm{H}}^\beta\left(u_0\right)_{+}<\infty\right\},
$$
then we say $r_{\mathrm{H}, \beta}^*=r_{\mathrm{H}, \beta}^*\left(u_0\right)$ is the decay character of $u_0$. We also define the decay character of $u_0$ in the two limit situations as follows:
$$
\begin{cases}
r_{\mathrm{H}, \beta}^*\left(u_0\right)=+\infty & \text{ if } r_{\mathrm{H}, \beta}\left(u_0\right)_{+}=r_{\mathrm{H}, \beta}\left(u_0\right)_{-}=+\infty, \\
r_{\mathrm{H}, \beta}^*\left(u_0\right)=-\frac{Q}{2}+\beta & \text{ if } r_{\mathrm{H}, \beta}\left(u_0\right)_{+}=r_{\mathrm{H}, \beta}\left(u_0\right)_{-}=-\frac{Q}{2}+\beta.
\end{cases}
$$
For $\beta=0$, we denote $P_{r_\mathrm{H}}\left(u_0\right)_{-}=P_{r_\mathrm{H}}^0\left(u_0\right)_{-}, P_{r_\mathrm{H}}\left(u_0\right)_{+}=P_{r_\mathrm{H}}^0\left(u_0\right)_{+}, r_{\mathrm{H}, 0}=r_\mathrm{H}$ and $r_{\mathrm{H}, 0}^*=r_\mathrm{H}^*$.
\end{definition}
\subsection{Decay characterization of solutions to the linear fractional diffusion equation on the Heisenberg group} \label{section_3.3}
\begin{theorem} \label{theorem:3.7}
Consider the Cauchy problem
\begin{equation} \label{heat_equation_Heisenberg}
\begin{cases}
\partial_t u(t, \eta)+\left(-\Delta_{\mathrm{H}}\right)^\alpha u(t, \eta) =0, & (t, \eta) \in[s, \infty) \times \mathbf{H}_n, \quad s \geq 0, \quad \alpha > 0, \quad n \in \mathbb{N}^*, \\ u(s, \eta) =u_0(\eta) \in L^2\left(\mathbf{H}_n\right), & \eta \in \mathbf{H}_n, \quad s \geq 0, \quad n \in \mathbb{N}^*,
\end{cases}
\end{equation}
Then
\begin{enumerate}[1)]
\item If $P_{r_\mathrm{H}}\left(u_0\right)_{-}>0$, then the solution $u(t, \eta)$ to \eqref{heat_equation_Heisenberg} satisfies
$$
\|u(t, \eta)\|_{L^2\left(\mathbf{H}_n\right)}^2 \gtrsim \left(1+t-s\right)^{-\frac{2 r_\mathrm{H} +Q}{2\alpha}} .
$$
\item If $P_{r_\mathrm{H}}\left(u_0\right)_{+}<\infty$, then the solution $u(t, \eta)$ to \eqref{heat_equation_Heisenberg} satisfies
$$
\|u(t, \eta)\|_{L^2\left(\mathbf{H}_n\right)}^2 \lesssim \left(1+t-s\right)^{-\frac{2 r_\mathrm{H} +Q}{2\alpha}}.
$$
\end{enumerate}
\end{theorem}
\begin{proof}[Proof of Theorem \ref{theorem:3.7}]
Applying the group Fourier transform to equation \eqref{heat_equation_Heisenberg} yields a Cauchy problem for a parameter-dependent functional differential equation involving $\widehat{u}(t, \lambda)$
$$
\begin{cases}
\partial_t \widehat{u}(t, \lambda)+\sigma_{\left(-\Delta_{\mathrm{H}}\right)^\alpha}(\lambda)\partial_t \widehat{u}(t, \lambda)=0, & (t, \lambda) \in [s, \infty) \times \mathbb{R}^*, \quad s \geq 0, \quad \alpha > 0, \\ \widehat{u}(s, \lambda)=\widehat{u}_0(\lambda), & \lambda \in \mathbb{R}^*, \quad s \geq 0,
\end{cases}
$$
where $\sigma_{\left(-\Delta_{\mathrm{H}}\right)^\alpha}(\lambda)=\left(\mu_k|\lambda|\right)^\alpha$ is the symbol of the fractional sub-Laplacian $\left(-\Delta_{\mathrm{H}}\right)^\alpha$ on the Heisenberg group. We now introduce the following notation:
$$
\widehat{u}(t, \lambda)_{k, \ell} := \left(\widehat{u}(t, \lambda) e_k, e_{\ell}\right)_{L^2\left(\mathbb{R}^n\right)} \text{ for any } k, \ell \in \mathbb{N}^n,
$$
where $\left\{e_k\right\}_{k \in \mathbb{N}^n}$ denotes the system of Hermite functions. Since $\mathrm{H}_w^\alpha e_k=\mu_k^\alpha e_k$, the term $\widehat{u}(t, \lambda)_{k, \ell}$ satisfies an ordinary differential equation in $t$ that depends on the parameters $\lambda \in \mathbb{R}^*$ and $k, \ell \in \mathbb{N}^n$
$$
\begin{cases}
\partial_t \widehat{u}(t, \lambda)_{k, \ell}+\left(\mu_k|\lambda|\right)^\alpha \widehat{u}(t, \lambda)_{k, \ell}=0, & (t, \lambda) \in [s, \infty) \times \mathbb{R}^*, \quad s \geq 0, \quad \alpha > 0, \\
\widehat{u}(s, \lambda)_{k, \ell}=\widehat{u}_0(\lambda)_{k, \ell}, & \lambda \in \mathbb{R}^*, \quad s \geq 0.
\end{cases}
$$
Solving this equation gives the explicit solution
\begin{equation} \label{eq:u_t_u_0}
\widehat{u}(t, \lambda)_{k, \ell}=\mathrm{e}^{-\left(\mu_k|\lambda|\right)^\alpha(t-s)} \widehat{u}_0(\lambda)_{k, \ell}.
\end{equation}
\begin{enumerate}[1)]
\item Let $B^*\left(\mu_k^{-1}\rho_\mathrm{H}^2(t)\right)=\left\{\lambda \in \mathbb{R}^*:|\lambda| \leq \mu_k^{-1}\rho_\mathrm{H}^2(t)\right\}$, for some non-increasing, continuous radius $\mu_k^{-1}\rho_\mathrm{H}^2=\mu_k^{-1}\rho_\mathrm{H}^2(t)$ to be determined later. Assume that $P_{r_\mathrm{H}}\left(u_0\right)_{-}>0$. Then, there exists a constant ${\rho_\mathrm{H}}_0>0$ such that for all $0<\rho_\mathrm{H} \leq {\rho_\mathrm{H}}_0$, the following inequality holds:
\begin{equation} \label{ineq:p_u0_1}
\rho_\mathrm{H}^{2 r_\mathrm{H}+Q}(t) \lesssim c_n \sum_{k, \ell \in \mathbb{N}^n} \int_{B^*\left(\mu_k^{-1}\rho_\mathrm{H}^2(t)\right)} |\lambda|^n \left|\widehat{u}_0(\lambda)_{k, \ell}\right|^2 \mathrm{d} \lambda.
\end{equation}
By applying \eqref{eq:u_t_u_0} and \eqref{ineq:p_u0_1}, we derive the estimate
\begin{align} \label{ineq:lower_bound}
\|u(t, \eta)\|_{L^2\left(\mathbf{H}_n\right)}^2 & \gtrsim c_n \sum_{k, \ell \in \mathbb{N}^n} \int_{B^*\left(\mu_k^{-1}\rho_\mathrm{H}^2(t)\right)} |\lambda|^n \left|\mathrm{e}^{-\left(\mu_k|\lambda|\right)^\alpha(t-s)} \widehat{u}_0(\lambda)_{k, \ell}\right|^2 \mathrm{d} \lambda \notag\\
& \gtrsim c_n \mathrm{e}^{-2\rho_\mathrm{H}^{2\alpha}(t)(t-s)} \sum_{k, \ell \in \mathbb{N}^n} \int_{B^*\left(\mu_k^{-1}\rho_\mathrm{H}^2(t)\right)} |\lambda|^n \left|\widehat{u}_0(\lambda)_{k, \ell}\right|^2 \mathrm{d} \lambda \notag\\
& \gtrsim c_n \mathrm{e}^{-2\rho_\mathrm{H}^{2\alpha}(t)(t-s)} \rho_\mathrm{H}^{2 r_\mathrm{H}+Q}(t).
\end{align}
By choosing $\rho_\mathrm{H}(t)={\rho_\mathrm{H}}_0 \left(1+t-s\right)^{-\frac{1}{2\alpha}}$, it follows that $c_n \mathrm{e}^{-2\rho_\mathrm{H}^{2\alpha}(t)(t-s)} \geq C>0$. Consequently, we establish that
$$
\|u(t, \eta)\|_{L^2\left(\mathbf{H}_n\right)}^2 \gtrsim \left(1+t-s\right)^{-\frac{2 r_\mathrm{H} +Q}{2\alpha}},
$$
which confirms the desired lower bound.
\item The upper bound is established using the Fourier splitting method. From \eqref{heat_equation_Heisenberg}, we derive
\begin{align}
\tfrac{\mathrm{d}}{\mathrm{d} t}\|u(t, \eta)\|_{L^2\left(\mathbf{H}_n\right)}^2 \lesssim & - c_n \sum_{k, \ell \in \mathbb{N}^n} \int_{\mathbb{R}^*} \mu_k^\alpha |\lambda|^{n+\alpha} \left|\widehat{u}(\lambda)_{k, \ell}\right|^2 \mathrm{d} \lambda \notag\\
\label{ineq:upper_bound}\lesssim & - \rho_\mathrm{H}^{2\alpha}(t) c_n \sum_{k, \ell \in \mathbb{N}^n} \int_{{B^*}^{C}\left(\mu_k^{-1}\rho_\mathrm{H}^2(t)\right)} |\lambda|^n \left|\widehat{u}(\lambda)_{k, \ell}\right|^2 \mathrm{d} \lambda,
\end{align}
where $B^*\left(\mu_k^{-1}\rho_\mathrm{H}^2(t)\right)$ is defined as before. Consequently, we obtain
\begin{equation} \label{ineq:bp}
\tfrac{\mathrm{d}}{\mathrm{d} t}\|u(t, \eta)\|_{L^2\left(\mathbf{H}_n\right)}^2+ \rho_\mathrm{H}^{2\alpha}(t) \|u(t, \eta)\|_{L^2\left(\mathbf{H}_n\right)}^2 \lesssim \rho_\mathrm{H}^{2\alpha}(t) c_n \sum_{k, \ell \in \mathbb{N}^n} \int_{B^*\left(\mu_k^{-1}\rho_\mathrm{H}^2(t)\right)} |\lambda|^n \left|\widehat{u}(\lambda)_{k, \ell}\right|^2 \mathrm{d} \lambda.
\end{equation}
Given that $P_{r_\mathrm{H}}\left(u_0\right)_{+}<\infty$, there exists ${\rho_\mathrm{H}}_0>0$ such that for all $0<\rho_\mathrm{H} \leq {\rho_\mathrm{H}}_0$, we have
\begin{equation} \label{ineq:p_u0} c_n \sum_{k, \ell \in \mathbb{N}^n} \int_{B^*\left(\mu_k^{-1}\rho_\mathrm{H}^2(t)\right)} |\lambda|^n \left|\widehat{u}_0(\lambda)_{k, \ell}\right|^2 \mathrm{d} \lambda \lesssim \rho_\mathrm{H}^{2 r_\mathrm{H}+Q}(t).
\end{equation}
Additionally, by \eqref{eq:u_t_u_0}, it holds that
\begin{equation} \label{ineq:u_u0}
c_n \sum_{k, \ell \in \mathbb{N}^n} \int_{B^*\left(\mu_k^{-1}\rho_\mathrm{H}^2(t)\right)} |\lambda|^n \left|\widehat{u}(\lambda)_{k, \ell}\right|^2 \mathrm{d} \lambda \lesssim  c_n \sum_{k, \ell \in \mathbb{N}^n} \int_{B^*\left(\mu_k^{-1}\rho_\mathrm{H}^2(t)\right)} |\lambda|^n \left|\widehat{u}_0(\lambda)_{k, \ell}\right|^2 \mathrm{d} \lambda.
\end{equation}
Then from \eqref{ineq:bp}, \eqref{ineq:p_u0} and \eqref{ineq:u_u0}, we obtain
$$ 
\tfrac{\mathrm{d}}{\mathrm{d} t}\|u(t, \eta)\|_{L^2\left(\mathbf{H}_n\right)}^2+ \rho_\mathrm{H}^{2\alpha}(t) \|u(t, \eta)\|_{L^2\left(\mathbf{H}_n\right)}^2 \lesssim \rho_\mathrm{H}^{2 r_\mathrm{H}+Q+2\alpha}(t).
$$
By setting $\rho_\mathrm{H}(t)=m^\frac{1}{2\alpha} \left(1+t-s\right)^{-\frac{1}{2\alpha}}$, with $m > \frac{2 r_\mathrm{H}+Q}{2\alpha}$ and multiplying both sides by the integrating factor $h(t)=\left(1+t-s\right)^m$, we deduce
$$
\tfrac{\mathrm{d}}{\mathrm{d} t}\left(\left(1+t-s\right)^m \|u(t, \eta)\|_{L^2\left(\mathbf{H}_n\right)}^2\right) \lesssim m^{\frac{2 r_\mathrm{H}+Q+2\alpha}{2\alpha}} \left(1+t-s\right)^{-\frac{2 r_\mathrm{H}+Q+2\alpha}{2\alpha}+m}.
$$
Integrating this inequality from $s$ to $t$ leads to
$$
\|u(t, \eta)\|_{L^2\left(\mathbf{H}_n\right)}^2 \lesssim \left(1+t-s\right)^{-\frac{2 r_\mathrm{H} +Q}{2\alpha}},
$$
which completes the proof of the upper bound.
\end{enumerate}
\end{proof}
The following theorem clearly demonstrates the significance of the concept of decay character.
\begin{theorem} \label{theorem:3.8}
Let $u(t, \eta)$ be a solution to \eqref{heat_equation_Heisenberg}, with the initial datum $u(s, \eta)=u_0(\eta)\in L^2\left(\mathbf{H}_n\right)$ have decay character $r_\mathrm{H}^*\left(u_0\right)=r_\mathrm{H}^*$, where $s\geq 0$. If $-\frac{Q}{2}<r_\mathrm{H}^*<\infty$, then
$$
\left(1+t-s\right)^{-\frac{2 r_\mathrm{H}^* +Q}{2\alpha}} \lesssim \|u(t, \eta)\|_{L^2\left(\mathbf{H}_n\right)}^2 \lesssim \left(1+t-s\right)^{-\frac{2 r_\mathrm{H}^* +Q}{2\alpha}} \quad \text{ with } \alpha >0.
$$
\end{theorem}
\begin{proof}[Proof of Theorem \ref{theorem:3.8}]
This theorem is a corollary of Theorem \ref{theorem:3.7}.
\end{proof}
Next, we obtain a result regarding the relation between $r_{\mathrm{H}, -\beta}\left(u_0\right)$ and $r_\mathrm{H}\left(u_0\right)$ for a function $u_0 \in L^2\left(\mathbf{H}_n\right)$.
\begin{theorem} \label{theorem:3.11}
Let $u_0 \in L^2\left(\mathbf{H}_n\right)$ and $r_\mathrm{H}\left(u_0\right)-\beta>-\frac{Q}{2}$ such that $P_{r_\mathrm{H}}\left(u_0\right)_{+}<\infty$ for some $\beta \geq 0$. Then $v_0 \in L^2\left(\mathbf{H}_n\right)$ and $r_{\mathrm{H}, -\beta}\left(u_0\right)=r_\mathrm{H}\left(v_0\right)=r_\mathrm{H}\left(u_0\right)-\beta$, where $\widehat{v}_0(\lambda)_{k, \ell}=\left(\mu_k|\lambda|\right)^{-\frac{\beta}{2}} \widehat{u}_0(\lambda)_{k, \ell}$.
\end{theorem}
\begin{proof}[Proof of Theorem \ref{theorem:3.11}]
Let $r_\mathrm{H}=r_\mathrm{H}\left(u_0\right)=q+\beta$, for some $q>-\frac{Q}{2}$. We can establish that 
\begin{align*}
&\rho_\mathrm{H}^{-2 q-Q} c_n \sum_{k, \ell \in \mathbb{N}^n} \int_{B^*\left(\mu_k^{-1}\rho_\mathrm{H}^2\right)} \mu_k^{-\beta} |\lambda|^{-\beta+n} \left|\widehat{u}_0(\lambda)_{k, \ell}\right|^2 \mathrm{d} \lambda \\
\geq & \rho_\mathrm{H}^{-2 r_\mathrm{H}-Q} c_n \sum_{k, \ell \in \mathbb{N}^n} \int_{B^*\left(\mu_k^{-1}\rho_\mathrm{H}^2\right)} |\lambda|^n \left|\widehat{u}_0(\lambda)_{k, \ell}\right|^2 \mathrm{d} \lambda.
\end{align*}
Taking the limit as $\rho_\mathrm{H} \rightarrow 0^{+}$, we obtain $0 < P_{r_\mathrm{H}}\left(u_0\right)_{+} \leq P_q^{-\beta}\left(u_0\right)_{-}\leq P_q^{-\beta}\left(u_0\right)_{+}\leq\infty$. We will show that $P_q^{-\beta}\left(u_0\right)_{+}<\infty$. Indeed, define $A_p\left(\mu_k^{-1}\rho_\mathrm{H}^2\right)=\left\{\lambda: \tfrac{\mu_k^{-1}\rho_\mathrm{H}^2}{2^{p+1}} \leq |\lambda| \leq \tfrac{\mu_k^{-1}\rho_\mathrm{H}^2}{2^p}\right\}$. Observe that $A_p\left(\mu_k^{-1}\rho_\mathrm{H}^2\right) \subset B^*\left(\frac{\mu_k^{-1}\rho_\mathrm{H}^2}{2^p}\right)$. By selecting $\rho_\mathrm{H} \leq {\rho_\mathrm{H}}_0$, we have $\frac{\rho_\mathrm{H}}{2^p} \leq \rho_\mathrm{H} \leq {\rho_\mathrm{H}}_0$, which implies
\begin{align} \label{ineq:A_rho}
& \left(\tfrac{\rho_\mathrm{H}}{2^p}\right)^{-2 r_\mathrm{H}-Q} c_n \sum_{k, \ell \in \mathbb{N}^n} \int_{A_p\left(\mu_k^{-1}\rho_\mathrm{H}^2\right)} \mu_k^{-\beta} |\lambda|^{-\beta+n} \left|\widehat{u}_0(\lambda)_{k, \ell}\right|^2 \mathrm{d} \lambda \notag\\
\leq & \left(\tfrac{\rho_\mathrm{H}}{2^p}\right)^{-2 r_\mathrm{H}-Q}\left(\tfrac{\rho_\mathrm{H}}{2^{p+1}}\right)^{-2\beta} c_n \sum_{k, \ell \in \mathbb{N}^n} \int_{A_p\left(\mu_k^{-1}\rho_\mathrm{H}^2\right)} |\lambda|^n \left|\widehat{u}_0(\lambda)_{k, \ell}\right|^2 \mathrm{d} \lambda \notag\\
\leq & \left(\tfrac{\rho_\mathrm{H}}{2^{p+1}}\right)^{-2\beta} \left(\tfrac{\rho_\mathrm{H}}{2^p}\right)^{-2 r_\mathrm{H}-Q} c_n \sum_{k, \ell \in \mathbb{N}^n} \int_{B^*\left(\frac{\mu_k^{-1}\rho_\mathrm{H}^2}{2^p}\right)} |\lambda|^n \left|\widehat{u}_0(\lambda)_{k, \ell}\right|^2 \mathrm{d} \lambda<\epsilon \left(\tfrac{\rho_\mathrm{H}}{2^{p+1}}\right)^{-2\beta}.
\end{align}
This leads to
\begin{equation} \label{ineq:varepsilon}
\left(\tfrac{\rho_\mathrm{H}}{2^{p+1}}\right)^{2\beta} \left(\tfrac{\rho_\mathrm{H}}{2^p}\right)^{-2 r_\mathrm{H}-Q} c_n \sum_{k, \ell \in \mathbb{N}^n} \int_{A_p\left(\mu_k^{-1}\rho_\mathrm{H}^2\right)} \mu_k^{-\beta} |\lambda|^{-\beta+n} \left|\widehat{u}_0(\lambda)_{k, \ell}\right|^2 \mathrm{d} \lambda <\epsilon.
\end{equation}
Since $r_\mathrm{H}=q+\beta$, the left hand side of \eqref{ineq:varepsilon} can be rewritten as
\begin{align*}
&\left(\tfrac{\rho_\mathrm{H}}{2^{p+1}}\right)^{2\beta} \left(\tfrac{\rho_\mathrm{H}}{2^p}\right)^{-2 r_\mathrm{H}-Q} c_n \sum_{k, \ell \in \mathbb{N}^n} \int_{A_p\left(\mu_k^{-1}\rho_\mathrm{H}^2\right)} \mu_k^{-\beta} |\lambda|^{-\beta+n} \left|\widehat{u}_0(\lambda)_{k, \ell}\right|^2 \mathrm{d} \lambda \\
=& 2^{p(2 q+Q)} 2^{-2\beta} \rho_\mathrm{H}^{-2 q-Q} c_n \sum_{k, \ell \in \mathbb{N}^n} \int_{A_p\left(\mu_k^{-1}\rho_\mathrm{H}^2\right)} \mu_k^{-\beta} |\lambda|^{-\beta+n} \left|\widehat{u}_0(\lambda)_{k, \ell}\right|^2 \mathrm{d} \lambda.
\end{align*}
Combining this with \eqref{ineq:A_rho}, we obtain
$$
\rho_\mathrm{H}^{-2 q-Q} c_n \sum_{k, \ell \in \mathbb{N}^n} \int_{A_p\left(\mu_k^{-1}\rho_\mathrm{H}^2\right)} \mu_k^{-\beta} |\lambda|^{-\beta+n} \left|\widehat{u}_0(\lambda)_{k, \ell}\right|^2 \mathrm{d} \lambda \leq \epsilon 2^{2\beta} 2^{-p(2 q+Q)}.
$$
Since $\bigcup_{p=0}^{\infty} A_p\left(\mu_k^{-1}\rho_\mathrm{H}^2\right)=B^*\left(\mu_k^{-1}\rho_\mathrm{H}^2\right)$, summing over $p$ gives
$$
\rho_\mathrm{H}^{-2 q-Q} c_n \sum_{k, \ell \in \mathbb{N}^n} \int_{B^*\left(\mu_k^{-1}\rho_\mathrm{H}^2\right)} \mu_k^{-\beta} |\lambda|^{-\beta+n} \left|\widehat{u}_0(\lambda)_{k, \ell}\right|^2 \mathrm{d} \lambda \leq \epsilon 2^{2\beta} \sum_{p=0}^{\infty} 2^{-p(2 q+Q)}.
$$
Therefore,
\begin{equation} \label{ineq:B*}
\rho_\mathrm{H}^{-2 q-Q} c_n \sum_{k, \ell \in \mathbb{N}^n} \int_{B^*\left(\mu_k^{-1}\rho_\mathrm{H}^2\right)} \mu_k^{-\beta} |\lambda|^{-\beta+n} \left|\widehat{u}_0(\lambda)_{k, \ell}\right|^2 \mathrm{d} \lambda \leq \epsilon 2^{2\beta} \tfrac{2^{2q+Q}}{2^{2q+Q}-1}<\infty.
\end{equation}
Since $\epsilon>0$ was arbitrary, it follows that
$$
P_q^{-\beta}\left(u_0\right)_{+}=\limsup_{\rho_\mathrm{H} \rightarrow 0^{+}} \rho_\mathrm{H}^{-2 q-Q} c_n \sum_{k, \ell \in \mathbb{N}^n} \int_{B^*\left(\mu_k^{-1}\rho_\mathrm{H}^2\right)} |\lambda|^n \left|\widehat{v}_0(\lambda)_{k, \ell}\right|^2 \mathrm{d} \lambda<\infty.
$$
Hence, by the Definition \ref{definition_3.7}, we obtain $r_\mathrm{H}=r_\mathrm{H}\left(u_0\right)=q+\beta=r_\mathrm{H}\left(v_0\right)+\beta$. On the other hand, we have
\begin{align} \label{ineq:B*_C}
&c_n \sum_{k, \ell \in \mathbb{N}^n} \int_{{B^*}^{C}\left(\mu_k^{-1}\rho_\mathrm{H}^2\right)} \mu_k^{-\beta} |\lambda|^{-\beta+n} \left|\widehat{u}_0(\lambda)_{k, \ell}\right|^2 \mathrm{d} \lambda \lesssim c_n \sum_{k, \ell \in \mathbb{N}^n} \int_{{B^*}^{C}\left(\mu_k^{-1}\rho_\mathrm{H}^2\right)} |\lambda|^n \left|\widehat{u}_0(\lambda)_{k, \ell}\right|^2 \mathrm{d} \lambda \notag\\
\lesssim & c_n \int_{\mathbb{R}^*}\left\|\widehat{u}_0(t, \lambda)\right\|_{\mathrm{HS}\left[L^2\left(\mathbb{R}^n\right)\right]}^2|\lambda|^n \mathrm{d} \lambda = \left\|u_0(t, \cdot)\right\|_{L^2\left(\mathbf{H}_n\right)}^2<\infty.
\end{align}
By using Plancherel formula and $\left\{e_k\right\}_{k \in \mathbb{N}^n}$ as orthonormal basis of $L^2\left(\mathbb{R}^n\right)$, we have
\begin{align} \label{B*_B*_C}
\left\|v_0(t, \cdot)\right\|_{L^2\left(\mathbf{H}_n\right)}^2 = & c_n \int_{\mathbb{R}^*}\left\|\widehat{v}_0(t, \lambda)\right\|_{\mathrm{HS}\left[L^2\left(\mathbb{R}^n\right)\right]}^2|\lambda|^n \mathrm{d} \lambda \notag\\
=& c_n \sum_{k, \ell \in \mathbb{N}^n}\left(\int_{B^*\left(\mu_k^{-1}\rho_\mathrm{H}^2\right)}+\int_{{B^*}^{C}\left(\mu_k^{-1}\rho_\mathrm{H}^2\right)}\right)\left|\widehat{v}_0(t, \lambda)_{k, \ell}\right|^2|\lambda|^n \mathrm{d} \lambda \notag\\
=& c_n \sum_{k, \ell \in \mathbb{N}^n}\left(\int_{B^*\left(\mu_k^{-1}\rho_\mathrm{H}^2\right)}+\int_{{B^*}^{C}\left(\mu_k^{-1}\rho_\mathrm{H}^2\right)}\right)\left(\mu_k|\lambda|\right)^{-\beta} \left|\widehat{u}_0(t, \lambda)_{k, \ell}\right|^2|\lambda|^n \mathrm{d} \lambda.
\end{align}
From \eqref{ineq:B*}, \eqref{ineq:B*_C} and \eqref{B*_B*_C}, we obtain $v_0 \in L^2\left(\mathbf{H}_n\right)$.
\end{proof}

\subsection{Decay characters in some function spaces on the Heisenberg group} \label{section_3.4}
\begin{proposition} \label{proposition:3.2}
If $\left(u_0, u_1\right) \in\left(H^{\delta_1}\left(\mathbf{H}_n\right) \cap L^1\left(\mathbf{H}_n\right)\right) \times\left(L^2\left(\mathbf{H}_n\right) \cap L^1\left(\mathbf{H}_n\right)\right)$, then $P_{r_\mathrm{H}}\left(u_0\right)_{+}, P_{r_\mathrm{H}}\left(u_1\right)_{+}<\infty$ for $r_\mathrm{H}\left(u_0\right) = \delta_1$ and $r_\mathrm{H}\left(u_1\right) = 0$. Moreover, we also have $r_\mathrm{H}^*\left(u_0\right) \geq \delta_1$ and $r_\mathrm{H}^*\left(u_1\right) \geq 0$.  
\end{proposition}
\begin{proof}[Proof of Proposition \ref{proposition:3.2}]
Let $\left(u_0, u_1\right) \in\left(H^{\delta_1}\left(\mathbf{H}_n\right) \cap L^1\left(\mathbf{H}_n\right)\right) \times\left(L^2\left(\mathbf{H}_n\right) \cap L^1\left(\mathbf{H}_n\right)\right)$. Then, we have that
\begin{align*}
&\rho_\mathrm{H}^{-2 r_\mathrm{H}\left(u_0\right)-Q} c_n \sum_{k, \ell \in \mathbb{N}^n} \int_{B^*\left(\mu_k^{-1}\rho_\mathrm{H}^2\right)} \mu_k^{\delta_1} |\lambda|^{n+\delta_1} \left|\widehat{u}_0(\lambda)_{k, \ell}\right|^2 \mathrm{d} \lambda \\
\lesssim & \rho_\mathrm{H}^{-2 r_\mathrm{H}\left(u_0\right)-Q+2\delta_1} c_n \sum_{k \in \mathbb{N}^n} \int_{B^*\left(\mu_k^{-1}\rho_\mathrm{H}^2\right)} |\lambda|^n \sum_{\ell \in \mathbb{N}^n}|\left(\widehat{u}_0(\lambda) e_k, e_{\ell}\right)_{L^2\left(\mathbb{R}^n\right)}|^2 \mathrm{d} \lambda \\
= & \rho_\mathrm{H}^{-2 r_\mathrm{H}\left(u_0\right)-Q+2\delta_1} c_n \sum_{k \in \mathbb{N}^n} \int_{B^*\left(\mu_k^{-1}\rho_\mathrm{H}^2\right)} |\lambda|^n \left\|\widehat{u}_0(\lambda) e_k\right\|_{L^2\left(\mathbb{R}^n\right)}^2 \mathrm{d} \lambda \\
\lesssim & \rho_\mathrm{H}^{-2 r_\mathrm{H}\left(u_0\right)-Q+2\delta_1} c_n \sum_{k \in \mathbb{N}^n} \int_0^{\mu_k^{-1}\rho_\mathrm{H}^2} \lambda^{n} \mathrm{d} \lambda \left\|u_0\right\|_{L^1\left(\mathbf{H}_n\right)}^2 \\
\lesssim & \rho_\mathrm{H}^{-2 r_\mathrm{H}\left(u_0\right)-Q+2n+2+2\delta_1} c_n \sum_{k \in \mathbb{N}^n} \mu_k^{-(n+1)} \lesssim \rho_\mathrm{H}^{-2 r_\mathrm{H}\left(u_0\right)+2\delta_1},
\end{align*}
where in the third step we employed Parseval's identity $\left\|\widehat{u}_0(\lambda) e_k\right\|_{L^2\left(\mathbb{R}^n\right)}^2=\sum_{\ell \in \mathbb{N}^n}|\left(\widehat{u}_0(\lambda) e_k, e_{\ell}\right)_{L^2\left(\mathbb{R}^n\right)}|^2$, we used \eqref{ineq:2.1} and $\left\|e_k\right\|_{L^2\left(\mathbb{R}^n\right)}=1$ to get the inequality
$$
\left\|\widehat{u}_0(\lambda) e_k\right\|_{L^2\left(\mathbb{R}^n\right)} \leq\left\|\widehat{u}_0(\lambda)\right\|_{\mathcal{L}\left(L^2\left(\mathbb{R}^n\right) \rightarrow L^2\left(\mathbb{R}^n\right)\right)}\left\|e_k\right\|_{L^2\left(\mathbb{R}^n\right)} \leq\left\|u_0\right\|_{L^1\left(\mathbf{H}_n\right)}
$$
and in the last estimate we used that the series $c_n \sum_{k \in \mathbb{N}^n} \mu_k^{-(n+1)}=c_n \sum_{k \in \mathbb{N}^n}(2|k|+n)^{-(n+1)}$ is convergent. Therefore, if $u_0 \in H^{\delta_1}\left(\mathbf{H}_n\right) \cap L^1\left(\mathbf{H}_n\right)$, then $P_{r_\mathrm{H}}\left(u_0\right)_{+}<\infty$ for $r_\mathrm{H}\left(u_0\right) = \delta_1$. Similarly, if $u_1 \in L^2\left(\mathbf{H}_n\right) \cap L^1\left(\mathbf{H}_n\right)$, then $P_{r_\mathrm{H}}\left(u_1\right)_{+}<\infty$ for $r_\mathrm{H}\left(u_1\right) = 0$. From the definition of decay character, we obtain that $r_\mathrm{H}^*\left(u_0\right) \geq \delta_1$ and $r_\mathrm{H}^*\left(u_1\right) \geq 0$.
\end{proof}
\begin{remark}
\rm{
If $\left(u_0, u_1\right) \in\left(H^{\delta_1}\left(\mathbf{H}_n\right) \cap L^1\left(\mathbf{H}_n\right)\right) \times\left(L^2\left(\mathbf{H}_n\right) \cap L^1\left(\mathbf{H}_n\right)\right)$ satisfying the following condition:

$\mathbf{(A)}$ : For some $\varepsilon>0, \exists$ constants $C_i>0, i=1, 2, 3, 4$ such that $C_1 \leq \left\|\widehat{u}_0(\lambda) e_k\right\|_{L^2\left(\mathbb{R}^n\right)} \leq C_2$ and $C_3 \leq \left\|\widehat{u}_1(\lambda) e_k\right\|_{L^2\left(\mathbb{R}^n\right)} \leq C_4$ for all $|\lambda| \leq \varepsilon$, then $r_\mathrm{H}^*\left(u_0\right) = \delta_1$ and $r_\mathrm{H}^*\left(u_1\right) = 0$.

In particular, if $\left(u_0, u_1\right) \in\left(H^{\delta_1}\left(\mathbf{H}_n\right) \cap L^1\left(\mathbf{H}_n\right)\right) \times\left(L^2\left(\mathbf{H}_n\right) \cap L^1\left(\mathbf{H}_n\right)\right)$ has a non-zero mean, i.e.,
\begin{align*}
\left\|\widehat{u}_0(0) e_k\right\|_{L^2\left(\mathbb{R}^n\right)}=&\int_{\mathbf{H}_n} u_0(\eta) \mathrm{d} \eta \left\|e_k\right\|_{L^2\left(\mathbb{R}^n\right)} =\int_{\mathbf{H}_n} u_0(\eta) \mathrm{d} \eta \neq 0, \\
\left\|\widehat{u}_1(0) e_k\right\|_{L^2\left(\mathbb{R}^n\right)}=&\int_{\mathbf{H}_n} u_1(\eta) \mathrm{d} \eta \left\|e_k\right\|_{L^2\left(\mathbb{R}^n\right)}=\int_{\mathbf{H}_n} u_1(\eta) \mathrm{d} \eta \neq 0,
\end{align*}
then, naturally, we also have $r_\mathrm{H}^*\left(u_0\right) = \delta_1$ and $r_\mathrm{H}^*\left(u_1\right) = 0$, since condition $\mathbf{(A)}$ is satisfied for such $u_0$ and $u_1$.
}
\end{remark}
\begin{proposition} \label{proposition:3.3}
We define $\dot{H}^s\left(\mathbf{H}_n\right):=\left\{f \in \mathcal{D}^{\prime}\left(\mathbf{H}_n\right):\left(-\Delta_{\mathrm{H}}\right)^\frac{s}{2} f \in L^2\left(\mathbf{H}_n\right)\right\}$, with $s \in \mathbb{R}$, where $\mathcal{D}^{\prime}\left(\mathbf{H}_n\right)$ is the space of tempered distributions on the Heisenberg group and $\dot{H}^s\left(\mathbf{H}_n\right)$-norm is given by $\|f\|_{\dot{H}^s\left(\mathbf{H}_n\right)}:=\left\|\left(-\Delta_{\mathrm{H}}\right)^\frac{s}{2} f\right\|_{L^2\left(\mathbf{H}_n\right)}$. If $\left(u_0, u_1\right) \in\left(H^{\delta_1}\left(\mathbf{H}_n\right) \cap \dot{H}^{-\gamma}\left(\mathbf{H}_n\right)\right) \times\left(L^2\left(\mathbf{H}_n\right) \cap \dot{H}^{-\gamma}\left(\mathbf{H}_n\right)\right)$, with $\gamma \geq 0$, then $P_{r_\mathrm{H}}\left(u_0\right)_{+}, P_{r_\mathrm{H}}\left(u_1\right)_{+}<\infty$ for $r_\mathrm{H}\left(u_0\right) = \gamma-\frac{Q}{2}+\delta_1$ and $r_\mathrm{H}\left(u_1\right) = \gamma-\frac{Q}{2}$. Moreover, we also have $r_\mathrm{H}^*\left(u_0\right) \geq \gamma-\frac{Q}{2}+\delta_1$ and $r_\mathrm{H}^*\left(u_1\right) \geq \gamma-\frac{Q}{2}$.
\end{proposition}
\begin{proof}[Proof of Proposition \ref{proposition:3.3}]
Let $\left(u_0, u_1\right) \in\left(H^{\delta_1}\left(\mathbf{H}_n\right) \cap \dot{H}^{-\gamma}\left(\mathbf{H}_n\right)\right) \times\left(L^2\left(\mathbf{H}_n\right) \cap \dot{H}^{-\gamma}\left(\mathbf{H}_n\right)\right)$, with $\gamma \geq 0$. Then, we have that
\begin{align*}
&\rho_\mathrm{H}^{-2 r_\mathrm{H}\left(u_0\right)-Q} c_n \sum_{k, \ell \in \mathbb{N}^n} \int_{B^*\left(\mu_k^{-1}\rho_\mathrm{H}^2\right)} \mu_k^{\delta_1} |\lambda|^{n+\delta_1} \left|\widehat{u}_0(\lambda)_{k, \ell}\right|^2 \mathrm{d} \lambda \\
\lesssim & \rho_\mathrm{H}^{-2 r_\mathrm{H}\left(u_0\right)-Q+2\gamma+2\delta_1} c_n \sum_{k, \ell \in \mathbb{N}^n} \int_{B^*\left(\mu_k^{-1}\rho_\mathrm{H}^2\right)} \mu_k^{-\gamma} |\lambda|^{n-\gamma} \left|\widehat{u}_0(\lambda)_{k, \ell}\right|^2 \mathrm{d} \lambda \\
\lesssim & \rho_\mathrm{H}^{-2 r_\mathrm{H}\left(u_0\right)-Q+2\gamma+2\delta_1} \left\|u_0\right\|_{\dot{H}^{-\gamma}\left(\mathbf{H}_n\right)}^2 \lesssim \rho_\mathrm{H}^{-2 r_\mathrm{H}\left(u_0\right)-Q+2\gamma+2\delta_1}.
\end{align*}
Therefore, if $u_0 \in H^{\delta_1}\left(\mathbf{H}_n\right) \cap \dot{H}^{-\gamma}\left(\mathbf{H}_n\right)$, then $P_{r_\mathrm{H}}\left(u_0\right)_{+}<\infty$ for $r_\mathrm{H}\left(u_0\right) = \gamma-\frac{Q}{2}+\delta_1$. Similarly, if $u_1 \in L^2\left(\mathbf{H}_n\right) \cap \dot{H}^{-\gamma}\left(\mathbf{H}_n\right)$, then $P_{r_\mathrm{H}}\left(u_1\right)_{+}<\infty$ for $r_\mathrm{H}\left(u_1\right) = \gamma-\frac{Q}{2}$. From the definition of decay character, we obtain that $r_\mathrm{H}^*\left(u_0\right) \geq \gamma-\frac{Q}{2}+\delta_1$ and $r_\mathrm{H}^*\left(u_1\right) \geq \gamma-\frac{Q}{2}$.
\end{proof}
\begin{proposition} \label{proposition:3.4}
If $\left(u_0, u_1\right) \in \left(H^{\delta_1}\left(\mathbf{H}_n\right) \cap L^m\left(\mathbf{H}_n\right)\right) \times \left(L^2\left(\mathbf{H}_n\right) \cap L^m\left(\mathbf{H}_n\right)\right)$, with $m \in (1,2]$, then $P_{r_\mathrm{H}}\left(u_0\right)_{+}, P_{r_\mathrm{H}}\left(u_1\right)_{+}<\infty$ for $r_\mathrm{H}\left(u_0\right) = -Q\left(1-\frac{1}{m}\right)+\delta_1$ and $r_\mathrm{H}\left(u_1\right) = -Q\left(1-\frac{1}{m}\right)$. Moreover, we also have $r_\mathrm{H}^*\left(u_0\right) \geq -Q\left(1-\frac{1}{m}\right)+\delta_1$ and $r_\mathrm{H}^*\left(u_1\right) \geq -Q\left(1-\frac{1}{m}\right)$.
\end{proposition}
\begin{proof}[Proof of Proposition \ref{proposition:3.4}]
Let $\left(u_0, u_1\right) \in\left(H^{\delta_1}\left(\mathbf{H}_n\right) \cap L^m\left(\mathbf{H}_n\right)\right) \times\left(L^2\left(\mathbf{H}_n\right) \cap L^m\left(\mathbf{H}_n\right)\right)$, with $m \in (1, 2]$. We fix $\gamma \in \left(0, \frac{Q}{2}\right)$ such that $\frac{1}{m}=\frac{1}{2}+\frac{\gamma}{Q}, \text { i.e. } \gamma=Q\left(\frac{1}{m}-\frac{1}{2}\right)$, hence, by Hardy-Littlewood-Sobolev inequality on the Heisenberg group (see Lemma \ref{Hardy_Littlewood}), we get $\|u_1\|_{\dot{H}^{-\gamma}\left(\mathbf{H}_n\right)} \leq C\|u_1\|_{L^m\left(\mathbf{H}_n\right)}$. From the above estimate and Proposition \ref{proposition:3.3}, we see that if $u_1 \in L^2\left(\mathbf{H}_n\right) \cap L^m\left(\mathbf{H}_n\right)$, then $P_{r_\mathrm{H}}\left(u_1\right)_{+}<\infty$ for $r_\mathrm{H}\left(u_1\right) = -Q\left(1-\frac{1}{m}\right)$. Similarly, if $u_0 \in H^{\delta_1}\left(\mathbf{H}_n\right) \cap L^m\left(\mathbf{H}_n\right)$, then $P_{r_\mathrm{H}}\left(u_0\right)_{+}<\infty$ for $r_\mathrm{H}\left(u_0\right) = -Q\left(1-\frac{1}{m}\right)+\delta_1$. From the definition of decay character, we obtain that $r_\mathrm{H}^*\left(u_0\right) \geq -Q\left(1-\frac{1}{m}\right)+\delta_1$ and $r_\mathrm{H}^*\left(u_1\right) \geq -Q\left(1-\frac{1}{m}\right)$.
\end{proof}
\begin{proposition} \label{proposition:3.5}
Let $\left(u_0, u_1\right) \in H_{\sigma \psi(t, \cdot)}^{\delta_1}\left(\mathbf{H}_n\right) \times L_{\sigma \psi(t,)}^2\left(\mathbf{H}_n\right)$, with $
\psi(t, \eta) := \frac{|x|^2+|y|^2+4|\tau|}{8(1+t)}$ for any $\eta=(x, y, \tau) \in \mathbf{H}_n$. Let $\sigma>0$ and $t \geq 0$. Similarly to the Euclidean case considered in \cite{Todorova2001} and \cite{Ikehata2005}, we define the Sobolev spaces $L^2\left(\mathbf{H}_n\right)$ and $H^{\delta_1}\left(\mathbf{H}_n\right)$ with exponential weight $\mathrm{e}^{\sigma \psi(t, \cdot)}$
\begin{align*}
L_{\sigma \psi(t,)}^2\left(\mathbf{H}_n\right):= &\left\{f \in L^2\left(\mathbf{H}_n\right):\|\mathrm{e}^{\sigma \psi(t, \cdot)} f\|_{L^2\left(\mathbf{H}_n\right)}<\infty\right\}, \\
H_{\sigma \psi(t, \cdot)}^{\delta_1}\left(\mathbf{H}_n\right):= &\left\{f \in H^{\delta_1}\left(\mathbf{H}_n\right):\|\mathrm{e}^{\sigma \psi(t, \cdot)} f\|_{L^2\left(\mathbf{H}_n\right)}+\|\mathrm{e}^{\sigma \psi(t, \cdot)} \left(-\Delta_{\mathrm{H}}\right)^\frac{\delta_1}{2} f\|_{L^2\left(\mathbf{H}_n\right)}<\infty\right\},
\end{align*}
endowed with the norms
\begin{align*}
\|f\|_{L_{\sigma \psi(t, \cdot)}^2\left(\mathbf{H}_n\right)}:= &\|\mathrm{e}^{\sigma \psi(t, \cdot)} f\|_{L^2\left(\mathbf{H}_n\right)}, \\
\|f\|_{H_{\sigma \psi(t, \cdot)}^{\delta_1}\left(\mathbf{H}_n\right)} := &\|\mathrm{e}^{\sigma \psi(t, \cdot)} f\|_{L^2\left(\mathbf{H}_n\right)}+\|\mathrm{e}^{\sigma \psi(t, \cdot)} \left(-\Delta_{\mathrm{H}}\right)^\frac{\delta_1}{2} f\|_{L^2\left(\mathbf{H}_n\right)}.
\end{align*}
If $\left(u_0, u_1\right) \in H_{\sigma \psi(t, \cdot)}^1\left(\mathbf{H}_n\right) \times L_{\sigma \psi(t,)}^2\left(\mathbf{H}_n\right)$, then $P_{r_\mathrm{H}}\left(u_0\right)_{+}, P_{r_\mathrm{H}}\left(u_1\right)_{+}<\infty$ for $r_\mathrm{H}\left(u_0\right) = \delta_1$ and $r_\mathrm{H}\left(u_1\right) = 0$. Moreover, we also have $r_\mathrm{H}^*\left(u_0\right) \geq \delta_1$ and $r_\mathrm{H}^*\left(u_1\right) \geq 0$.
\end{proposition}
\begin{proof}[Proof of Proposition \ref{proposition:3.5}]
Let us point out that the requirement $\left(u_0, u_1\right) \in H_{\sigma \psi(t, \cdot)}^{\delta_1}\left(\mathbf{H}_n\right) \times L_{\sigma \psi(t,)}^2\left(\mathbf{H}_n\right)$ is stronger than the assumption $\left(u_0, u_1\right) \in\left(H^{\delta_1}\left(\mathbf{H}_n\right) \cap L^1\left(\mathbf{H}_n\right)\right) \times\left(L^2\left(\mathbf{H}_n\right) \cap L^1\left(\mathbf{H}_n\right)\right)$. Indeed, the embedding
$$
L_{\sigma \psi(t, \cdot)}^2\left(\mathbf{H}_n\right) \hookrightarrow L^1\left(\mathbf{H}_n\right) \cap L^2\left(\mathbf{H}_n\right)
$$
holds for any $\sigma>0$ and $t \geq 0$. Furthermore, by H$\ddot{\text{o}}$lder's interpolation inequality we have also the embedding of $L_{\sigma \psi(t, \cdot)}^2\left(\mathbf{H}_n\right)$ in each $L^\omega\left(\mathbf{H}_n\right)$ for any $\omega \in[1,2]$, where the embedding constant depends on $t$, clearly. Therefore, if $\left(u_0, u_1\right) \in H_{\sigma \psi(t, \cdot)}^{\delta_1}\left(\mathbf{H}_n\right) \times L_{\sigma \psi(t,)}^2\left(\mathbf{H}_n\right)$, then $P_{r_\mathrm{H}}\left(u_0\right)_{+}, P_{r_\mathrm{H}}\left(u_1\right)_{+}<\infty$ for $r_\mathrm{H}\left(u_0\right) = \delta_1$ and $r_\mathrm{H}\left(u_1\right) = 0$. From the definition of decay character, we obtain that $r_\mathrm{H}^*\left(u_0\right) \geq \delta_1$ and $r_\mathrm{H}^*\left(u_1\right) \geq 0$.
\end{proof}
\begin{proposition} \label{proposition:3.6}
Let $q \in \mathbb{R}, k \in \mathbb{N}^n$ and let $e_k \in L^2\left(\mathbb{R}^n\right)$ be the $k$-th Hermite function. We define
$$
\mathcal{Y}^q\left(\mathbf{H}_n\right) := \left\{f\in \mathcal{D}^{\prime}\left(\mathbf{H}_n\right):\widehat{f}(\lambda) \in L_{\mathrm{loc}}^2\left(\mathbb{R}^*, \mathrm{HS}\left(L^2\left(\mathbb{R}^n\right)\right)\right) \text{ and } \underset{\lambda \in \mathbb{R}^*}{\sup}\left\{|\lambda|^q\|\widehat{f}(\lambda) e_k\|_{L^2\left(\mathbb{R}^n\right)}\right\}<\infty \right\},
$$
where $\mathcal{D}^{\prime}\left(\mathbf{H}_n\right)$ is the space of tempered distributions and $\mathcal{Y}^q\left(\mathbf{H}_n\right)$-norm is given by
$$
\|f\|_{\mathcal{Y}^q\left(\mathbf{H}_n\right)} := \underset{\lambda \in \mathbb{R}^*}{\sup}\left\{|\lambda|^q\|\widehat{f}(\lambda) e_k\|_{L^2\left(\mathbb{R}^n\right)}\right\}.
$$
If $\left(u_0, u_1\right) \in\left(H^{\delta_1}\left(\mathbf{H}_n\right) \cap \mathcal{Y}^q\left(\mathbf{H}_n\right)\right) \times\left(L^2\left(\mathbf{H}_n\right) \cap \mathcal{Y}^q\left(\mathbf{H}_n\right)\right)$, with $q \leq \frac{Q}{2}$, then $P_{r_\mathrm{H}}\left(u_0\right)_{+}, P_{r_\mathrm{H}}\left(u_1\right)_{+}<\infty$ for $r_\mathrm{H}\left(u_0\right) = -q+\delta_1$ and $r_\mathrm{H}\left(u_1\right) = -q$. Moreover, we also have $r_\mathrm{H}^*\left(u_0\right) \geq -q+\delta_1$ and $r_\mathrm{H}^*\left(u_1\right) \geq -q$.
\end{proposition}
\begin{proof}[Proof of Proposition \ref{proposition:3.6}]
Let $\left(u_0, u_1\right) \in\left(H^{\delta_1}\left(\mathbf{H}_n\right) \cap \mathcal{Y}^q\left(\mathbf{H}_n\right)\right) \times\left(L^2\left(\mathbf{H}_n\right) \cap \mathcal{Y}^q\left(\mathbf{H}_n\right)\right)$, with $q \leq \frac{Q}{2}$. Then we have
\begin{align*}
&\rho_\mathrm{H}^{-2 r_\mathrm{H}\left(u_0\right)-Q} c_n \sum_{k, \ell \in \mathbb{N}^n} \int_{B^*\left(\mu_k^{-1}\rho_\mathrm{H}^2\right)} \mu_k |\lambda|^{n+\delta_1} \left|\widehat{u}_0(\lambda)_{k, \ell}\right|^2 \mathrm{d} \lambda \\
\lesssim & \rho_\mathrm{H}^{-2 r_\mathrm{H}\left(u_0\right)-Q+2\delta_1} c_n \sum_{k \in \mathbb{N}^n} \int_{B^*\left(\mu_k^{-1}\rho_\mathrm{H}^2\right)} |\lambda|^n \sum_{\ell \in \mathbb{N}^n}|\left(\widehat{u}_0(\lambda) e_k, e_{\ell}\right)_{L^2\left(\mathbb{R}^n\right)}|^2 \mathrm{d} \lambda \\
= & \rho_\mathrm{H}^{-2 r_\mathrm{H}\left(u_0\right)-Q+2\delta_1} c_n \sum_{k \in \mathbb{N}^n} \int_{B^*\left(\mu_k^{-1}\rho_\mathrm{H}^2\right)} |\lambda|^{n-q} |\lambda|^{q} \left\|\widehat{u}_0(\lambda) e_k\right\|_{L^2\left(\mathbb{R}^n\right)}^2 \mathrm{d} \lambda \\
\lesssim & \rho_\mathrm{H}^{-2 r_\mathrm{H}\left(u_0\right)-Q+2\delta_1} c_n \sum_{k \in \mathbb{N}^n} \int_0^{\mu_k^{-1}\rho_\mathrm{H}^2} \lambda^{n-q} \mathrm{d} \lambda \left\|u_0\right\|_{\mathcal{Y}^q\left(\mathbf{H}_n\right)}^2 \\
\lesssim & \rho_\mathrm{H}^{-2 r_\mathrm{H}\left(u_0\right)-Q+2n-2q+2+2\delta_1} c_n \sum_{k \in \mathbb{N}^n} \mu_k^{-(n-q+1)} \lesssim \rho_\mathrm{H}^{-2 r_\mathrm{H}\left(u_0\right)-2q+2\delta_1},
\end{align*}
where in the third step we employed Parseval's identity $\left\|\widehat{u}_0(\lambda) e_k\right\|_{L^2\left(\mathbb{R}^n\right)}^2=\sum_{\ell \in \mathbb{N}^n}|\left(\widehat{u}_0(\lambda) e_k, e_{\ell}\right)_{L^2\left(\mathbb{R}^n\right)}|^2$ and in the last estimate we used that the series $c_n \sum_{k \in \mathbb{N}^n} \mu_k^{-(n-q+1)}=c_n \sum_{k \in \mathbb{N}^n}(2|k|+n)^{-\left(\frac{Q}{2}-q \right)}$ is convergent. Therefore, if $u_0 \in H^{\delta_1}\left(\mathbf{H}_n\right) \cap \mathcal{Y}^q\left(\mathbf{H}_n\right)$, then $P_{r_\mathrm{H}}\left(u_0\right)_{+}<\infty$ for $r_\mathrm{H}\left(u_0\right) = -q+\delta_1$. Similarly, if $u_1 \in L^2\left(\mathbf{H}_n\right) \cap \mathcal{Y}^q\left(\mathbf{H}_n\right)$, then $P_{r_\mathrm{H}}\left(u_1\right)_{+}<\infty$ for $r_\mathrm{H}\left(u_1\right) = -q$. From the definition of decay character, we obtain that $r_\mathrm{H}^*\left(u_0\right) \geq -q+\delta_1$ and $r_\mathrm{H}^*\left(u_1\right) \geq -q$.
\end{proof}
\begin{proposition} \label{proposition:3.7}
We define $\dot{H}_q^s\left(\mathbf{H}_n\right):=\left\{f \in \mathcal{D}^{\prime}\left(\mathbf{H}_n\right):\left(-\Delta_{\mathrm{H}}\right)^\frac{s}{2} f \in L^q\left(\mathbf{H}_n\right)\right\}, q \geq 1, s \in \mathbb{R}$, where $\mathcal{D}^{\prime}\left(\mathbf{H}_n\right)$ is the space of tempered distributions on the Heisenberg group and $\dot{H}_q^s\left(\mathbf{H}_n\right)$-norm is given by $\|f\|_{\dot{H}^s\left(\mathbf{H}_n\right)}:=\left\|\left(-\Delta_{\mathrm{H}}\right)^\frac{s}{2} f\right\|_{L^q\left(\mathbf{H}_n\right)}$. If $\left(u_0, u_1\right) \in\left(H^{\delta_1}\left(\mathbf{H}_n\right) \cap \dot{H}_m^{-\gamma}\left(\mathbf{H}_n\right)\right) \times\left(L^2\left(\mathbf{H}_n\right) \cap \dot{H}_m^{-\gamma}\left(\mathbf{H}_n\right)\right)$, with $m \in (1, 2], \gamma \geq 0$, then $P_{r_\mathrm{H}}\left(u_0\right)_{+}, P_{r_\mathrm{H}}\left(u_1\right)_{+}<\infty$ for $r_\mathrm{H}\left(u_0\right) = \gamma-\frac{Q(m-1)}{m}+\delta_1$ and $r_\mathrm{H}\left(u_1\right) = \gamma-\frac{Q(m-1)}{m}$. Moreover, we also have $r_\mathrm{H}^*\left(u_0\right) \geq \gamma-\frac{Q(m-1)}{m}+\delta_1$ and $r_\mathrm{H}^*\left(u_1\right) \geq \gamma-\frac{Q(m-1)}{m}$.
\end{proposition}
\begin{proof}[Proof of Proposition \ref{proposition:3.7}]
Let $\left(u_0, u_1\right) \in\left(H^{\delta_1}\left(\mathbf{H}_n\right) \cap \dot{H}_m^{-\gamma}\left(\mathbf{H}_n\right)\right) \times\left(L^2\left(\mathbf{H}_n\right) \cap \dot{H}_m^{-\gamma}\left(\mathbf{H}_n\right)\right)$, with $m \in (1, 2]$ and $\gamma \geq 0$. Then, we have
\begin{align*}
&\rho_\mathrm{H}^{-2 r_\mathrm{H}\left(u_0\right)-Q} c_n \sum_{k, \ell \in \mathbb{N}^n} \int_{B^*\left(\mu_k^{-1}\rho_\mathrm{H}^2\right)} \mu_k^{\delta_1} |\lambda|^{n+\delta_1} \left|\widehat{u}_0(\lambda)_{k, \ell}\right|^2 \mathrm{d} \lambda \\
\lesssim & \rho_\mathrm{H}^{-2 r_\mathrm{H}\left(u_0\right)-Q+2\gamma+2\delta_1} c_n \sum_{k \in \mathbb{N}^n} \int_{B^*\left(\mu_k^{-1}\rho_\mathrm{H}^2\right)} \mu_k^{-\gamma} |\lambda|^{n-\gamma} \sum_{\ell \in \mathbb{N}^n}|\left(\widehat{u}_0(\lambda) e_k, e_{\ell}\right)_{L^2\left(\mathbb{R}^n\right)}|^2 \mathrm{d} \lambda \\
=& \rho_\mathrm{H}^{-2 r_\mathrm{H}\left(u_0\right)-Q+2\gamma+2\delta_1} c_n \sum_{k \in \mathbb{N}^n} \int_{B^*\left(\mu_k^{-1}\rho_\mathrm{H}^2\right)} \mu_k^{-\gamma} |\lambda|^{n-\gamma} \left\|\widehat{u}_0(\lambda) e_k\right\|_{L^2\left(\mathbb{R}^n\right)}^2 \mathrm{d} \lambda,
\end{align*}
where in the last step we employed Parseval's identity $\left\|\widehat{u}_0(\lambda) e_k\right\|_{L^2\left(\mathbb{R}^n\right)}^2=\sum_{\ell \in \mathbb{N}^n}|\left(\widehat{u}_0(\lambda) e_k, e_{\ell}\right)_{L^2\left(\mathbb{R}^n\right)}|^2$. Applying H\"{o}lder's inequality with $\frac{1}{m}+\frac{1}{m^{\prime}}=1$, we obtain
\begin{align*}
&\rho_\mathrm{H}^{-2 r_\mathrm{H}\left(u_0\right)-Q+2\gamma+2\delta_1} c_n \sum_{k \in \mathbb{N}^n} \int_{B^*\left(\mu_k^{-1}\rho_\mathrm{H}^2\right)} \mu_k^{-\gamma} |\lambda|^{n-\gamma} \left\|\widehat{u}_0(\lambda) e_k\right\|_{L^2\left(\mathbb{R}^n\right)}^2 \mathrm{d} \lambda \\
\lesssim & \rho_\mathrm{H}^{-2 r_\mathrm{H}\left(u_0\right)-Q+2\gamma+2\delta_1} \left(c_n\sum_{k \in \mathbb{N}^n}\int_{B^*\left(\mu_k^{-1}\rho_\mathrm{H}^2\right)} \left(\mu_k|\lambda|\right)^{-\frac{m^{\prime}\gamma}{2}} \left\|\widehat{u}_0(\lambda) e_k\right\|_{L^2\left(\mathbb{R}^n\right)}^{m^{\prime}} |\lambda|^n \mathrm{d} \lambda\right)^\frac{2}{m^{\prime}} \\
&\times \left(c_n\sum_{k \in \mathbb{N}^n}\int_{B^*\left(\mu_k^{-1}\rho_\mathrm{H}^2\right)} |\lambda|^n \mathrm{d} \lambda\right)^{1-\frac{2}{m^{\prime}}} \lesssim \rho_\mathrm{H}^{-2 r_\mathrm{H}\left(u_0\right)-Q+2\gamma+2\delta_1} \left(c_n \sum_{k \in \mathbb{N}^n}\int_0^{\mu_k^{-1}\rho_\mathrm{H}^2} \lambda^{n} \mathrm{d} \lambda\right)^{1-\frac{2}{m^{\prime}}} \\
\lesssim & \rho_\mathrm{H}^{-2 r_\mathrm{H}\left(u_0\right)-Q+2\gamma+2\delta_1} \left(\rho_\mathrm{H}^{2n+2} c_n \sum_{k \in \mathbb{N}^n} \mu_k^{-(n+1)}\right)^{1-\frac{2}{m^{\prime}}} \lesssim \rho_\mathrm{H}^{-2 r_\mathrm{H}\left(u_0\right)+2\gamma-\frac{2Q(m-1)}{m}+2\delta_1},
\end{align*}
where in the third step we used $u_0\in \dot{H}_m^{-\gamma}\left(\mathbf{H}_n\right)$ and the Hausdorff-Young inequality (see \cite{Fischer2016}) to get the inequality $c_n\sum_{k \in \mathbb{N}^n}\int_{B^*\left(\mu_k^{-1}\rho_\mathrm{H}^2\right)} \left(\mu_k|\lambda|\right)^{-\frac{m^{\prime}\gamma}{2}} \left\|\widehat{u}_0(\lambda) e_k\right\|_{L^2\left(\mathbb{R}^n\right)}^{m^{\prime}} |\lambda|^n \mathrm{d} \lambda < \infty$ and in the last estimate, we used that the series $\left(c_n \sum_{k \in \mathbb{N}^n} \mu_k^{-(n+1)}\right)^{1-\frac{2}{m^{\prime}}}=\left(c_n \sum_{k \in \mathbb{N}^n}(2|k|+n)^{-(n+1)}\right)^{1-\frac{2}{m^{\prime}}}$ is convergent. Therefore, if $u_0 \in H^{\delta_1}\left(\mathbf{H}_n\right) \cap \dot{H}_m^{-\gamma}\left(\mathbf{H}_n\right)$, then $P_{r_\mathrm{H}}\left(u_0\right)_{+}<\infty$ for $r_\mathrm{H}\left(u_0\right) = \gamma-\frac{Q(m-1)}{m}+\delta_1$. Similarly, if $u_1 \in L^2\left(\mathbf{H}_n\right) \cap \dot{H}_m^{-\gamma}\left(\mathbf{H}_n\right)$, then $P_{r_\mathrm{H}}\left(u_1\right)_{+}<\infty$ for $r_\mathrm{H}\left(u_1\right) = \gamma-\frac{Q(m-1)}{m}$. From the definition of decay character, we obtain that $r_\mathrm{H}^*\left(u_0\right) \geq \gamma-\frac{Q(m-1)}{m}+\delta_1$ and $r_\mathrm{H}^*\left(u_1\right) \geq \gamma-\frac{Q(m-1)}{m}$.
\end{proof}

\section{decay estimates of the solution of the linear Cauchy problems} \label{section_4}
In this section, we establish the decay rates of solutions for the linear Cauchy problem \eqref{eq:1.1} that can be expressed in terms of the upper decay indicators on the Heisenberg group of initial data $\left(u_0, u_1\right)$.
\begin{theorem} \label{theorem:1.2}
Let $Q=2n+2$ be the homogeneous dimension on the Heisenberg group. Assume that $\left(u_0, u_1\right) \in H^{\delta_1}\left(\mathbf{H}_n\right) \times L^2\left(\mathbf{H}_n\right)$ such that $P_{r_\mathrm{H}}\left(u_0\right)_{+}, P_{r_\mathrm{H}}\left(u_1\right)_{+}<\infty$ for some $r_\mathrm{H}\left(u_0\right), r_\mathrm{H}\left(u_1\right)$ satisfying $-\frac{Q}{2}\leq r_\mathrm{H}\left(u_0\right), r_\mathrm{H}\left(u_1\right)-2\delta_2$. Then there exists a solution $u \in \mathcal{C}\left([0, \infty), H^{\delta_1}\left(\mathbf{H}_n\right) \right) \cap \mathcal{C}^1\left([0, \infty),  L^2\left(\mathbf{H}_n\right) \right)$ to \eqref{eq:1.1} 
satisfying the following decay estimates for all $\alpha \in \left[0, \delta_1\right]$:
\begin{align}
\label{ineq:1.3} \left\| u(t, \cdot)\right\|_{\dot{H}^\alpha \left(\mathbf{H}_n\right)} \lesssim & (1+t)^{-\frac{2 r_\mathrm{H}\left(u_0\right) +Q+2\alpha}{4\delta_1-4\delta_2}} \|u_0\|_{P, \alpha} + (1+t)^{-\frac{2 r_\mathrm{H}\left(u_1\right) +Q+2\alpha-4\delta_2}{4\delta_1-4\delta_2}} \|u_1\|_{P, 0}, \\
\label{ineq:1.4} \left\|\partial_t u(t, \cdot)\right\|_{L^2\left(\mathbf{H}_n\right)} \lesssim & (1+t)^{-\frac{2 r_\mathrm{H}\left(u_0\right) +Q+4\delta_1-4\delta_2}{4\delta_1-4\delta_2}} \|u_0\|_{P, \delta_1} + (1+t)^{-\frac{2 r_\mathrm{H}\left(u_1\right) +Q+4\delta_1-8\delta_2}{4\delta_1-4\delta_2}} \|u_1\|_{P, 0}.
\end{align}
\end{theorem}
Consider $u$ as a solution to equation \eqref{eq:1.1}. By applying the group Fourier transform to equation \eqref{eq:1.1}, we obtain a Cauchy problem related to a parameter-dependent functional differential equation for $\widehat{u}(t, \lambda)$
$$
\begin{cases}
\partial_t^2 \widehat{u}(t, \lambda)+\sigma_{\left(-\Delta_{\mathrm{H}}\right)^{\delta_2}}(\lambda)\partial_t \widehat{u}(t, \lambda)+\sigma_{\left(-\Delta_{\mathrm{H}}\right)^{\delta_1}}(\lambda) \widehat{u}(t, \lambda)=0, & \lambda \in \mathbb{R}^*, t>0, \\ \widehat{u}(0, \lambda)=\widehat{u}_0(\lambda), \quad \partial_t \widehat{u}(0, \lambda)=\widehat{u}_1(\lambda), & \lambda \in \mathbb{R}^*,\end{cases}
$$
where $\sigma_{\left(-\Delta_{\mathrm{H}}\right)^{\delta_i}}(\lambda)=\left(\mu_k|\lambda|\right)^{\delta_i}$ is the symbol of the fractional sub-Laplacian $\left(-\Delta_{\mathrm{H}}\right)^{\delta_i}$ for any $i=0, 1$ on the Heisenberg group. We now introduce the following notation:
$$
\widehat{u}(t, \lambda)_{k, \ell} :=\left(\widehat{u}(t, \lambda) e_k, e_{\ell}\right)_{L^2\left(\mathbb{R}^n\right)} \text { for any } k, \ell \in \mathbb{N}^n,
$$
where $\left\{e_k\right\}_{k \in \mathbb{N}^n}$ is the system of Hermite functions. Since $\mathrm{H}_w^{\delta_i} e_k=\mu_k^{\delta_i} e_k$ for any $i=0, 1$, it follows that $\widehat{u}(t, \lambda)_{k, \ell}$ solves an ordinary differential equation with respect to the variable $t$, depending on parameters $\lambda \in \mathbb{R}^*$ and $k, \ell \in \mathbb{N}^n$
$$
\begin{cases}
\partial_t^2 \widehat{u}(t, \lambda)_{k, \ell}+\left(\mu_k|\lambda|\right)^{\delta_2} \partial_t \widehat{u}(t, \lambda)_{k, \ell}+\left(\mu_k|\lambda|\right)^{\delta_1} \widehat{u}(t, \lambda)_{k, \ell}=0, & \lambda \in \mathbb{R}^*, t>0, \\
\widehat{u}(0, \lambda)_{k, \ell}=\widehat{u}_0(\lambda)_{k, \ell}, \quad \partial_t \widehat{u}(0, \lambda)_{k, \ell}=\widehat{u}_1(\lambda)_{k, \ell}, & \lambda \in \mathbb{R}^*,
\end{cases}
$$
where $\widehat{u}_h(\lambda)_{k, \ell} :=\left(\widehat{u}_h(\lambda) e_k, e_{\ell}\right)_{L^2\left(\mathbb{R}^n\right)}$ for any $h=0,1$ and any $k, \ell \in \mathbb{N}^n$. The roots of the characteristic equation $\tau^2+\left(\mu_k|\lambda|\right)^{\delta_2}\tau+\left(\mu_k|\lambda|\right)^{\delta_1}=0$ are
$$
\tau_{ \pm}= \begin{cases}-\frac{\left(\mu_k|\lambda|\right)^{\delta_2}}{2} \pm i \sqrt{\left(\mu_k|\lambda|\right)^{\delta_1}-\frac{\left(\mu_k|\lambda|\right)^{2 \delta_2}}{4}} & \text { if } 4 \left(\mu_k|\lambda|\right)^{\delta_1 - 2 \delta_2}>1, \\
-\frac{\left(\mu_k|\lambda|\right)^{\delta_2}}{2} & \text { if } 4 \left(\mu_k|\lambda|\right)^{\delta_1 - 2 \delta_2}=1, \\
-\frac{\left(\mu_k|\lambda|\right)^{\delta_2}}{2} \pm \sqrt{\frac{\left(\mu_k|\lambda|\right)^{2 \delta_2}}{4}-\left(\mu_k|\lambda|\right)^{\delta_1}} & \text { if } 4 \left(\mu_k|\lambda|\right)^{\delta_1 - 2 \delta_2}<1.
\end{cases}
$$
Elementary computations yield the following representation formula:
\begin{align} \label{eq:4.4}
\widehat{u}(t, \lambda)_{k, \ell}= & \widehat{u}_0(\lambda)_{k, \ell} F(t, \lambda, k)+\left(\tfrac{\left(\mu_k|\lambda|\right)^{\delta_2}}{2} \widehat{u}_0(\lambda)_{k, \ell}+\widehat{u}_1(\lambda)_{k, \ell}\right) G(t, \lambda, k) \notag\\
=& \left(F(t, \lambda, k)+\tfrac{\left(\mu_k|\lambda|\right)^{\delta_2}}{2}G(t, \lambda, k)\right)\widehat{u}_0(\lambda)_{k, \ell}+G(t, \lambda, k) \widehat{u}_1(\lambda)_{k, \ell},
\end{align}
where $F$ satisfies
\begin{equation} \label{eq:F}
\begin{cases}
\partial_t^2 F(t, \lambda, k)+\left(\mu_k|\lambda|\right)^{\delta_2} \partial_t F(t, \lambda, k)+\left(\mu_k|\lambda|\right)^{\delta_1} F(t, \lambda, k)=0, \\
F(0, \lambda, k)=1, \quad \partial_t F(0, \lambda, k)=-\frac{\left(\mu_k|\lambda|\right)^{\delta_2}}{2},
\end{cases}
\end{equation}
and $G$ satisfies
\begin{equation} \label{eq:G}
\begin{cases}
\partial_t^2 G(t, \lambda, k)+\left(\mu_k|\lambda|\right)^{\delta_2} \partial_t G(t, \lambda, k)+\left(\mu_k|\lambda|\right)^{\delta_1} G(t, \lambda, k)=0, \\
G(0, \lambda, k)=0, \quad \partial_t G(0, \lambda, k)=1.
\end{cases}
\end{equation}
Problems \eqref{eq:F} and \eqref{eq:G} can be solved explicitly to find $F, G$ in each of the following cases.
\begin{itemize}
\item \textbf{Case A:} $2 \delta_2<\delta_1$
\end{itemize}
\begin{align*}
F(t, \lambda, k):=&
\begin{cases}\mathrm{e}^{-\frac{\left(\mu_k|\lambda|\right)^{\delta_2}}{2} t} \cosh \left(\frac{\left(\mu_k|\lambda|\right)^{\delta_2} \sqrt{1-4\left(\mu_k|\lambda|\right)^{\delta_1-2 \delta_2}}}{2} t\right), & \mu_k|\lambda| \leq 2^{-\left(\frac{\delta_1}{2}- \delta_2\right)^{-1}}, \\
\mathrm{e}^{-\frac{\left(\mu_k|\lambda|\right)^{\delta_2}}{2} t} \cos \left(\frac{\left(\mu_k|\lambda|\right)^{\delta_2} \sqrt{4\left(\mu_k|\lambda|\right)^{\delta_1-2\delta_2}-1}}{2} t\right), & \mu_k|\lambda|>2^{-\left(\frac{\delta_1}{2}- \delta_2\right)^{-1}},
\end{cases} \\
G(t, \lambda, k):=&
\begin{cases}
\frac{2 \mathrm{e}^{-\frac{\left(\mu_k|\lambda|\right)^{\delta_2}}{2}t}}{\left(\mu_k|\lambda|\right)^{\delta_2} \sqrt{1-4\left(\mu_k|\lambda|\right)^{\delta_1-2\delta_2}}} \sinh \left(\frac{\left(\mu_k|\lambda|\right)^{\delta_2} \sqrt{1-4\left(\mu_k|\lambda|\right)^{\delta_1-2\delta_2}}}{2} t\right), & \mu_k|\lambda| \leq 2^{-\left(\frac{\delta_1}{2}- \delta_2\right)^{-1}}, \\
\frac{2 \mathrm{e}^{-\frac{\left(\mu_k|\lambda|\right)^{\delta_2}}{2}t}}{\left(\mu_k|\lambda|\right)^{\delta_2} \sqrt{4\left(\mu_k|\lambda|\right)^{\delta_1-2\delta_2}-1}} \sin \left(\frac{\left(\mu_k|\lambda|\right)^{\delta_2} \sqrt{4\left(\mu_k|\lambda|\right)^{\delta_1-2\delta_2}-1}}{2} t\right), & \mu_k|\lambda|>2^{-\left(\frac{\delta_1}{2}- \delta_2\right)^{-1}}.
\end{cases}
\end{align*}
\begin{itemize}
\item \textbf{Case B:} $2 \delta_2=\delta_1$
\end{itemize}
\begin{align*}
F(t, \lambda, k):=&\mathrm{e}^{-\frac{\left(\mu_k|\lambda|\right)^\frac{\delta_1}{2}}{2} t} \cos \left(\tfrac{\left(\mu_k|\lambda|\right)^\frac{\delta_1}{2} \sqrt{3}}{2} t\right), & \forall \lambda \in \mathbb{R}^*, \\
G(t, \lambda, k):=&\tfrac{2 \mathrm{e}^{-\frac{\left(\mu_k|\lambda|\right)^\frac{\delta_1}{2}}{2} t}}{\left(\mu_k|\lambda|\right)^{\frac{\delta_1}{2}} \sqrt{3}} \sin \left(\tfrac{\left(\mu_k|\lambda|\right)^{\frac{\delta_1}{2}} \sqrt{3}}{2} t\right), & \forall \lambda \in \mathbb{R}^*.
\end{align*}
To show the proof of Theorem \ref{theorem:1.2}, first we notice that, with $c>0, \mu_k|\lambda|<\varepsilon$, we have
\begin{equation} \label{ineq:gamma_beta}
\left| \left(\mu_k|\lambda|\right)^\gamma \mathrm{e}^{-c t \left(\mu_k|\lambda|\right)^\beta} \right| \lesssim (1+t)^{-\frac{\gamma}{\beta}} \quad \text { for } \gamma \geq 0 \text{ and } \beta > 0.
\end{equation}
\begin{proof}[Proof of Theorem \ref{theorem:1.2}] Our analysis below relies on the notion of decay characters of the initial data. We consider separately the cases: $\delta_2 \in\left(0, \frac{\delta_1}{2}\right)$, $\delta_2=0$ and $\delta_2=\frac{\delta_1}{2}$.
\begin{itemize}
\item \textbf{Case 1:} $\mathbf{\delta_2 \in\left(0, \frac{\delta_1}{2}\right)}$.
\end{itemize}

Estimating for $\left\|\partial_t u(t, \cdot)\right\|_{L^2\left(\mathbf{H}_n\right)}$. By using Plancherel formula and $\left\{e_k\right\}_{k \in \mathbb{N}^n}$ as orthonormal basis of $L^2\left(\mathbb{R}^n\right)$, we have
\begin{align} \label{eq:4.5}
&\left\|\partial_t u(t, \cdot)\right\|_{L^2\left(\mathbf{H}_n\right)}^2 =c_n \int_{\mathbb{R}^*}\left\|\partial_t \widehat{u}(t, \lambda)\right\|_{\mathrm{HS}\left[L^2\left(\mathbb{R}^n\right)\right]}^2|\lambda|^n \mathrm{d} \lambda \notag\\
=& c_n \sum_{k, \ell \in \mathbb{N}^n}\left(\int_{0<\mu_k|\lambda|<\varepsilon}+\int_{\varepsilon \leq \mu_k|\lambda| \leq 2^{-\left(\frac{\delta_1}{2}-\delta_2\right)^{-1}}}+\int_{\mu_k|\lambda|>2^{-\left(\frac{\delta_1}{2}-\delta_2\right)^{-1}}}\right)\left|\partial_t \widehat{u}(t, \lambda)_{k, \ell}\right|^2|\lambda|^n \mathrm{d} \lambda \notag\\
:= & J_1+J_2+J_3,
\end{align}
where $0<\varepsilon < 2^{-\left(\frac{\delta_1}{2}-\delta_2\right)^{-1}}$. We begin with $J_1$ first. Observe that
\begin{equation} \label{ineq:4.6}
-4 y \leq-1+\sqrt{1-4 y} \leq-2 y \quad \text { for any } y \in\left[0, \tfrac{1}{4}\right],
\end{equation}
therefore, for $\mu_k|\lambda| \leq 2^{-\left(\frac{\delta_1}{2}-\delta_2 \right)^{-1}}$
\begin{align} \label{ineq:4.7}
&\mathrm{e}^{\frac{-\left(\mu_k|\lambda|\right)^{\delta_2}}{2} t} \cosh \left(\tfrac{\left(\mu_k|\lambda|\right)^{\delta_2} \sqrt{1-4\left(\mu_k|\lambda|\right)^{\delta_1-2\delta_2}}}{2} t\right) \notag\\
\leq & \mathrm{e}^{\frac{-\left(\mu_k|\lambda|\right)^{\delta_2}}{2} t} \mathrm{e}^{\frac{\left(\mu_k|\lambda|\right)^{\delta_2} \sqrt{1-4 \left(\mu_k|\lambda|\right)^{\delta_1-2\delta_2}}}{2} t} \leq \mathrm{e}^{-C\left(\mu_k|\lambda|\right)^{\delta_1-\delta_2} t}
\end{align}
and
\begin{align} \label{ineq:4.8}
&\mathrm{e}^{\frac{-\left(\mu_k|\lambda|\right)^{\delta_2}}{2} t} \sinh \left(\tfrac{\left(\mu_k|\lambda|\right)^{\delta_2} \sqrt{1-4\left(\mu_k|\lambda|\right)^{\delta_1-2\delta_2}}}{2} t\right) \notag\\
\leq & \mathrm{e}^{\frac{-\left(\mu_k|\lambda|\right)^{\delta_2}}{2} t} \mathrm{e}^{\frac{\left(\mu_k|\lambda|\right)^{\delta_2} \sqrt{1-4\left(\mu_k|\lambda|\right)^{\delta_1-2\delta_2}}}{2} t} \leq \mathrm{e}^{-C\left(\mu_k|\lambda|\right)^{\delta_1-\delta_2} t},
\end{align}
where $C$ is a positive constant and we used $\cosh y \leq \mathrm{e}^y$ and $\sinh y \leq \mathrm{e}^y, \forall y \geq 0$. Choose $\varepsilon < 2^{-\left(\frac{\delta_1}{2}-\delta_2\right)^{-1}}$ small enough such that $1-4\left(\mu_k|\lambda|\right)^{\delta_1-2\delta_2}>\frac{1}{2}, \forall \mu_k|\lambda| < \varepsilon$. Then for $\mu_k|\lambda|<\varepsilon$
\begin{align} \label{ineq:4.9}
& \left|\partial_t\left(F(t, \lambda, k)+\tfrac{\left(\mu_k|\lambda|\right)^{\delta_2}}{2} G(t, \lambda, k)\right)\right| \lesssim \left(\mu_k|\lambda|\right)^{\delta_2} \mathrm{e}^{-\frac{\left(\mu_k|\lambda|\right)^{\delta_2}}{2} t} \sinh \left(\tfrac{\left(\mu_k|\lambda|\right)^{\delta_2} \sqrt{1-4\left(\mu_k|\lambda|\right)^{\delta_1-2\delta_2}}}{2} t\right) \notag\\
& \times \left|\sqrt{1-4\left(\mu_k|\lambda|\right)^{\delta_1-2\delta_2}}-\tfrac{1}{\sqrt{1-4\left(\mu_k|\lambda|\right)^{\delta_1-2\delta_2}}}\right|
\lesssim \left(\mu_k|\lambda|\right)^{\delta_1-\delta_2} \mathrm{e}^{-C\left(\mu_k|\lambda|\right)^{\delta_1-\delta_2} t}
\end{align}
and
\begin{align} \label{ineq:4.10}
\left|\partial_tG(t, \lambda, k)\right| \lesssim &\left|\mathrm{e}^{-\frac{\left(\mu_k|\lambda|\right)^{\delta_2}}{2} t} \cosh \left(\tfrac{\left(\mu_k|\lambda|\right)^{\delta_2} \sqrt{1-4\left(\mu_k|\lambda|\right)^{\delta_1-2\delta_2}}}{2} t\right) \right. \notag\\
& \left. -\tfrac{\mathrm{e}^{-\frac{\left(\mu_k|\lambda|\right)^{\delta_2}}{2} t}}{\sqrt{1-4\left(\mu_k|\lambda|\right)^{\delta_1-2\delta_2}}} \sinh \left(\tfrac{\left(\mu_k|\lambda|\right)^{\delta_2} \sqrt{1-4\left(\mu_k|\lambda|\right)^{\delta_1-2\delta_2}}}{2} t\right)\right| \notag\\
\lesssim & \left|\left(1-\tfrac{1}{\sqrt{1-4\left(\mu_k|\lambda|\right)^{\delta_1-2\delta_2}}}\right) \mathrm{e}^{\frac{\left(\mu_k|\lambda|\right)^{\delta_2} t}{2}\left(-1+\sqrt{1-4\left(\mu_k|\lambda|\right)^{\delta_1-2\delta_2}}\right)} \right. \notag\\
& \left. +\left(1+\tfrac{1}{\sqrt{1-4\left(\mu_k|\lambda|\right)^{\delta_1-2\delta_2}}}\right) \mathrm{e}^{\frac{\left(\mu_k|\lambda|\right)^{\delta_2} t}{2}}\left(-1-\sqrt{1-4\left(\mu_k|\lambda|\right)^{\delta_1-2\delta_2}}\right)\right| \notag\\
\lesssim & \left(\mu_k|\lambda|\right)^{\delta_1-2\delta_2} \mathrm{e}^{-C\left(\mu_k|\lambda|\right)^{\delta_1-\delta_2} t}+\mathrm{e}^{-C\left(\mu_k|\lambda|\right)^{\delta_2} t}.
\end{align}
By \eqref{eq:4.4}, \eqref{ineq:4.9} and \eqref{ineq:4.10}, we obtain
\begin{align} \label{ineq:4.11}
\left|\partial_t \widehat{u}(t, \lambda)_{k, \ell}\right| \lesssim & \left(\mu_k|\lambda|\right)^{\delta_1-\delta_2} \mathrm{e}^{-C\left(\mu_k|\lambda|\right)^{\delta_1-\delta_2} t}\left|\widehat{u}_0(\lambda)_{k, \ell}\right| \notag\\
& + \left(\left(\mu_k|\lambda|\right)^{\delta_1-2\delta_2} \mathrm{e}^{-C\left(\mu_k|\lambda|\right)^{\delta_1-\delta_2} t}+\mathrm{e}^{-C\left(\mu_k|\lambda|\right)^{\delta_2} t}\right)\left|\widehat{u}_1(\lambda)_{k, \ell}\right|
\end{align}
for any $\mu_k|\lambda|<\varepsilon$. From \eqref{ineq:4.11}, by \eqref{ineq:gamma_beta} and Theorem \ref{theorem:3.7}, it results
\begin{align} \label{ineq:4.12}
J_1= & c_n \sum_{k, \ell \in \mathbb{N}^n} \int_{0<\mu_k|\lambda|<\varepsilon}\left|\partial_t \widehat{u}(t, \lambda)_{k, \ell}\right|^2|\lambda|^n \mathrm{d} \lambda \notag\\
\lesssim & c_n \sum_{k, \ell \in \mathbb{N}^n} \int_{0<\mu_k|\lambda|<\varepsilon} \left(\mu_k|\lambda|\right)^{2\delta_1-2\delta_2}\mathrm{e}^{-2C\left(\mu_k|\lambda|\right)^{\delta_1-\delta_2} t}\left|\widehat{u}_0(\lambda)_{k, \ell}\right|^2|\lambda|^n \mathrm{d} \lambda \notag\\
& + c_n \sum_{k, \ell \in \mathbb{N}^n} \int_{0<\mu_k|\lambda|<\varepsilon} \left(\mu_k|\lambda|\right)^{2\delta_1-4\delta_2} \mathrm{e}^{-2C\left(\mu_k|\lambda|\right)^{\delta_1-\delta_2} t}\left|\widehat{u}_1(\lambda)_{k, \ell}\right|^2|\lambda|^n \mathrm{d} \lambda \notag\\
& + c_n \sum_{k, \ell \in \mathbb{N}^n} \int_{0<\mu_k|\lambda|<\varepsilon} \mathrm{e}^{-2C\left(\mu_k|\lambda|\right)^{\delta_2} t}\left|\widehat{u}_1(\lambda)_{k, \ell}\right|^2|\lambda|^n \mathrm{d} \lambda \notag\\
\lesssim & (1+t)^{-2} c_n \sum_{k, \ell \in \mathbb{N}^n} \int_{0<\mu_k|\lambda|<\varepsilon} \mathrm{e}^{-C\left(\mu_k|\lambda|\right)^{\delta_1-\delta_2} t}\left|\widehat{u}_0(\lambda)_{k, \ell}\right|^2 |\lambda|^{
n} \mathrm{d} \lambda \notag\\
& + (1+t)^{-\frac{2\delta_1-4\delta_2}{\delta_1-\delta_2}} c_n \sum_{k, \ell \in \mathbb{N}^n} \int_{0<\mu_k|\lambda|<\varepsilon} \mathrm{e}^{-C\left(\mu_k|\lambda|\right)^{\delta_1-\delta_2} t} \left|\widehat{u}_1(\lambda)_{k, \ell}\right|^2 |\lambda|^n \mathrm{d} \lambda \notag\\
& + c_n \sum_{k, \ell \in \mathbb{N}^n} \int_{0<\mu_k|\lambda|<\varepsilon} \mathrm{e}^{-2C\left(\mu_k|\lambda|\right)^{\delta_2} t}\left|\widehat{u}_1(\lambda)_{k, \ell}\right|^2|\lambda|^n \mathrm{d} \lambda \notag\\
\lesssim & (1+t)^{-\frac{2 r_\mathrm{H}\left(u_0\right) +Q+4\delta_1-4\delta_2}{2\delta_1-2\delta_2}} \|u_0\|_{P, 0}^2 + (1+t)^{-\frac{2 r_\mathrm{H}\left(u_1\right) +Q+4\delta_1-8\delta_2}{2\delta_1-2\delta_2}} \|u_1\|_{P, 0}^2 \notag\\
&+ (1+t)^{-\frac{2 r_\mathrm{H}\left(u_1\right) +Q}{2\delta_2}} \|u_1\|_{P, 0}^2 \\
\lesssim & (1+t)^{-\frac{2 r_\mathrm{H}\left(u_0\right) +Q+4\delta_1-4\delta_2}{2\delta_1-2\delta_2}} \|u_0\|_{P, 0}^2 + (1+t)^{-\frac{2 r_\mathrm{H}\left(u_1\right) +Q+4\delta_1-8\delta_2}{2\delta_1-2\delta_2}} \|u_1\|_{P, 0}^2, 
\end{align}
where we used the assumption $-\frac{Q}{2}<r^*\left(u_1\right)-2\delta_2$. We estimate now the other term
$$
J_2=c_n \sum_{k, \ell \in \mathbb{N}^n} \int_{\varepsilon \leq \mu_k|\lambda| \leq 2^{-\left(\frac{\delta_1}{2}-\delta_2\right)^{-1}}} \left|\partial_t \widehat{u}(t, \lambda)_{k, \ell}\right|^2|\lambda|^n \mathrm{d} \lambda.
$$
For $\varepsilon \leq \mu_k|\lambda| \leq 2^{-\left(\frac{\delta_1}{2}-\delta_2\right)^{-1}}$, using $\frac{\sinh y}{y} \leq \cosh y \leq \mathrm{e}^y$ for $y>0$ and \eqref{ineq:4.6}, we see that
\begin{align} \label{ineq:4.13}
&\left|\partial_t\left(F(t, \lambda, k)+\tfrac{\left(\mu_k|\lambda|\right)^{\delta_2}}{2} G(t, \lambda, k)\right)\right| \lesssim \left(\mu_k|\lambda|\right)^{\delta_2} \mathrm{e}^{-\frac{\left(\mu_k|\lambda|\right)^{\delta_2}}{2} t} \sinh \left(\tfrac{\left(\mu_k|\lambda|\right)^{\delta_2} \sqrt{1-4\left(\mu_k|\lambda|\right)^{\delta_1-2\delta_2}}}{2} t\right) \notag\\
&\times \left|\sqrt{1-4\left(\mu_k|\lambda|\right)^{\delta_1-2\delta_2}}-\tfrac{1}{\sqrt{1-4\left(\mu_k|\lambda|\right)^{\delta_1-2\delta_2}}}\right| = 4\left(\mu_k|\lambda|\right)^{\delta_1} \mathrm{e}^{-\frac{\left(\mu_k|\lambda|\right)^{\delta_2}}{2} t} \notag\\
&\times \tfrac{\sinh \left(\frac{\left(\mu_k|\lambda|\right)^{\delta_2} \sqrt{1-4\left(\mu_k|\lambda|\right)^{\delta_1-2\delta_2}}}{2} t\right)}{\left(\mu_k|\lambda|\right)^{\delta_2} \sqrt{1-4\left(\mu_k|\lambda|\right)^{\delta_1-2\delta_2}}} \lesssim \left(\mu_k|\lambda|\right)^{\delta_1} t \mathrm{e}^{-\left(\mu_k|\lambda|\right)^{\delta_1-\delta_2} t} \lesssim \mathrm{e}^{-C t}
\end{align}
and
\begin{align} \label{ineq:4.14}
\left|\partial_tG(t, \lambda, k)\right| \lesssim & \left|\mathrm{e}^{-\frac{\left(\mu_k|\lambda|\right)^{\delta_2}}{2} t} \cosh \left(\tfrac{\left(\mu_k|\lambda|\right)^{\delta_2} \sqrt{1-4\left(\mu_k|\lambda|\right)^{\delta_1-2\delta_2}}}{2} t\right) \right. \notag\\
& \left. +\tfrac{2 \mathrm{e}^{-\frac{\left(\mu_k|\lambda|\right)^{\delta_2}}{2} t}}{\left(\mu_k|\lambda|\right)^{\delta_2} \sqrt{1-4\left(\mu_k|\lambda|\right)^{\delta_1-2\delta_2}}}\left(-\tfrac{\left(\mu_k|\lambda|\right)^{\delta_2}}{2}\right) \sinh \left(\tfrac{\left(\mu_k|\lambda|\right)^{\delta_2} \sqrt{1-4\left(\mu_k|\lambda|\right)^{\delta_1-2\delta_2}}}{2} t\right)\right| \notag\\
\lesssim & \mathrm{e}^{-\left(\mu_k|\lambda|\right)^{\delta_1-\delta_2} t}+\left(\mu_k|\lambda|\right)^{\delta_2} t \mathrm{e}^{-\left(\mu_k|\lambda|\right)^{\delta_1-\delta_2} t} \lesssim \mathrm{e}^{-C t},
\end{align}
where $C$ is a positive constant. From \eqref{eq:4.4}, \eqref{ineq:4.13} and \eqref{ineq:4.14}, leading to
$$
\left|\partial_t \widehat{u}(t, \lambda)_{k, \ell}\right| \lesssim \mathrm{e}^{-C t} \left(\left|\widehat{u}_0(\lambda)_{k, \ell}\right|+ \left|\widehat{u}_1(\lambda)_{k, \ell}\right|\right)
$$
for any $\varepsilon \leq \mu_k|\lambda| \leq 2^{-\left(\frac{\delta_1}{2}-\delta_2\right)^{-1}}$. Therefore,
\begin{align} \label{ineq:4.15}
J_2 \lesssim & \sum_{k, \ell \in \mathbb{N}^n} \int_{\varepsilon \leq \mu_k|\lambda| \leq 2^{-\left(\frac{\delta_1}{2}-\delta_2\right)^{-1}}}\left|\partial_t \widehat{u}(t, \lambda)_{k, \ell}\right|^2|\lambda|^n \mathrm{d} \lambda \notag\\
\lesssim & \mathrm{e}^{-2C t} \sum_{k, \ell \in \mathbb{N}^n} \int_{\varepsilon \leq \mu_k|\lambda| \leq 2^{-\left(\frac{\delta_1}{2}-\delta_2\right)^{-1}}}\left(\left|\widehat{u}_0(\lambda)_{k, \ell}\right|^2+\left|\widehat{u}_1(\lambda)_{k, \ell}\right|^2\right)|\lambda|^n \mathrm{d} \lambda \notag\\
\lesssim & \mathrm{e}^{-2C t} \left(\left\|u_0\right\|_{L^2\left(\mathbf{H}_n\right)}^2+\left\|u_1\right\|_{L^2\left(\mathbf{H}_n\right)}^2\right),
\end{align}
where in the last step we applied Plancherel formula to $u_0$ and $u_1$. Next, we estimate
$$
J_3=c_n \sum_{k, \ell \in \mathbb{N}^n} \int_{\mu_k|\lambda|>2^{-\left(\frac{\delta_1}{2}-\delta_2\right)^{-1}}} \left|\partial_t \widehat{u}(t, \lambda)_{k, \ell}\right|^2|\lambda|^n \mathrm{d} \lambda.
$$
Observe that
\begin{equation} \label{ineq:4.16}
\left|\sin y \right| \leq y, \quad \forall y>0.
\end{equation}
Therefore, for $\mu_k|\lambda|>2^{-\left(\frac{\delta_1}{2}-\delta_2\right)^{-1}}$, we have
\begin{align} \label{ineq:4.17}
&\left|\partial_t \left(F(t, \lambda, k)+\tfrac{\left(\mu_k|\lambda|\right)^{\delta_2}}{2} G(t, \lambda, k)\right)\right| \lesssim \mathrm{e}^{-\frac{\left(\mu_k|\lambda|\right)^{\delta_2}}{2} t} \left(\mu_k|\lambda|\right)^{\delta_2} \sqrt{4\left(\mu_k|\lambda|\right)^{\delta_1-2\delta_2}-1} \notag\\
&\times \left|\sin \left(\tfrac{\left(\mu_k|\lambda|\right)^{\delta_2} \sqrt{4\left(\mu_k|\lambda|\right)^{\delta_1-2\delta_2}-1}}{2} t\right) \right|+\mathrm{e}^{-\frac{\left(\mu_k|\lambda|\right)^{\delta_2}}{2} t} \left(\mu_k|\lambda|\right)^{2\delta_2} \tfrac{1}{\sqrt{4\left(\mu_k|\lambda|\right)^{\delta_1-2\delta_2}-1}} \notag\\
&\times \left|\sin \left(\tfrac{\left(\mu_k|\lambda|\right)^{\delta_2} \sqrt{4\left(\mu_k|\lambda|\right)^{\delta_1-2\delta_2}-1}}{2} t\right)\right|\lesssim \left(\mu_k|\lambda|\right)^{\frac{\delta_1}{2}} \mathrm{e}^{-C t}+\left(\mu_k|\lambda|\right)^{\delta_2} \mathrm{e}^{-C t} \lesssim \left(\mu_k|\lambda|\right)^{\frac{\delta_1}{2}} \mathrm{e}^{-C t}
\end{align}
and
\begin{align} \label{ineq:4.18}
&\left|\partial_t G(t, \lambda, k)\right| \lesssim \left|\mathrm{e}^{-\frac{\left(\mu_k|\lambda|\right)^{\delta_2}}{2} t} \cos \left(\tfrac{\left(\mu_k|\lambda|\right)^{\delta_2} \sqrt{1-4\left(\mu_k|\lambda|\right)^{\delta_1-2\delta_2}}}{2} t\right) +\tfrac{2 \mathrm{e}^{-\frac{\left(\mu_k|\lambda|\right)^{\delta_2}}{2} t}}{\left(\mu_k|\lambda|\right)^{\delta_2} \sqrt{1-4\left(\mu_k|\lambda|\right)^{\delta_1-2\delta_2}}}\left(-\tfrac{\left(\mu_k|\lambda|\right)^{\delta_2}}{2}\right) \right. \notag\\
& \left. \times \sin \left(\tfrac{\left(\mu_k|\lambda|\right)^{\delta_2} \sqrt{1-4\left(\mu_k|\lambda|\right)^{\delta_1-2\delta_2}}}{2} t\right)\right| \lesssim \mathrm{e}^{-\frac{\left(\mu_k|\lambda|\right)^{\delta_2}}{2} t}+\left(\mu_k|\lambda|\right)^{\delta_2} t \mathrm{e}^{-\frac{\left(\mu_k|\lambda|\right)^{\delta_2}}{2} t} \lesssim \mathrm{e}^{-C t},
\end{align}
where $C$ is a positive constant. By \eqref{eq:4.4}, \eqref{ineq:4.17} and \eqref{ineq:4.18}, we see that
$$
\left|\partial_t \widehat{u}(t, \lambda)_{k, \ell}\right| \lesssim \mathrm{e}^{-C t} \left(\left(\mu_k|\lambda|\right)^\frac{\delta_1}{2}\left|\widehat{u}_0(\lambda)_{k, \ell}\right|+ \left|\widehat{u}_1(\lambda)_{k, \ell}\right|\right)
$$
for any $\mu_k|\lambda|>2^{-\left(\frac{\delta_1}{2}-\delta_2\right)^{-1}}$. Therefore,
\begin{align} \label{ineq:4.19}
J_3 \lesssim & \sum_{k, \ell \in \mathbb{N}^n} \int_{\varepsilon \leq \mu_k|\lambda| \leq 2^{-\left(\frac{\delta_1}{2}-\delta_2\right)^{-1}}}\left|\partial_t \widehat{u}(t, \lambda)_{k, \ell}\right|^2|\lambda|^n \mathrm{d} \lambda \notag\\
\lesssim & \mathrm{e}^{-2C t} \sum_{k, \ell \in \mathbb{N}^n} \int_{\varepsilon \leq \mu_k|\lambda| \leq 2^{-\left(\frac{\delta_1}{2}-\delta_2\right)^{-1}}}\left(\left(\mu_k|\lambda|\right)^{\delta_1}\left|\widehat{u}_0(\lambda)_{k, \ell}\right|^2+\left|\widehat{u}_1(\lambda)_{k, \ell}\right|^2\right)|\lambda|^n \mathrm{d} \lambda \notag\\
\lesssim & \mathrm{e}^{-2C t} \left(\left\|u_0\right\|_{\dot{H}^{\delta_1}\left(\mathbf{H}_n\right)}^2+\left\|u_1\right\|_{L^2\left(\mathbf{H}_n\right)}^2\right),
\end{align}
where in the last step we applied Plancherel formula to $u_0$ and $u_1$. So, substituting \eqref{ineq:4.12}, \eqref{ineq:4.15} and \eqref{ineq:4.19} into \eqref{eq:4.5}, we obtain \eqref{ineq:1.4}.

Estimating for $\|u(t, \cdot)\|_{\dot{H}^\alpha \left(\mathbf{H}_n\right)}$. By using Plancherel formula and $\left\{e_k\right\}_{k \in \mathbb{N}^n}$ as orthonormal basis of $L^2\left(\mathbb{R}^n\right)$, we have
\begin{align} \label{eq:4.20}
& \|u(t, \cdot)\|_{\dot{H}^\alpha \left(\mathbf{H}_n\right)}^2 =c_n \int_{\mathbb{R}^*}\|\left(-\sigma_{\Delta_{\mathrm{H}}}(\lambda)\right)^{\frac{\alpha}{2}}\widehat{u}(t, \lambda)\|_{\mathrm{HS}\left[L^2\left(\mathbb{R}^n\right)\right]}^2|\lambda|^n \mathrm{d} \lambda \notag\\
=& c_n \sum_{k, \ell \in \mathbb{N}^n} \int_{\mathbb{R}^*} \mu_k^\alpha |\lambda|^{\alpha +
n} \left(\widehat{u}(t, \lambda) e_k, e_{\ell}\right)_{L^2\left(\mathbb{R}^n\right)}^2 \mathrm{d} \lambda \notag\\
=&c_n \sum_{k, \ell \in \mathbb{N}^n}\left(\int_{0<\mu_k|\lambda|<\varepsilon}+\int_{\varepsilon \leq \mu_k|\lambda| \leq 2^{-\left(\frac{\delta_1}{2}-\delta_2\right)^{-1}}}+\int_{\mu_k|\lambda|>2^{-\left(\frac{\delta_1}{2}-\delta_2\right)^{-1}}} \right) \mu_k^\alpha |\lambda|^{\alpha +
n} \left|\widehat{u}(t, \lambda)_{k, \ell}\right|^2 \mathrm{d} \lambda \notag\\
:= & I_1+I_2+I_3,
\end{align}
where $0<\varepsilon < 2^{-\left(\frac{\delta_1}{2}-\delta_2\right)^{-1}}$. We estimate $I_1$ first. Choose $\varepsilon < 2^{-\left(\frac{\delta_1}{2}-\delta_2\right)^{-1}}$ small enough such that $1-4\left(\mu_k|\lambda|\right)^{\delta_1-2\delta_2}>\frac{1}{2}$ for $\mu_k|\lambda| < \varepsilon$. Then for $\mu_k|\lambda|<\varepsilon$, using \eqref{ineq:4.7}, \eqref{ineq:4.8} and $1-4\left(\mu_k|\lambda|\right)^{\delta_1-2\delta_2}>\frac{1}{2}$, we obtain
\begin{align} \label{ineq:4.21}
& \left|F(t, \lambda, k)+\tfrac{\left(\mu_k|\lambda|\right)^{\delta_2}}{2} G(t, \lambda, k)\right|
\lesssim \mathrm{e}^{-\frac{\left(\mu_k|\lambda|\right)^{\delta_2}}{2} t} \cosh \left(\tfrac{\left(\mu_k|\lambda|\right)^{\delta_2} \sqrt{1-4\left(\mu_k|\lambda|\right)^{\delta_1-2\delta_2}}}{2} t\right) \notag\\
&+ \tfrac{\mathrm{e}^{-\frac{\left(\mu_k|\lambda|\right)^{\delta_2}}{2} t}}{\sqrt{1-4\left(\mu_k|\lambda|\right)^{\delta_1-2\delta_2}}} \sinh \left(\tfrac{\left(\mu_k|\lambda|\right)^{\delta_2} \sqrt{1-4\left(\mu_k|\lambda|\right)^{\delta_1-2\delta_2}}}{2} t\right) \lesssim \mathrm{e}^{-C\left(\mu_k|\lambda|\right)^{\delta_1-\delta_2} t}
\end{align}
and
\begin{equation} \label{ineq:4.22}
\left|G(t, \lambda, k)\right| \lesssim \tfrac{\mathrm{e}^{-\frac{\left(\mu_k|\lambda|\right)^{\delta_2}}{2} t}}{\left(\mu_k|\lambda|\right)^{\delta_2}\sqrt{1-4\left(\mu_k|\lambda|\right)^{\delta_1-2\delta_2}}} \sinh \left(\tfrac{\left(\mu_k|\lambda|\right)^{\delta_2} \sqrt{1-4\left(\mu_k|\lambda|\right)^{\delta_1-2\delta_2}}}{2} t\right) \lesssim \tfrac{\mathrm{e}^{-C\left(\mu_k|\lambda|\right)^{\delta_1-\delta_2} t}}{\left(\mu_k|\lambda|\right)^{\delta_2}},
\end{equation}
where $C$ is a positive constant. By \eqref{eq:4.4}, \eqref{ineq:4.21} and \eqref{ineq:4.22}, we obtain
\begin{equation} \label{ineq:4.23}
\left|\widehat{u}(t, \lambda)_{k, \ell}\right| \lesssim \mathrm{e}^{-C\left(\mu_k|\lambda|\right)^{\delta_1-\delta_2} t}\left|\widehat{u}_0(\lambda)_{k, \ell}\right| + \tfrac{\mathrm{e}^{-C\left(\mu_k|\lambda|\right)^{\delta_1-\delta_2} t}}{\left(\mu_k|\lambda|\right)^{\delta_2}}\left|\widehat{u}_1(\lambda)_{k, \ell}\right|
\end{equation}
for any $\mu_k|\lambda|<\varepsilon$. We consider several cases

For $2 \delta_2 \leq \alpha \leq \delta_1$. From \eqref{ineq:4.23}, by \eqref{ineq:gamma_beta} and Theorem \ref{theorem:3.7}, it results
\begin{align} \label{ineq:4.24}
I_1= & c_n \sum_{k, \ell \in \mathbb{N}^n} \int_{0<\mu_k|\lambda|<\varepsilon}\left(\mu_k|\lambda|\right)^\alpha \left|\widehat{u}(t, \lambda)_{k, \ell}\right|^2|\lambda|^n \mathrm{d} \lambda \notag\\
\lesssim & c_n \sum_{k, \ell \in \mathbb{N}^n} \int_{0<\mu_k|\lambda|<\varepsilon} \left(\mu_k|\lambda|\right)^\alpha \mathrm{e}^{-2C\left(\mu_k|\lambda|\right)^{\delta_1-\delta_2} t}\left|\widehat{u}_0(\lambda)_{k, \ell}\right|^2|\lambda|^n \mathrm{d} \lambda \notag\\
& + c_n \sum_{k, \ell \in \mathbb{N}^n} \int_{0<\mu_k|\lambda|<\varepsilon} \left(\mu_k|\lambda|\right)^{\alpha-2\delta_2} \mathrm{e}^{-2C\left(\mu_k|\lambda|\right)^{\delta_1-\delta_2} t} \left|\widehat{u}_1(\lambda)_{k, \ell}\right|^2|\lambda|^n \mathrm{d} \lambda \notag\\
\lesssim & (1+t)^{-\frac{\alpha}{\delta_1-\delta_2}} c_n \sum_{k, \ell \in \mathbb{N}^n} \int_{0<\mu_k|\lambda|<\varepsilon} \mathrm{e}^{-C\left(\mu_k|\lambda|\right)^{\delta_1-\delta_2} t}\left|\widehat{u}_0(\lambda)_{k, \ell}\right|^2 |\lambda|^{
n} \mathrm{d} \lambda \notag\\
& + (1+t)^{-\frac{\alpha-2\delta_2}{\delta_1-\delta_2}} c_n \sum_{k, \ell \in \mathbb{N}^n} \int_{0<\mu_k|\lambda|<\varepsilon} \mathrm{e}^{-C\left(\mu_k|\lambda|\right)^{\delta_1-\delta_2} t} \left|\widehat{u}_1(\lambda)_{k, \ell}\right|^2 |\lambda|^n \mathrm{d} \lambda \notag\\
\lesssim & (1+t)^{-\frac{2 r_\mathrm{H}\left(u_0\right) +Q+2\alpha}{2\delta_1-2\delta_2}} \|u_0\|_{P, 0}^2 + (1+t)^{-\frac{2 r_\mathrm{H}\left(u_1\right) +Q+2\alpha-4\delta_2}{2\delta_1-2\delta_2}} \|u_1\|_{P, 0}^2. 
\end{align}

For $0 \leq \alpha<2 \delta_2$. In that case, thanks to our assumption $-\frac{Q}{2}<r_\mathrm{H}\left(u_1\right)-2\delta_2$, we have $\frac{Q}{2}+r_\mathrm{H}\left(u_1\right)-(2\delta_2-\alpha)>0$. Since $\alpha-2 \delta_2<0$, we cannot apply the technique from the previous case to estimate the integral with small frequencies $c_n \sum_{k, \ell \in \mathbb{N}^n} \int_{0<\mu_k|\lambda|<\varepsilon} \left(\mu_k|\lambda|\right)^{\alpha-2\delta_2} \mathrm{e}^{-2C\left(\mu_k|\lambda|\right)^{\delta_1-\delta_2} t} \left|\widehat{u}_1(\lambda)_{k, \ell}\right|^2|\lambda|^n \mathrm{d} \lambda$. However, this term can be estimated as follows, by applying the result stated in Theorem \ref{theorem:3.11}. Put $\widehat{v}_1(\lambda)_{k, \ell}=\left(\mu_k|\lambda|\right)^{-\left(\delta_2-\frac{\alpha}{2}\right)} \widehat{u}_1(\lambda)_{k, \ell}$. Since $\frac{Q}{2}+r_\mathrm{H}\left(u_1\right)-(2\delta_2-\alpha)>0$, we can apply Theorem \ref{theorem:3.11} to obtain that $v_1 \in L^2\left(\mathbf{H}_n\right)$ and $r_\mathrm{H}\left(v_1\right)=r_\mathrm{H}\left(u_1\right)-(2\delta_2-\alpha)$. Hence, by Theorem \ref{theorem:3.7}, we have
\begin{align*}
& c_n \sum_{k, \ell \in \mathbb{N}^n} \int_{0<\mu_k|\lambda|<\varepsilon} \left(\mu_k|\lambda|\right)^{\alpha-2\delta_2} \mathrm{e}^{-2C\left(\mu_k|\lambda|\right)^{\delta_1-\delta_2} t} \left|\widehat{u}_1(\lambda)_{k, \ell}\right|^2|\lambda|^n \mathrm{d} \lambda \\
=& c_n \sum_{k, \ell \in \mathbb{N}^n} \int_{0<\mu_k|\lambda|<\varepsilon} \mathrm{e}^{-2C\left(\mu_k|\lambda|\right)^{\delta_1-\delta_2} t} \left|\widehat{v}_1(\lambda)_{k, \ell}\right|^2|\lambda|^n \mathrm{d} \lambda \\
\lesssim & P_{r_\mathrm{H}\left(v_1\right)}(v_1)\left(C_3+t\right)^{-\frac{2 r_\mathrm{H}\left(v_1\right) +Q}{2\delta_1-2\delta_2}}+\left\|v_1\right\|_{L^2\left(\mathbf{H}_n\right)}^2\left(C_3+t\right)^{-m}
\end{align*}
for any $m>\frac{2 r_\mathrm{H}\left(v_1\right) +Q}{2\delta_1-2\delta_2}$ and for sufficiently large $C_3=$ const $>0$.
By fixing $m>\frac{2 r_\mathrm{H}\left(v_1\right) +Q}{2\delta_1-2\delta_2}$ and choosing $C_3$ large enough, we always can achieve that
$$
\left\|v_1\right\|_{L^2\left(\mathbf{H}_n\right)}^2\left(C_3+t\right)^{-m} \lesssim\left(P_{r_\mathrm{H}\left(u_1\right)}\left(u_1\right)+\left\|u_1\right\|_{L^2\left(\mathbf{H}_n\right)}^2\right)(1+t)^{-\frac{2 r_\mathrm{H}\left(v_1\right) +Q}{2\delta_1-2\delta_2}}.
$$
Therefore, recalling the definition of the norm $\|\cdot\|_{P, 0}$ we obtain
\begin{align}
&c_n \sum_{k, \ell \in \mathbb{N}^n} \int_{0<\mu_k|\lambda|<\varepsilon} \left(\mu_k|\lambda|\right)^{\alpha-2\delta_2} \mathrm{e}^{-2C\left(\mu_k|\lambda|\right)^{\delta_1-\delta_2} t} \left|\widehat{u}_1(\lambda)_{k, \ell}\right|^2|\lambda|^n \mathrm{d} \lambda \notag\\
\lesssim & \left\|u_1\right\|_{P, 0}^2(1+t)^{-\frac{2 r_\mathrm{H}\left(u_1\right) +Q+2\alpha-4\delta_2}{2\delta_1-2\delta_2}}.
\end{align}
Therefore, similar to the case when $2 \delta_2 \leq \alpha \leq \delta_1$, we also have
\begin{align} \label{ineq:4.25}
I_1= & c_n \sum_{k, \ell \in \mathbb{N}^n} \int_{0<\mu_k|\lambda|<\varepsilon}\left(\mu_k|\lambda|\right)^\alpha \left|\widehat{u}(t, \lambda)_{k, \ell}\right|^2|\lambda|^n \mathrm{d} \lambda \notag\\
\lesssim & (1+t)^{-\frac{2 r_\mathrm{H}\left(u_0\right) +Q+2\alpha}{2\delta_1-2\delta_2}} \|u_0\|_{P, 0}^2 + (1+t)^{-\frac{2 r_\mathrm{H}\left(u_1\right) +Q+2\alpha-4\delta_2}{2\delta_1-2\delta_2}} \|u_1\|_{P, 0}^2. 
\end{align}
We estimate now the other term $I_2$. For $\varepsilon \leq \mu_k|\lambda| \leq 2^{-\left(\frac{\delta_1}{2}-\delta_2\right)^{-1}}$, using $\frac{\sinh y}{y} \leq \cosh y \leq \mathrm{e}^y$ for $y>0$ and \eqref{ineq:4.6}, we see that
\begin{align} \label{ineq:4.26}
&\left|F(t, \lambda, k)+\tfrac{\left(\mu_k|\lambda|\right)^{\delta_2}}{2} G(t, \lambda, k)\right| \\
\lesssim & \mathrm{e}^{-\frac{\left(\mu_k|\lambda|\right)^{\delta_2}}{2} t} \cosh \left(\tfrac{\left(\mu_k|\lambda|\right)^{\delta_2} \sqrt{1-4\left(\mu_k|\lambda|\right)^{\delta_1-2\delta_2}}}{2} t\right) + \tfrac{\mathrm{e}^{-\frac{\left(\mu_k|\lambda|\right)^{\delta_2}}{2} t}}{\sqrt{1-4\left(\mu_k|\lambda|\right)^{\delta_1-2\delta_2}}} \sinh \left(\tfrac{\left(\mu_k|\lambda|\right)^{\delta_2} \sqrt{1-4\left(\mu_k|\lambda|\right)^{\delta_1-2\delta_2}}}{2} t\right) \notag\\
\lesssim & \mathrm{e}^{-\left(\mu_k|\lambda|\right)^{\delta_1-\delta_2} t} + \left(\mu_k|\lambda|\right)^{\delta_2} t \mathrm{e}^{-\left(\mu_k|\lambda|\right)^{\delta_1-\delta_2} t} \lesssim \mathrm{e}^{-C t}
\end{align}
and
\begin{align} \label{ineq:4.27}
\left(\mu_k|\lambda|\right)^\frac{\alpha}{2} \left|G(t, \lambda, k)\right| \lesssim & \tfrac{\mathrm{e}^{-\frac{\left(\mu_k|\lambda|\right)^{\delta_2}}{2} t}}{\left(\mu_k|\lambda|\right)^{\delta_2-\frac{\alpha}{2}}\sqrt{1-4\left(\mu_k|\lambda|\right)^{\delta_1-2\delta_2}}} \sinh \left(\tfrac{\left(\mu_k|\lambda|\right)^{\delta_2} \sqrt{1-4\left(\mu_k|\lambda|\right)^{\delta_1-2\delta_2}}}{2} t\right) \notag\\
\lesssim & \left(\mu_k|\lambda|\right)^\frac{\alpha}{2} t \mathrm{e}^{-\left(\mu_k|\lambda|\right)^{\delta_1-\delta_2} t} \lesssim \mathrm{e}^{-C t},
\end{align}
where $C$ is a positive constant. From \eqref{eq:4.4}, \eqref{ineq:4.26} and \eqref{ineq:4.27}, leading to
$$
\left|\widehat{u}(t, \lambda)_{k, \ell}\right| \lesssim \mathrm{e}^{-C t}\left(\left|\widehat{u}_0(\lambda)_{k, \ell}\right|+\left(\mu_k|\lambda|\right)^{-\frac{\alpha}{2}} \left|\widehat{u}_1(\lambda)_{k, \ell}\right|\right)
$$
for any $\varepsilon \leq \mu_k|\lambda| \leq 2^{-\left(\frac{\delta_1}{2}-\delta_2\right)^{-1}}$. Therefore,
\begin{align} \label{ineq:4.28}
I_2 \lesssim & \mathrm{e}^{-2C t} \sum_{k, \ell \in \mathbb{N}^n} \int_{\varepsilon \leq \mu_k|\lambda| \leq 2^{-\left(\frac{\delta_1}{2}-\delta_2\right)^{-1}}} \mu_k^\alpha |\lambda|^{\alpha +
n} \left(\left|\widehat{u}_0(\lambda)_{k, \ell}\right|^2+\left(\mu_k|\lambda|\right)^{-\alpha}\left|\widehat{u}_1(\lambda)_{k, \ell}\right|^2\right) \mathrm{d}
 \lambda \notag\\
\lesssim & \mathrm{e}^{-2C t} \int_{\mathbb{R}^*} \sum_{k, \ell \in \mathbb{N}^n} \mu_k^\alpha |\lambda|^{\alpha +
n} \left(\left|\widehat{u}_0(\lambda)_{k, \ell}\right|^2+\left(\mu_k|\lambda|\right)^{-\alpha}\left|\widehat{u}_1(\lambda)_{k, \ell}\right|^2\right) \mathrm{d}
 \lambda \notag\\
=&\mathrm{e}^{-2C t} \int_{\mathbb{R}^*}\left(\|\left(-\sigma_{\Delta_{\mathrm{H}}}(\lambda)\right)^{\frac{\alpha}{2}}\widehat{u}_0(\lambda)\|_{\mathrm{HS}\left[L^2\left(\mathbb{R}^n\right)\right]}^2+\left\|\widehat{u}_1(\lambda)\right\|_{\mathrm{HS}\left[L^2\left(\mathbb{R}^n\right)\right]}^2\right)|\lambda|^n \mathrm{d} \lambda \notag\\
\approx & \mathrm{e}^{-2C t}\left(\left\|u_0\right\|_{\dot{H}^\alpha \left(\mathbf{H}_n\right)}^2+\left\|u_1\right\|_{L^2\left(\mathbf{H}_n\right)}^2\right),
\end{align}
where in the last step we applied Plancherel formula to $u_0$ and $u_1$. Next, we estimate $I_3$. For $\mu_k|\lambda|>2^{-\left(\frac{\delta_1}{2}-\delta_2\right)^{-1}}$, using \eqref{ineq:4.16} and inequality $y < \mathrm{e}^y, \forall y > 0$, we see that
\begin{align} \label{ineq:4.29}
&\left|F(t, \lambda, k)+\tfrac{\left(\mu_k|\lambda|\right)^{\delta_2}}{2} G(t, \lambda, k)\right| \notag\\
\lesssim & \mathrm{e}^{-\frac{\left(\mu_k|\lambda|\right)^{\delta_2}}{2} t} \left|\cos \left(\tfrac{\left(\mu_k|\lambda|\right)^{\delta_2} \sqrt{4\left(\mu_k|\lambda|\right)^{\delta_1-2\delta_2}-1}}{2} t\right) \right| +\tfrac{\mathrm{e}^{-\frac{\left(\mu_k|\lambda|\right)^{\delta_2}}{2} t}}{\sqrt{4\left(\mu_k|\lambda|\right)^{\delta_1-2\delta_2}-1}} \left|\sin \left(\tfrac{\left(\mu_k|\lambda|\right)^{\delta_2} \sqrt{4\left(\mu_k|\lambda|\right)^{\delta_1-2\delta_2}-1}}{2} t\right) \right| \notag\\
\lesssim & \mathrm{e}^{-\frac{\left(\mu_k|\lambda|\right)^{\delta_2}}{2} t}+\left(\mu_k|\lambda|\right)^{\delta_2} t \mathrm{e}^{-\frac{\left(\mu_k|\lambda|\right)^{\delta_2}}{2} t}=\mathrm{e}^{-\frac{\left(\mu_k|\lambda|\right)^{\delta_2}}{2} t}+\left(\mu_k|\lambda|\right)^{\delta_2} t \mathrm{e}^{-\frac{\left(\mu_k|\lambda|\right)^{\delta_2}}{4} t} \mathrm{e}^{-\frac{\left(\mu_k|\lambda|\right)^{\delta_2}}{4} t} \notag\\
\lesssim & \mathrm{e}^{-\frac{\left(\mu_k|\lambda|\right)^{\delta_2}}{2} t}+\mathrm{e}^{-\frac{\left(\mu_k|\lambda|\right)^{\delta_2}}{4} t} \lesssim \mathrm{e}^{-C t}
\end{align}
and
\begin{align} \label{ineq:4.30}
\left(\mu_k|\lambda|\right)^\frac{\alpha}{2}\left|G(t, \lambda, k)\right| \lesssim & \left(\mu_k|\lambda|\right)^\frac{\alpha}{2}\left|\tfrac{\mathrm{e}^{-\frac{\left(\mu_k|\lambda|\right)^{\delta_2}}{2} t}}{\left(\mu_k|\lambda|\right)^{\delta_2} \sqrt{4\left(\mu_k|\lambda|\right)^{\delta_1-2\delta_2}-1}} \sin \left(\tfrac{\left(\mu_k|\lambda|\right)^{\delta_2} \sqrt{4\left(\mu_k|\lambda|\right)^{\delta_1-2\delta_2}-1}}{2} t\right)\right| \notag\\
\lesssim & \left(\mu_k|\lambda|\right)^\frac{\alpha}{2} t \mathrm{e}^{-\frac{\left(\mu_k|\lambda|\right)^{\delta_2}}{2} t} = \left(\mu_k|\lambda|\right)^\alpha t \mathrm{e}^{-\frac{\left(\mu_k|\lambda|\right)^{\delta_2}}{4} t} \mathrm{e}^{-\frac{\left(\mu_k|\lambda|\right)^{\delta_2}}{4} t} \notag\\
\lesssim & \left(\mu_k|\lambda|\right)^{\frac{\alpha}{2}-\delta_2}\mathrm{e}^{-\frac{\left(\mu_k|\lambda|\right)^{\delta_2}}{4} t} \lesssim \mathrm{e}^{-C t},
\end{align}
for some $C>0$, where $C$ and the unexpressed multiplicative constant hereafter are independent of the time variable and of the parameters $\lambda$ and $k, \ell$ as well. Then, by \eqref{eq:4.4}, \eqref{ineq:4.29} and \eqref{ineq:4.30}, we obtain
$$
\left|\widehat{u}(t, \lambda)_{k, \ell}\right| \lesssim \mathrm{e}^{-C t}\left(\left|\widehat{u}_0(\lambda)_{k, \ell}\right|+\left(\mu_k|\lambda|\right)^{-\frac{\alpha}{2}} \left|\widehat{u}_1(\lambda)_{k, \ell}\right|\right)
$$
for any $\mu_k|\lambda|>2^{-\left(\frac{\delta_1}{2}-\delta_2\right)^{-1}}$. Therefore,
\begin{align} \label{ineq:4.31}
I_3 \lesssim & \mathrm{e}^{-2C t} \sum_{k, \ell \in \mathbb{N}^n} \int_{\mu_k|\lambda|>2^{-\left(\frac{\delta_1}{2}-\delta_2\right)^{-1}}} \mu_k^\alpha |\lambda|^{\alpha +
n} \left(\left|\widehat{u}_0(\lambda)_{k, \ell}\right|^2+\left(\mu_k|\lambda|\right)^{-\alpha}\left|\widehat{u}_1(\lambda)_{k, \ell}\right|^2\right) \mathrm{d}
 \lambda \notag\\
\lesssim & \mathrm{e}^{-2C t} \int_{\mathbb{R}^*} \sum_{k, \ell \in \mathbb{N}^n} \mu_k^\alpha |\lambda|^{\alpha +
n} \left(\left|\widehat{u}_0(\lambda)_{k, \ell}\right|^2+\left(\mu_k|\lambda|\right)^{-\alpha}\left|\widehat{u}_1(\lambda)_{k, \ell}\right|^2\right) \mathrm{d}
 \lambda \notag\\
=&\mathrm{e}^{-2C t} \int_{\mathbb{R}^*}\left(\|\left(-\sigma_{\Delta_{\mathrm{H}}}(\lambda)\right)^{\frac{\alpha}{2}}\widehat{u}_0(\lambda)\|_{\mathrm{HS}\left[L^2\left(\mathbb{R}^n\right)\right]}^2+\left\|\widehat{u}_1(\lambda)\right\|_{\mathrm{HS}\left[L^2\left(\mathbb{R}^n\right)\right]}^2\right)|\lambda|^n \mathrm{d}
 \lambda \notag\\
\approx & \mathrm{e}^{-2C t}\left(\left\|u_0\right\|_{\dot{H}^\alpha \left(\mathbf{H}_n\right)}^2+\left\|u_1\right\|_{L^2\left(\mathbf{H}_n\right)}^2\right),
\end{align}
where in the last step we applied Plancherel formula to $u_0$ and $u_1$. So, substituting \eqref{ineq:4.24}, \eqref{ineq:4.25}, \eqref{ineq:4.28} and \eqref{ineq:4.31} into \eqref{eq:4.20}, we obtain \eqref{ineq:1.3}.
\begin{itemize}
\item \textbf{Case 2:} $\mathbf{\delta_2=0}$.
\end{itemize}

Estimating for $\left\|\partial_t u(t, \cdot)\right\|_{L^2\left(\mathbf{H}_n\right)}$. Applying the obvious inequalities \eqref{ineq:4.6}, we can now proceed analogously to the arguments used in the previous case with $\delta_2 \in \left(0, \frac{\delta_1}{2} \right)$ to estimate $\partial_t\left(F(t, \lambda, k)\right)+\frac{1}{2} G(t, \lambda, k)$ as follows:
\begin{align*}
\left|\partial_t\left(F(t, \lambda, k)+\tfrac{1}{2} G(t, \lambda, k)\right)\right| \lesssim & \mathrm{e}^{-\frac{t}{2}} \sinh \left(\tfrac{\sqrt{1-4\left(\mu_k|\lambda|\right)^{\delta_1}}}{2} t\right) \left| \sqrt{1-4\left(\mu_k|\lambda|\right)^{\delta_1}}-\tfrac{1}{\sqrt{1-4\left(\mu_k|\lambda|\right)^{\delta_1}}} \right| \\
\lesssim & \mathrm{e}^{-C\left(\mu_k|\lambda|\right)^{\delta_1} t}\left(\mu_k|\lambda|\right)^{\delta_1}
\end{align*}
for $\mu_k|\lambda|<\varepsilon$, where $\varepsilon \leq 2^{-\left(\frac{\delta_1}{2}\right)^{-1}}$ is small enough, such that $1-4\left(\mu_k|\lambda|\right)^{\delta_1}>\frac{1}{2}$ for $\mu_k|\lambda|<\varepsilon$ and $C$ is a positive constant. Regarding the estimate for $\left|\partial_t G(t, \lambda, k)\right|$, it is enough to observe that for small frequencies $\mu_k|\lambda|<\varepsilon$:
\begin{align*}
\left|\partial_t G(t, \lambda, k)\right| = & \left|\mathrm{e}^{-\frac{1}{2} t} \cosh \left(\tfrac{\sqrt{1-4\left(\mu_k|\lambda|\right)^{\delta_1}}}{2} t\right)-\tfrac{\mathrm{e}^{-\frac{1}{2} t}}{\sqrt{1-4\left(\mu_k|\lambda|\right)^{\delta_1}}} \sinh \left(\tfrac{\sqrt{1-4\left(\mu_k|\lambda|\right)^{\delta_1}}}{2} t\right)\right| \\
\lesssim & \left|\left(1-\tfrac{1}{\sqrt{1-4\left(\mu_k|\lambda|\right)^{\delta_1}}}\right) \mathrm{e}^{\frac{t}{2}\left(-1+\sqrt{1-4\left(\mu_k|\lambda|\right)^{\delta_1}}\right)} \right. \\
& \left. + \left(1+\tfrac{1}{\sqrt{1-4\left(\mu_k|\lambda|\right)^{\delta_1}}}\right) \mathrm{e}^{\frac{t}{2}\left(-1-\sqrt{1-4\left(\mu_k|\lambda|\right)^{\delta_1}}\right)}\right| \lesssim \left(\mu_k|\lambda|\right)^{\delta_1} \mathrm{e}^{-C\left(\mu_k|\lambda|\right)^{\delta_1} t}+\mathrm{e}^{-C t},
\end{align*}
where $C$ is a positive constant. Thus, in the same manner as in the case $\delta_2 \in \left(0, \frac{\delta_1}{2}\right)$, we can get
\begin{align*}
J_1= & c_n \sum_{k, \ell \in \mathbb{N}^n} \int_{0<\mu_k|\lambda|<\varepsilon}\left|\partial_t \widehat{u}(t, \lambda)_{k, \ell}\right|^2|\lambda|^n \mathrm{d} \lambda \\
\lesssim & (1+t)^{-\frac{2 r_\mathrm{H}\left(u_0\right) +Q+4\delta_1}{2\delta_1}} \|u_0\|_{P, 0}^2 + (1+t)^{-\frac{2 r_\mathrm{H}\left(u_1\right) +Q+4\delta_1}{2\delta_1}} \|u_1\|_{P, 0}^2. 
\end{align*}
For $\varepsilon \leq \mu_k|\lambda| \leq 2^{-\left(\frac{\delta_1}{2}-\delta_2\right)^{-1}}$ and $\mu_k|\lambda|>2^{-\left(\frac{\delta_1}{2}-\delta_2\right)^{-1}}$, we proceed in exactly the same manner as in the case $\delta_2 \in \left(0, \frac{\delta_1}{2}\right)$. Therefore, the estimate 
\eqref{ineq:1.4} for $\|\partial_t u(t, \cdot)\|_{L^2\left(\mathbf{H}_n\right)}$ can also be established for the case $\delta_2=0$.

Estimating for $\|u(t, \cdot)\|_{\dot{H}^\alpha \left(\mathbf{H}_n\right)}$. The estimate \eqref{ineq:1.3} for $\|u(t, \cdot)\|_{\dot{H}^\alpha \left(\mathbf{H}_n\right)}$, with $\alpha \in [0, \delta_1]$ is handled in a similar way as in the case $\delta_2 \in \left(0, \frac{\delta_1}{2}\right)$.
\begin{itemize}
\item \textbf{Case 3:} $\mathbf{\delta_2=\frac{\delta_1}{2}}$.
\end{itemize}

Estimating for $\left\|\partial_t u(t, \cdot)\right\|_{L^2\left(\mathbf{H}_n\right)}$. It is easy to see that
\begin{equation} \label{ineq:4.32}
\left|\partial_t\left(F(t, \lambda, k)+\tfrac{1}{2} G(t, \lambda, k)\right)\right| \lesssim \left(\mu_k|\lambda|\right)^\frac{\delta_1}{2} \mathrm{e}^{-C\left(\mu_k|\lambda|\right)^{\frac{\delta_1}{2}} t}+\mathrm{e}^{-C\left(\mu_k|\lambda|\right)^{\frac{\delta_1}{2}} t}
\end{equation}
and
\begin{equation} \label{ineq:4.33}
\left|\partial_t\left(G(t, \lambda, k)\right)\right| \lesssim \mathrm{e}^{-C\left(\mu_k|\lambda|\right)^{\frac{\delta_1}{2}} t} \end{equation}
for $\lambda \in \mathbb{R}^*$ and where $C$ is a positive constant. By \eqref{eq:4.4}, \eqref{ineq:4.32} and \eqref{ineq:4.33}, we obtain
\begin{equation} \label{ineq:4.34}
\left|\partial_t \widehat{u}(t, \lambda)_{k, \ell}\right| \lesssim \left(\left(\mu_k|\lambda|\right)^\frac{\delta_1}{2} \mathrm{e}^{-C\left(\mu_k|\lambda|\right)^{\frac{\delta_1}{2}} t}+\mathrm{e}^{-C\left(\mu_k|\lambda|\right)^{\frac{\delta_1}{2}} t}\right)\left|\widehat{u}_0(\lambda)_{k, \ell}\right| + \mathrm{e}^{-C\left(\mu_k|\lambda|\right)^{\frac{\delta_1}{2}} t}\left|\widehat{u}_1(\lambda)_{k, \ell}\right|.
\end{equation}
By using Plancherel formula and $\left\{e_k\right\}_{k \in \mathbb{N}^n}$ as orthonormal basis of $L^2\left(\mathbb{R}^n\right)$, we have
\begin{align*}
\left\|\partial_t u(t, \cdot)\right\|_{L^2\left(\mathbf{H}_n\right)}^2 = & c_n \int_{\mathbb{R}^*}\left\|\partial_t \widehat{u}(t, \lambda)\right\|_{\mathrm{HS}\left[L^2\left(\mathbb{R}^n\right)\right]}^2|\lambda|^n \mathrm{d} \lambda \\
=& c_n \sum_{k, \ell \in \mathbb{N}^n}\left(\int_{0<\mu_k|\lambda|<\varepsilon}+\int_{\mu_k|\lambda|\geq\varepsilon}\right)\left|\partial_t \widehat{u}(t, \lambda)_{k, \ell}\right|^2|\lambda|^n \mathrm{d} \lambda
:= M_1+M_2,
\end{align*}
where $0<\varepsilon<1$. We estimate $M_1$ first. From \eqref{ineq:4.34}, by \eqref{ineq:gamma_beta}, Theorem \ref{theorem:3.7}, and proceeding similarly to $J_1$ in the case $\delta_2 \in\left(0, \frac{\delta_1}{2}\right)$, we obtain
\begin{align*}
M_1= & c_n \sum_{k, \ell \in \mathbb{N}^n} \int_{0<\mu_k|\lambda|<\varepsilon}\left|\partial_t \widehat{u}(t, \lambda)_{k, \ell}\right|^2|\lambda|^n \mathrm{d} \lambda \\
\lesssim & (1+t)^{-\frac{2 r_\mathrm{H}\left(u_0\right) +Q+2\delta_1}{\delta_1}} \|u_0\|_{P, 0}^2 + (1+t)^{-\frac{2 r_\mathrm{H}\left(u_1\right) +Q}{\delta_1}} \|u_1\|_{P, 0}^2, 
\end{align*}
For $M_2$, by \eqref{ineq:4.34}, we proceed in exactly the same way as $J_3$ in the case $\delta_2 \in\left(0, \frac{\delta_1}{2}\right)$. Therefore, the estimate \eqref{ineq:1.4} for $\|\partial_t u(t, \cdot)\|_{L^2\left(\mathbf{H}_n\right)}$ can also be established for the case $\delta_2=\frac{\delta_1}{2}$.

Estimating for $\|u(t, \cdot)\|_{\dot{H}^\alpha \left(\mathbf{H}_n\right)}$. By using Plancherel formula and $\left\{e_k\right\}_{k \in \mathbb{N}^n}$ as orthonormal basis of $L^2\left(\mathbb{R}^n\right)$, we have
\begin{align} \label{eq:4.35}
& \|u(t, \cdot)\|_{\dot{H}^\alpha \left(\mathbf{H}_n\right)}^2 =c_n \int_{\mathbb{R}^*}\|\left(-\sigma_{\Delta_{\mathrm{H}}}(\lambda)\right)^{\frac{\alpha}{2}}\widehat{u}(t, \lambda)\|_{\mathrm{HS}\left[L^2\left(\mathbb{R}^n\right)\right]}^2|\lambda|^n \mathrm{d} \lambda \notag\\
=& c_n \sum_{k, \ell \in \mathbb{N}^n} \int_{\mathbb{R}^*} \mu_k^\alpha |\lambda|^{\alpha +
n} \left(\widehat{u}(t, \lambda) e_k, e_{\ell}\right)_{L^2\left(\mathbb{R}^n\right)}^2 \mathrm{d} \lambda \notag\\
=&c_n \sum_{k, \ell \in \mathbb{N}^n}\left(\int_{0<\mu_k|\lambda|<\varepsilon}+\int_{\mu_k|\lambda|\geq\varepsilon} \right) \mu_k^\alpha |\lambda|^{\alpha +
n} \left|\widehat{u}(t, \lambda)_{k, \ell}\right|^2 \mathrm{d} \lambda
:= N_1+N_2,
\end{align}
where $0<\varepsilon<1$. We estimate $N_1$ first. For $\mu_k|\lambda| < \varepsilon$, we have,
\begin{align} \label{ineq:4.36}
& \left|F(t, \lambda, k)+\tfrac{\left(\mu_k|\lambda|\right)^{\delta_2}}{2} G(t, \lambda, k)\right| \notag\\
\lesssim & \mathrm{e}^{-\frac{\left(\mu_k|\lambda|\right)^{\frac{\delta_1}{2}}}{2} t} \left|\cos \left(\tfrac{\left(\mu_k|\lambda|\right)^{\frac{\delta_1}{2}} \sqrt{3}}{2} t\right)\right| + \mathrm{e}^{-\frac{\left(\mu_k|\lambda|\right)^{\frac{\delta_1}{2}}}{2} t} \left|\sin \left(\tfrac{\left(\mu_k|\lambda|\right)^{\frac{\delta_1}{2}} \sqrt{3}}{2} t\right)\right| \lesssim \mathrm{e}^{-C\left(\mu_k|\lambda|\right)^\frac{\delta_1}{2} t}
\end{align}
and
\begin{equation} \label{ineq:4.37}
\left|G(t, \lambda, k)\right| \lesssim \tfrac{\mathrm{e}^{-\frac{\left(\mu_k|\lambda|\right)^{\frac{\delta_1}{2}}}{2} t}}{\left(\mu_k|\lambda|\right)^{\frac{\delta_1}{2}}} \left|\sin \left(\tfrac{\left(\mu_k|\lambda|\right)^{\frac{\delta_1}{2}} \sqrt{3}}{2} t\right)\right| \lesssim t \mathrm{e}^{-C\left(\mu_k|\lambda|\right)^{\frac{\delta_1}{2}} t},
\end{equation}
where $C$ is a positive constant and we used the inequalities $\left|\sin y \right|\leq 1, \left|\cos y \right|\leq 1, \forall y \in \mathbb{R}$ and \eqref{ineq:4.16}. By \eqref{eq:4.4}, \eqref{ineq:4.36} and \eqref{ineq:4.37}, we obtain
\begin{equation} \label{ineq:4.38}
\left|\widehat{u}(t, \lambda)_{k, \ell}\right| \lesssim \mathrm{e}^{-C\left(\mu_k|\lambda|\right)^\frac{\delta_1}{2} t}\left|\widehat{u}_0(\lambda)_{k, \ell}\right| + t \mathrm{e}^{-C\left(\mu_k|\lambda|\right)^\frac{\delta_1}{2} t}\left|\widehat{u}_1(\lambda)_{k, \ell}\right|
\end{equation}
for any $\mu_k|\lambda|<\varepsilon$. From \eqref{ineq:4.38}, by \eqref{ineq:gamma_beta} and Theorem \ref{theorem:3.7}, it results
\begin{align} \label{ineq:4.39}
N_1= & c_n \sum_{k, \ell \in \mathbb{N}^n} \int_{0<\mu_k|\lambda|<\varepsilon}\left(\mu_k|\lambda|\right)^\alpha \left|\widehat{u}(t, \lambda)_{k, \ell}\right|^2|\lambda|^n \mathrm{d} \lambda \notag\\
\lesssim & c_n \sum_{k, \ell \in \mathbb{N}^n} \int_{0<\mu_k|\lambda|<\varepsilon}\left(\mu_k|\lambda|\right)^\alpha \mathrm{e}^{-2C\left(\mu_k|\lambda|\right)^\frac{\delta_1}{2} t}\left|\widehat{u}_0(\lambda)_{k, \ell}\right|^2|\lambda|^n \mathrm{d} \lambda \notag\\
& + (1+t) c_n \sum_{k, \ell \in \mathbb{N}^n} \int_{0<\mu_k|\lambda|<\varepsilon}\left(\mu_k|\lambda|\right)^\alpha \mathrm{e}^{-2C\left(\mu_k|\lambda|\right)^\frac{\delta_1}{2} t}\left|\widehat{u}_1(\lambda)_{k, \ell}\right|^2|\lambda|^n \mathrm{d} \lambda \notag\\
\lesssim & (1+t)^{-\frac{2\alpha}{\delta_1}} c_n \sum_{k, \ell \in \mathbb{N}^n} \int_{0<\mu_k|\lambda|<\varepsilon} \mathrm{e}^{-C\left(\mu_k|\lambda|\right)^\frac{\delta_1}{2} t}\left|\widehat{u}_0(\lambda)_{k, \ell}\right|^2 |\lambda|^{
n} \mathrm{d} \lambda \notag\\
& + (1+t)^{1-\frac{2\alpha}{\delta_1}} c_n \sum_{k, \ell \in \mathbb{N}^n} \int_{0<\mu_k|\lambda|<\varepsilon} \mathrm{e}^{-C\left(\mu_k|\lambda|\right)^\frac{\delta_1}{2} t} \left|\widehat{u}_1(\lambda)_{k, \ell}\right|^2 |\lambda|^n \mathrm{d} \lambda \notag\\
\lesssim & (1+t)^{-\frac{2r_\mathrm{H}\left(u_0\right)+Q+2\alpha}{\delta_1}} \|u_0\|_{P, 0}^2 + (1+t)^{-\frac{2r_\mathrm{H}\left(u_1\right)+Q+2\alpha-\delta_1}{\delta_1}} \|u_1\|_{P, 0}^2. 
\end{align}
Next, we estimate $N_2$. For $\mu_k|\lambda|\geq\varepsilon$, we see that
\begin{align} \label{ineq:4.40}
& \left|F(t, \lambda, k)+\tfrac{\left(\mu_k|\lambda|\right)^{\delta_2}}{2} G(t, \lambda, k)\right| \notag\\
\lesssim & \mathrm{e}^{-\frac{\left(\mu_k|\lambda|\right)^{\frac{\delta_1}{2}}}{2} t} \left|\cos \left(\tfrac{\left(\mu_k|\lambda|\right)^{\frac{\delta_1}{2}} \sqrt{3}}{2} t\right)\right| + \mathrm{e}^{-\frac{\left(\mu_k|\lambda|\right)^{\frac{\delta_1}{2}}}{2} t} \left|\sin \left(\tfrac{\left(\mu_k|\lambda|\right)^{\frac{\delta_1}{2}} \sqrt{3}}{2} t\right)\right| \lesssim \mathrm{e}^{-\frac{\left(\mu_k|\lambda|\right)^{\frac{\delta_1}{2}}}{2} t} \lesssim \mathrm{e}^{-C t}
\end{align}
and
\begin{align} \label{ineq:4.41}
\left(\mu_k|\lambda|\right)^\frac{\alpha}{2}\left|G(t, \lambda, k)\right| \lesssim & \tfrac{\mathrm{e}^{-\frac{\left(\mu_k|\lambda|\right)^{\frac{\delta_1}{2}}}{2} t}}{\left(\mu_k|\lambda|\right)^{\frac{\delta_1-\alpha}{2}}} \left|\sin \left(\tfrac{\left(\mu_k|\lambda|\right)^{\frac{\delta_1}{2}} \sqrt{3}}{2} t\right)\right| \lesssim \mathrm{e}^{-\frac{\left(\mu_k|\lambda|\right)^{\frac{\delta_1}{2}}}{2} t} \lesssim \mathrm{e}^{-C t}, \quad \forall \alpha \in \left[0, \delta_1\right],
\end{align}
where $C$ is a positive constant and we used the inequalities $\left|\sin y \right|\leq 1, \left|\cos y \right|\leq 1, \forall y \in \mathbb{R}$. Then, by \eqref{eq:4.4}, \eqref{ineq:4.40} and \eqref{ineq:4.41}, we obtain
$$
\left|\widehat{u}(t, \lambda)_{k, \ell}\right| \lesssim \mathrm{e}^{-C t}\left(\left|\widehat{u}_0(\lambda)_{k, \ell}\right|+\left(\mu_k|\lambda|\right)^{-\frac{\alpha}{2}} \left|\widehat{u}_1(\lambda)_{k, \ell}\right|\right)
$$
for any $\mu_k|\lambda|\geq\varepsilon$. Therefore,
\begin{align} \label{ineq:4.42}
N_2 \lesssim & \mathrm{e}^{-2C t} \sum_{k, \ell \in \mathbb{N}^n} \int_{\mu_k|\lambda|\geq\varepsilon} \mu_k^\alpha |\lambda|^{\alpha +
n} \left(\left|\widehat{u}_0(\lambda)_{k, \ell}\right|^2+\left(\mu_k|\lambda|\right)^{-\alpha}\left|\widehat{u}_1(\lambda)_{k, \ell}\right|^2\right) \mathrm{d}
 \lambda \notag\\
\lesssim & \mathrm{e}^{-2C t} \int_{\mathbb{R}^*} \sum_{k, \ell \in \mathbb{N}^n} \mu_k^\alpha |\lambda|^{\alpha +
n} \left(\left|\widehat{u}_0(\lambda)_{k, \ell}\right|^2+\left(\mu_k|\lambda|\right)^{-\alpha}\left|\widehat{u}_1(\lambda)_{k, \ell}\right|^2\right) \mathrm{d}
 \lambda \notag\\
=&\mathrm{e}^{-2C t} \int_{\mathbb{R}^*}\left(\|\left(-\sigma_{\Delta_{\mathrm{H}}}(\lambda)\right)^{\frac{\alpha}{2}}\widehat{u}_0(\lambda)\|_{\mathrm{HS}\left[L^2\left(\mathbb{R}^n\right)\right]}^2+\left\|\widehat{u}_1(\lambda)\right\|_{\mathrm{HS}\left[L^2\left(\mathbb{R}^n\right)\right]}^2\right)|\lambda|^n \mathrm{d}
 \lambda \notag\\
\approx & \mathrm{e}^{-2C t}\left(\left\|u_0\right\|_{\dot{H}^\alpha \left(\mathbf{H}_n\right)}^2+\left\|u_1\right\|_{L^2\left(\mathbf{H}_n\right)}^2\right),
\end{align}
where in the last step we applied Plancherel formula to $u_0$ and $u_1$. So, substituting \eqref{ineq:4.39} and \eqref{ineq:4.42} into \eqref{eq:4.35}, we obtain \eqref{ineq:1.3}. The proof for Theorem \ref{theorem:1.2} is now complete.
\end{proof}
\begin{remark}
\rm{
In Theorem \ref{theorem:1.2}, if $\left(u_0, u_1\right) \in \left(H^{\delta_1}\left(\mathbf{H}_n\right) \cap L^1\left(\mathbf{H}_n\right)\right) \times\left(L^2\left(\mathbf{H}_n\right) \cap L^1\left(\mathbf{H}_n\right)\right)$, then $P_{r_\mathrm{H}}\left(u_0\right)_{+}, P_{r_\mathrm{H}}\left(u_1\right)_{+}<\infty$ for $r_\mathrm{H}\left(u_0\right) = \delta_1$ and $r_\mathrm{H}\left(u_1\right) = 0$ (see Proposition \ref{proposition:3.2}). As a consequence, equation \eqref{eq:1.1} admits a solution $u$ that complies with the following decay estimates for $\alpha \in \left[0, \delta_1\right]$:
\begin{align*}
\left\|u(t, \cdot)\right\|_{\dot{H}^\alpha \left(\mathbf{H}_n\right)} \lesssim & (1+t)^{-\frac{Q-4\delta_2+2\alpha}{4\delta_1-4\delta_2}}\left(\|u_0\|_{H^\alpha\left(\mathbf{H}_n\right) \cap L^1\left(\mathbf{H}_n\right)}+\|u_1\|_{L^2\left(\mathbf{H}_n\right) \cap L^1\left(\mathbf{H}_n\right)}\right), \\
\left\|\partial_t u(t, \cdot)\right\|_{L^2\left(\mathbf{H}_n\right)} \lesssim & (1+t)^{-\frac{Q-4\delta_2}{4\delta_1-4\delta_2}-1}\left(\|u_0\|_{H^{\delta_1}\left(\mathbf{H}_n\right) \cap L^1\left(\mathbf{H}_n\right)}+\|u_1\|_{L^2\left(\mathbf{H}_n\right) \cap L^1\left(\mathbf{H}_n\right)}\right).
\end{align*}
Thus, our results precisely recover the linear decay rate bounds for the $L^2$-norm of the solution $u$ as described in \cite[Theorem 1.1]{Palmieri2020}, in the case when $\delta_1=1$ and $\delta_2=0$.
}
\end{remark}
\begin{remark}
\rm{
In Theorem \ref{theorem:1.2}, if $\left(u_0, u_1\right) \in\left(H^{\delta_1}\left(\mathbf{H}_n\right) \cap \dot{H}^{-\gamma}\left(\mathbf{H}_n\right)\right) \times\left(L^2\left(\mathbf{H}_n\right) \cap \dot{H}^{-\gamma}\left(\mathbf{H}_n\right)\right)$, with $\gamma \geq 0$, then $P_{r_\mathrm{H}}\left(u_0\right)_{+}, P_{r_\mathrm{H}}\left(u_1\right)_{+}<\infty$ for $r_\mathrm{H}\left(u_0\right) = \gamma-\frac{Q}{2}+\delta_1$ and $r_\mathrm{H}\left(u_1\right) = \gamma-\frac{Q}{2}$ (see Proposition \ref{proposition:3.3}). As a result, a solution $u$ to \eqref{eq:1.1} exists, which satisfies the decay estimates for $\alpha \in \left[0, \delta_1\right]$ as follows:
\begin{align*}
\left\|u(t, \cdot)\right\|_{\dot{H}^\alpha \left(\mathbf{H}_n\right)} \lesssim & (1+t)^{-\frac{\gamma-2\delta_2+\alpha}{2\delta_1-2\delta_2}}\left(\|u_0\|_{H^\alpha\left(\mathbf{H}_n\right) \cap \dot{H}^{-\gamma}\left(\mathbf{H}_n\right)}+\|u_1\|_{L^2\left(\mathbf{H}_n\right) \cap \dot{H}^{-\gamma}\left(\mathbf{H}_n\right)}\right),, \\
\left\|\partial_t u(t,\cdot)\right\|_{L^2\left(\mathbf{H}_n\right)} \lesssim & (1+t)^{-\frac{\gamma-2\delta_2}{2\delta_1-2\delta_2}-1}\left(\|u_0\|_{H^{\delta_1}\left(\mathbf{H}_n\right) \cap \dot{H}^{-\gamma}\left(\mathbf{H}_n\right)}+\|u_1\|_{L^2\left(\mathbf{H}_n\right) \cap \dot{H}^{-\gamma}\left(\mathbf{H}_n\right)}\right).
\end{align*}
Thus, we have successfully recovered the estimate for $\left\|u(t, \cdot)\right\|_{\dot{H}^\alpha \left(\mathbf{H}_n\right)}$ established by Dasgupta, Kumar, Mondal, and Ruzhansky in \cite[Theorem 1.1]{Dasgupta2024}, in the case when $\delta_1=1$ and $\delta_2=0$. Additionally, we have provided an estimate for the decay rate (in time) of $\left\|\partial_t u(t, \cdot)\right\|_{L^2\left(\mathbf{H}_n\right)}$.
}
\end{remark}
\begin{remark}
\rm{
In Theorem \ref{theorem:1.2}, if $\left(u_0, u_1\right) \in \left(H^{\delta_1}\left(\mathbf{H}_n\right) \cap L^m\left(\mathbf{H}_n\right)\right) \times \left(L^2\left(\mathbf{H}_n\right) \cap L^m\left(\mathbf{H}_n\right)\right)$, with $m \in (1,2]$, then $P_{r_\mathrm{H}}\left(u_0\right)_{+}, P_{r_\mathrm{H}}\left(u_1\right)_{+}<\infty$ for $r_\mathrm{H}\left(u_0\right) = -Q\left(1-\frac{1}{m}\right)+\delta_1$ and $r_\mathrm{H}\left(u_1\right) = -Q\left(1-\frac{1}{m}\right)$ (see Proposition \ref{proposition:3.4}). Consequently, there exists a solution $u$ to \eqref{eq:1.1} that satisfies the following decay estimates for $\alpha \in \left[0, \delta_1\right]$:
\begin{align*}
\left\|u(t, \cdot)\right\|_{\dot{H}^\alpha \left(\mathbf{H}_n\right)} \lesssim & (1+t)^{-\frac{Q}{2\delta_1-2\delta_2}\left(\frac{1}{m}-\frac{1}{2}\right)-\frac{\alpha-2\delta_2}{2\delta_1-2\delta_2}}\left(\|u_0\|_{H^\alpha\left(\mathbf{H}_n\right) \cap L^m\left(\mathbf{H}_n\right)}+\|u_1\|_{L^2\left(\mathbf{H}_n\right) \cap L^m\left(\mathbf{H}_n\right)}\right), \\
\left\|\partial_t u(t, \cdot)\right\|_{L^2\left(\mathbf{H}_n\right)} \lesssim & (1+t)^{-\frac{Q}{2\delta_1-2\delta_2}\left(\frac{1}{m}-\frac{1}{2}\right)+\frac{\delta_2}{\delta_1-\delta_2}-1}\left(\|u_0\|_{H^{\delta_1}\left(\mathbf{H}_n\right) \cap L^m\left(\mathbf{H}_n\right)}+\|u_1\|_{L^2\left(\mathbf{H}_n\right) \cap L^m\left(\mathbf{H}_n\right)}\right).
\end{align*}
}
\end{remark}
\begin{remark}
\rm{
In Theorem \ref{theorem:1.2}, if $\left(u_0, u_1\right) \in H_{\sigma \psi(t, \cdot)}^1\left(\mathbf{H}_n\right) \times L_{\sigma \psi(t,)}^2\left(\mathbf{H}_n\right)$, with $
\psi(t, \eta) := \frac{|x|^2+|y|^2+4|\tau|}{8(1+t)}$ for any $\sigma>0, t \geq 0$ and $\eta=(x, y, \tau) \in \mathbf{H}_n$, then $P_{r_\mathrm{H}}\left(u_0\right)_{+}, P_{r_\mathrm{H}}\left(u_1\right)_{+}<\infty$ for $r_\mathrm{H}\left(u_0\right) = \delta_1$ and $r_\mathrm{H}\left(u_1\right) = 0$ (see Proposition \ref{proposition:3.5}). Under these circumstances, there exists a solution $u$ to \eqref{eq:1.1} that satisfies the following decay estimates for $\alpha \in \left[0, \delta_1\right]$:
\begin{align*}
\left\|u(t, \cdot)\right\|_{\dot{H}^\alpha \left(\mathbf{H}_n\right)} \lesssim & (1+t)^{-\frac{Q-4\delta_2+2\alpha}{4\delta_1-4\delta_2}}\left(\|u_0\|_{H_{\sigma \psi(t, \cdot)}^\alpha\left(\mathbf{H}_n\right)}+\|u_1\|_{L_{\sigma \psi(t,)}^2\left(\mathbf{H}_n\right)}\right), \\
\left\|\partial_t u(t, \cdot)\right\|_{L^2\left(\mathbf{H}_n\right)} \lesssim & (1+t)^{-\frac{Q-4\delta_2}{4\delta_1-4\delta_2}-1}\left(\|u_0\|_{H_{\sigma \psi(t, \cdot)}^1\left(\mathbf{H}_n\right)}+\|u_1\|_{L_{\sigma \psi(t,)}^2\left(\mathbf{H}_n\right)}\right).
\end{align*}
}
\end{remark}
\begin{remark}
\rm{
In Theorem \ref{theorem:1.2}, if $\left(u_0, u_1\right) \in\left(H^{\delta_1}\left(\mathbf{H}_n\right) \cap \mathcal{Y}^q\left(\mathbf{H}_n\right)\right) \times\left(L^2\left(\mathbf{H}_n\right) \cap \mathcal{Y}^q\left(\mathbf{H}_n\right)\right)$, with $q \leq \frac{Q}{2}$, then $P_{r_\mathrm{H}}\left(u_0\right)_{+}, P_{r_\mathrm{H}}\left(u_1\right)_{+}<\infty$ for $r_\mathrm{H}\left(u_0\right) = -q+\delta_1$ and $r_\mathrm{H}\left(u_1\right) = -q$ (see Proposition \ref{proposition:3.6}). Under these conditions, the problem \eqref{eq:1.1} admits a solution $u$ that satisfies the following decay estimates for $\alpha \in \left[0, \delta_1\right]$:
\begin{align*}
\left\|u(t, \cdot)\right\|_{\dot{H}^\alpha \left(\mathbf{H}_n\right)} \lesssim & (1+t)^{-\frac{Q-2q-4\delta_2+2\alpha}{4\delta_1-4\delta_2}}\left(\|u_0\|_{H^\alpha\left(\mathbf{H}_n\right) \cap \mathcal{Y}^q\left(\mathbf{H}_n\right)}+\|u_1\|_{L^2\left(\mathbf{H}_n\right) \cap \mathcal{Y}^q\left(\mathbf{H}_n\right)}\right), \\
\left\|\partial_t u(t, \cdot)\right\|_{L^2\left(\mathbf{H}_n\right)} \lesssim & (1+t)^{-\frac{Q-2q-4\delta_2}{4\delta_1-4\delta_2}-1}\left(\|u_0\|_{H^{\delta_1}\left(\mathbf{H}_n\right) \cap \mathcal{Y}^q\left(\mathbf{H}_n\right)}+\|u_1\|_{L^2\left(\mathbf{H}_n\right) \cap \mathcal{Y}^q\left(\mathbf{H}_n\right)}\right).
\end{align*}
}
\end{remark}
\begin{remark}
\rm{
In Theorem \ref{theorem:1.2}, if $\left(u_0, u_1\right) \in\left(H^{\delta_1}\left(\mathbf{H}_n\right) \cap \dot{H}_m^{-\gamma}\left(\mathbf{H}_n\right)\right) \times\left(L^2\left(\mathbf{H}_n\right) \cap \dot{H}_m^{-\gamma}\left(\mathbf{H}_n\right)\right)$, with $m \in (1, 2], \gamma \geq 0$, then $P_{r_\mathrm{H}}\left(u_0\right)_{+}, P_{r_\mathrm{H}}\left(u_1\right)_{+}<\infty$ for $r_\mathrm{H}\left(u_0\right) = \gamma-\frac{Q(m-1)}{m}+\delta_1$ and $r_\mathrm{H}\left(u_1\right) = \gamma-\frac{Q(m-1)}{m}$ (see Proposition \ref{proposition:3.7}). Under these conditions, the problem \eqref{eq:1.1} admits a solution $u$ that satisfies the following decay estimates for $\alpha \in \left[0, \delta_1\right]$:
\begin{align*}
\left\|u(t, \cdot)\right\|_{\dot{H}^\alpha \left(\mathbf{H}_n\right)}
\lesssim &(1+t)^{-\frac{Q}{2\delta_1-2\delta_2}\left(\frac{1}{m}-\frac{1}{2}\right)-\frac{\gamma-2\delta_2+\alpha}{2\delta_1-2\delta_2}}\left(\|u_0\|_{H^\alpha\left(\mathbf{H}_n\right) \cap \dot{H}_m^{-\gamma}\left(\mathbf{H}_n\right)}+\|u_1\|_{L^2\left(\mathbf{H}_n\right) \cap \dot{H}_m^{-\gamma}\left(\mathbf{H}_n\right)}\right), \\
\left\|\partial_t u(t, \cdot)\right\|_{L^2\left(\mathbf{H}_n\right)} \lesssim & (1+t)^{-\frac{Q}{2\delta_1-2\delta_2}\left(\frac{1}{m}-\frac{1}{2}\right)-\frac{\gamma-2\delta_2}{2\delta_1-2\delta_2}-1}\left(\|u_0\|_{H^{\delta_1}\left(\mathbf{H}_n\right) \cap \dot{H}_m^{-\gamma}\left(\mathbf{H}_n\right)}+\|u_1\|_{L^2\left(\mathbf{H}_n\right) \cap \dot{H}_m^{-\gamma}\left(\mathbf{H}_n\right)}\right).
\end{align*}
}
\end{remark}
Next, we consider the following linear Cauchy problem with $\delta_1 \in \left[0, \frac{\delta_2}{2}\right]$:
\begin{equation} \label{eq:lemma:5.1}
\begin{cases}
\partial_t^2 v(t, \eta)+\left(-\Delta_{\mathrm{H}}\right)^{\delta_1} v(t, \eta)+\left(-\Delta_{\mathrm{H}}\right)^{\delta_2} \partial_t v(t, \eta)=0, & (t, \eta) \in[s, \infty) \times \mathbf{H}_n, \, s \geq 0, \, n \in \mathbb{N}^*, \\ v(s, \eta)=0, \quad v_t(s, \eta)=g(s, \eta), & \eta \in \mathbf{H}_n, \, s \geq 0, \, n \in \mathbb{N}^*.
\end{cases}
\end{equation}
The following result is applied in the proofs of Theorems \ref{theorem:1.3} and \ref{theorem:1.4}:
\begin{proposition} \label{proposition:7.2}
Let $Q=2n+2$ be the homogeneous dimension on the Heisenberg group and $g=g(s, \eta) \in L^2\left(\mathbf{H}_n\right) \cap L^\omega \left(\mathbf{H}_n\right)$, with $\omega \in [1, 2]$. Then the Sobolev solution $v=v(t, \eta)$ to \eqref{eq:lemma:5.1} satisfies the following estimates for all $\alpha \in \left[0, \delta_1\right]$:
\begin{align*}
\left\|v(t, \cdot)\right\|_{\dot{H}^\alpha \left(\mathbf{H}_n\right)} \lesssim & (1+t-s)^{-\frac{Q}{2\delta_1-2\delta_2}\left(\frac{1}{\omega}-\frac{1}{2}\right)-\frac{\alpha}{2\delta_1-2\delta_2}+\frac{\delta_2}{\delta_1-\delta_2}}\left(\|g(s, \cdot)\|_{L^2\left(\mathbf{H}_n\right)}+\|g(s, \cdot)\|_{L^\omega\left(\mathbf{H}_n\right)}\right), \\
\left\|\partial_t v(t, \cdot)\right\|_{L^2\left(\mathbf{H}_n\right)} \lesssim & (1+t-s)^{-\frac{Q}{2\delta_1-2\delta_2}\left(\frac{1}{\omega}-\frac{1}{2}\right)+\frac{\delta_2}{\delta_1-\delta_2}-1}\left(\|g(s, \cdot)\|_{L^2\left(\mathbf{H}_n\right)}+\|g(s, \cdot)\|_{L^\omega\left(\mathbf{H}_n\right)}\right).
\end{align*}
\end{proposition}
\begin{proof}[Proof of Proposition \ref{proposition:7.2}]
The proof of the above theorem is entirely analogous to that of Theorem \ref{theorem:1.2}. As a result, the decay rates for equation \eqref{eq:lemma:5.1} are obtained as follows:
\begin{align}
\label{decay_rate_v} \left\|v(t, \cdot)\right\|_{\dot{H}^\alpha \left(\mathbf{H}_n\right)} \lesssim & (1+t-s)^{-\frac{2 \min \left\{ r_\mathrm{H}\left(u_0\right), r_\mathrm{H}\left(u_1\right)-2\delta_2 \right\} +Q+2\alpha}{4\delta_1-4\delta_2}} \|g(s, \cdot)\|_{P, 0}, \\
\label{decay_rate_v_t} \left\|v_t(t, \cdot)\right\|_{L^2\left(\mathbf{H}_n\right)} \lesssim & (1+t-s)^{-\frac{2 \min \left\{ r_\mathrm{H}\left(u_0\right), r_\mathrm{H}\left(u_1\right)-2\delta_2 \right\} +Q+4\delta_1-4\delta_2}{4\delta_1-4\delta_2}} \|g(s, \cdot)\|_{P, 0}.
\end{align}
With $g(s, \eta) \in L^2\left(\mathbf{H}_n\right) \cap L^\omega \left(\mathbf{H}_n\right), \omega \in [1, 2)$, by using \eqref{decay_rate_v} and \eqref{decay_rate_v_t}, together with Examples \ref{proposition:3.2} and \ref{proposition:3.4}, we obtain the following estimates:
\begin{align*}
\left\|v(t, \cdot)\right\|_{\dot{H}^\alpha \left(\mathbf{H}_n\right)} \lesssim & (1+t-s)^{-\frac{Q}{2\delta_1-2\delta_2}\left(\frac{1}{\omega}-\frac{1}{2}\right)-\frac{\alpha}{2\delta_1-2\delta_2}+\frac{\delta_2}{\delta_1-\delta_2}}\left(\|g(s, \cdot)\|_{L^2\left(\mathbf{H}_n\right)}+\|g(s, \cdot)\|_{L^\omega\left(\mathbf{H}_n\right)}\right), \\
\left\|\partial_t v(t, \cdot)\right\|_{L^2\left(\mathbf{H}_n\right)} \lesssim & (1+t-s)^{-\frac{Q}{2\delta_1-2\delta_2}\left(\frac{1}{\omega}-\frac{1}{2}\right)+\frac{\delta_2}{\delta_1-\delta_2}-1}\left(\|g(s, \cdot)\|_{L^2\left(\mathbf{H}_n\right)}+\|g(s, \cdot)\|_{L^\omega\left(\mathbf{H}_n\right)}\right).
\end{align*}
This completes the proof.
\end{proof}

\section{global (in time) existence of solutions} \label{section_5}
We obtain the next results on the global (in time) existence of small data solutions.
\begin{theorem} \label{theorem:1.3}
Let $Q=2n+2$ be the homogeneous dimension on the Heisenberg group. Assume that $\left(u_0, u_1\right) \in D^{\delta_1, 0}_\mathrm{H} := H^{\delta_1}\left(\mathbf{H}_n\right) \times L^2\left(\mathbf{H}_n\right)$ such that $P_{r_\mathrm{H}}\left(u_0\right)_{+}, P_{r_\mathrm{H}}\left(u_1\right)_{+}<\infty$ for some $r_\mathrm{H}\left(u_0\right), r_\mathrm{H}\left(u_1\right)$ satisfying $-\frac{Q}{2} \leq \min \left\{ r_\mathrm{H}\left(u_0\right), r_\mathrm{H}\left(u_1\right)-2\delta_2 \right\} \leq -2\delta_2$. Suppose that $p$ satisfies the following conditions:
\begin{equation} \label{p_*}
p>p_{\mathrm{Fuji}}\left(\tfrac{Q-2 \omega \delta_2}{\omega \delta_1}\right):= 1+\tfrac{2 \omega \delta_1}{Q-2 \omega \delta_2}, \quad \text{ with } \omega:=\tfrac{Q}{Q+\min \left\{ r_\mathrm{H}^*\left(u_0\right), r_\mathrm{H}^*\left(u_1\right)-2\delta_2 \right\}+2\delta_2}
\end{equation}
and
\begin{equation} \label{condition_omega_p}
p \in \left[\tfrac{2}{\omega}, \infty\right) \text { if } Q \leq 2 \alpha, \text{ or } p \in\left[\tfrac{2}{\omega}, \tfrac{Q}{Q-2 \alpha}\right] \text { if } 2 \alpha <Q \leq \tfrac{4 \alpha}{2-\omega}, \text{ with } \alpha \in (0, \delta_1].
\end{equation}
Then there exists a constant $\varepsilon>0$ such that if the initial data $\left(u_0, u_1\right) \in D^{\delta_1, 0}_\mathrm{H} := H^{\delta_1}\left(\mathbf{H}_n\right) \times L^2\left(\mathbf{H}_n\right)$ with the norm $\left\|\left(u_0, u_1\right)\right\|_{D^{\delta_1, 0}_\mathrm{H}} := \left\|u_0\right\|_{P, \delta_1}+\left\|u_1\right\|_{P, 0}<\varepsilon$, the problem \eqref{eq:1.0} admits a unique global (in time) solution $u \in \mathcal{C}\left([0, \infty), H^{\delta_1}\left(\mathbf{H}_n\right) \right) \cap \mathcal{C}^1 \left([0, \infty), L^2\left(\mathbf{H}_n\right) \right)$. Moreover, we have the following decay estimates for the solution and its derivative
\begin{align} \label{ineq:theorem:1.3.1}
\|u(t, \cdot)\|_{L^2\left(\mathbf{H}_n\right)} \lesssim & \left\|\left(u_0, u_1\right)\right\|_{D^{\delta_1, 0}_\mathrm{H}} (1+t)^{-\frac{2 \min \left\{ r_\mathrm{H}\left(u_0\right), r_\mathrm{H}\left(u_1\right)-2\delta_2 \right\} +Q}{4\delta_1-4\delta_2}}, \\
\label{ineq:theorem:1.3.2} \left\| u(t, \cdot)\right\|_{\dot{H}^\alpha \left(\mathbf{H}_n\right)}\lesssim & \left\|\left(u_0, u_1\right)\right\|_{D^{\delta_1, 0}_\mathrm{H}} (1+t)^{-\frac{2 \min \left\{ r_\mathrm{H}\left(u_0\right), r_\mathrm{H}\left(u_1\right)-2\delta_2 \right\} +Q+2\alpha}{4\delta_1-4\delta_2}}, \\
\label{ineq:theorem:1.3.3} \left\|\partial_t u(t, \cdot)\right\|_{L^2\left(\mathbf{H}_n\right)} \lesssim & \left\|\left(u_0, u_1\right)\right\|_{D^{\delta_1, 0}_\mathrm{H}} \log(\mathrm{e}+t) (1+t)^{-\frac{2 \min \left\{ r_\mathrm{H}\left(u_0\right), r_\mathrm{H}\left(u_1\right)-2\delta_2 \right\} +Q+4\delta_1-4\delta_2}{4\delta_1-4\delta_2}}.
\end{align}
\end{theorem}
Note that if $\min \left\{ r_\mathrm{H}\left(u_0\right), r_\mathrm{H}\left(u_1\right)-2\delta_2 \right\} = -2\delta_2$ in Theorem \ref{theorem:1.3}, then $\omega=1$, which implies that the exponent $p_{\mathrm{Fuji}}\left(\frac{Q-2 \omega \delta_2}{\omega \delta_1}\right):=1+\frac{2 \omega \delta_1}{Q-2 \omega \delta_2}$ does not satisfy the condition \eqref{condition_omega_p}. 

An example where $\min \left\{ r_\mathrm{H}\left(u_0\right), r_\mathrm{H}\left(u_1\right)-2\delta_2 \right\} = -2\delta_2$ can be observed when the initial data $(u_0, u_1) \in \left(H^{\delta_1}\left(\mathbf{H}_n\right) \cap L^1\left(\mathbf{H}_n\right)\right) \times \left(L^2\left(\mathbf{H}_n\right) \cap L^1\left(\mathbf{H}_n\right)\right)$ and let $r_\mathrm{H}\left(u_0\right) = \delta_1, r_\mathrm{H}\left(u_1\right) = 0$ such that $P_{r_\mathrm{H}}\left(u_0\right)_{+}, P_{r_\mathrm{H}}\left(u_1\right)_{+}<\infty$ , as discussed in Proposition \ref{proposition:3.2}.

However, if $\min \left\{ r_\mathrm{H}\left(u_0\right), r_\mathrm{H}\left(u_1\right)-2\delta_2 \right\} < -2\delta_2$, we do not yet have sufficient grounds to determine whether the exponent $p_{\mathrm{Fuji}}\left(\frac{Q-2 \omega \delta_2}{\omega \delta_1}\right):=1+\frac{2 \omega \delta_1}{Q-2 \omega \delta_2}$ satisfies the condition of Theorem \ref{theorem:1.3}. This suggests that, for $p_{\mathrm{Fuji}}\left(\frac{Q-2 \omega \delta_2}{\omega \delta_1}\right):=1+\frac{2 \omega \delta_1}{Q-2 \omega \delta_2}$, when $\min \left\{ r_\mathrm{H}\left(u_0\right), r_\mathrm{H}\left(u_1\right)-2\delta_2 \right\} = -2\delta_2$, the Cauchy problem \eqref{eq:1.0} may not have a global solution (in time), while for $\min \left\{ r_\mathrm{H}\left(u_0\right), r_\mathrm{H}\left(u_1\right)-2\delta_2 \right\} < -2\delta_2$, the Cauchy problem \eqref{eq:1.0} might have a global (in time) solution.
\begin{theorem} \label{theorem:1.4}
Assume that $u_0, u_1$ satisfy the hypotheses in Theorem \ref{theorem:1.3} together with the condition $-\frac{Q}{2} \leq \min \left\{ r_\mathrm{H}\left(u_0\right), r_\mathrm{H}\left(u_1\right)-2\delta_2 \right\} < -2\delta_2$. Suppose that $p$ satisfies the following condition:
\begin{equation} \label{p_*_1}
p:=p_{\mathrm{Fuji}}\left(\tfrac{Q-2 \omega \delta_2}{\omega \delta_1}\right):=1+\tfrac{2 \omega \delta_1}{Q-2 \omega \delta_2}, \quad \text{ with } \omega:=\tfrac{Q}{Q+\min \left\{ r_\mathrm{H}\left(u_0\right), r_\mathrm{H}\left(u_1\right)-2\delta_2 \right\}+2\delta_2}
\end{equation}
and
\begin{equation} \label{condition_omega_p_crit}
p \in \left(\tfrac{2}{\omega}, \infty\right) \text { if } Q \leq 2 \alpha, \text{ or } p \in\left(\tfrac{2}{\omega}, \tfrac{Q}{Q-2 \alpha}\right] \text { if } 2 \alpha <Q \leq \tfrac{4 \alpha}{2-\omega}, \text{ with } \alpha \in (0, \delta_1].
\end{equation}
Then there exists a constant $\varepsilon>0$ such that if the initial data $(u_0, u_1) \in D^{\delta_1, 0}_\mathrm{H}:=H^{\delta_1}\left(\mathbf{H}_n\right) \times L^2\left(\mathbf{H}_n\right)$ with the norm $\left\|(u_0, u_1)\right\|_{D^{\delta_1, 0}_\mathrm{H}}:=\left\|u_0\right\|_{P, \delta_1}+\left\|u_1\right\|_{P, 0}<\varepsilon$, the problem \eqref{eq:1.0} admits a unique global (in time) solution $u \in \mathcal{C}\left([0, \infty), H^{\delta_1}\left(\mathbf{H}_n\right) \right) \cap \mathcal{C}^1 \left([0, \infty), L^2\left(\mathbf{H}_n\right) \right)$. Moreover, we have the decay estimates for the solution and its time derivative as in \eqref{ineq:theorem:1.3.1}, \eqref{ineq:theorem:1.3.2} and \eqref{ineq:theorem:1.3.3}.
\end{theorem}
\begin{remark}
\rm{
In Theorems \ref{theorem:1.3} and \ref{theorem:1.4} above, we have presented a method that allows one to compute the exact value of $p_{\mathrm{Fuji}}\left(\frac{Q-2 \omega \delta_2}{\omega \delta_1}\right):=1+\frac{2 \omega \delta_1}{Q-2 \omega \delta_2}$ when the initial data are restricted to specific function spaces. If in Theorems \ref{theorem:1.3} and \ref{theorem:1.4}, we do not use the decay rates of solutions with respect to $r_\mathrm{H}\left(u_0\right)$ and $r_\mathrm{H}\left(u_1\right)$ as in Theorem \ref{theorem:1.2}, but instead use the decay rates with respect to $r_\mathrm{H}^*\left(u_0\right)$ and $r_\mathrm{H}^*\left(u_1\right)$, then we obtain the critical exponent $p_{\mathrm{Fuji}}\left(\frac{Q-2 \omega \delta_2}{\omega \delta_1}\right):=1+\frac{2 \omega \delta_1}{Q-2 \omega \delta_2}$ with $\omega:=\frac{Q}{Q+\min \left\{ r_\mathrm{H}^*\left(u_0\right), r_\mathrm{H}^*\left(u_1\right)-2\delta_2 \right\}+2\delta_2}$. Note that in this case, $\omega$ depends on $r_\mathrm{H}^*\left(u_0\right)$ and $r_\mathrm{H}^*\left(u_1\right)$. Consequently, computing the exact value of the exponent $p_{\mathrm{Fuji}}\left(\frac{Q-2 \omega \delta_2}{\omega \delta_1}\right):=1+\frac{2 \omega \delta_1}{Q-2 \omega \delta_2}$ becomes very difficult when the initial data are restricted to different function spaces, since only lower bounds for $r_\mathrm{H}^*\left(u_0\right)$ and $r_\mathrm{H}^*\left(u_1\right)$ are available (see Section \ref{section_3.4}).
} 
\end{remark}
\subsection*{Outlines of the proofs of Theorems \ref{theorem:1.3} and \ref{theorem:1.4}}
Since $E_0(t, 0, \eta)$ and $E_1(t, 0, \eta)$, represent the fundamental solutions of \eqref{eq:1.1}, the function $u^{\operatorname{lin}} := u_0(\eta)*_{(\eta)}E_0(t, 0, \eta)+u_1(\eta)*_{(\eta)}E_1(t, 0, \eta)$ solves problem
\eqref{eq:1.1}. Given $T>0$, we introduce the operator $N: u \in X_\mathrm{H}(T) \rightarrow N u=N u(t, \eta) := u^{\mathrm{lin}}(t, \eta)+u^{\mathrm{non}}(t, \eta)$, where $X_\mathrm{H}(T)$ denotes the evolution space to be defined and $u^{\mathrm{non}}(t, \eta)$ is defined through the integral operator: $u^{\mathrm{non}}(t, \eta) := \int_0^t |u(s, \eta)|^p *_{(\eta)} E_1(t, s, \eta) \mathrm{d} s$. The global (in time) solution to the semilinear problem \eqref{eq:1.0} is identified as a fixed point of the operator $N$. To establish the existence of such a fixed point, we will verify that the mapping $N$ satisfies the following estimates:
\begin{align} \label{ineq:5.1}
\|N u\|_{X_\mathrm{H}(T)} \lesssim & \left\|\left(u_0, u_1\right)\right\|_{D^{\delta_1, 0}_\mathrm{H}}+\|u\|_{X_\mathrm{H}(T)}^p, \\
\label{ineq:5.2}
\|N u-N v\|_{X_\mathrm{H}(T)} \lesssim & \|u-v\|_{X_\mathrm{H}(T)}\left(\|u\|_{X_\mathrm{H}(T)}^{p-1}+\|v\|_{X_\mathrm{H}(T)}^{p-1}\right),
\end{align}
where $D^{\delta_1, 0}_\mathrm{H}$  is specified in the statements of Theorems \ref{theorem:1.3} and \ref{theorem:1.4}. Assuming $\left\|\left(u_0, u_1\right)\right\|_{D^{\delta_1, 0}_\mathrm{H}}:=\varepsilon$ is sufficiently small, estimates \eqref{ineq:5.1} and \eqref{ineq:5.2} yield, via a standard argument, the existence of a unique local (in time) solution for large data and a unique global (in time) solution for small data in $X_\mathrm{H}(T)$.
\begin{proof}[Proof of Theorem \ref{theorem:1.3}]
We define the solutions space $X_\mathrm{H}(T)$ by
$$
X_\mathrm{H}(T) := \mathcal{C}\left([0, T], H^{\delta_1}\left(\mathbf{H}_n\right)\right) \cap \mathcal{C}^1 \left([0, T], L^2\left(\mathbf{H}_n\right) \right),
$$
with its corresponding norm
\begin{align*}
\|u\|_{X_\mathrm{H}(T)} := \sup _{0 \leq t \leq T}&\left[(1+t)^{\frac{2 \min \left\{ r_\mathrm{H}\left(u_0\right), r_\mathrm{H}\left(u_1\right)-2\delta_2 \right\} +Q}{4\delta_1-4\delta_2}} \|u(t, \cdot)\|_{L^2\left(\mathbf{H}_n\right)} \right.\\
&\left.
+(1+t)^{\frac{2 \min \left\{ r_\mathrm{H}\left(u_0\right), r_\mathrm{H}\left(u_1\right)-2\delta_2 \right\} +Q+2\alpha}{4\delta_1-4\delta_2}} \left\| u(t, \cdot)\right\|_{\dot{H}^\alpha \left(\mathbf{H}_n\right)} \right.\\
&\left.
+\log(\mathrm{e}+t)^{-1}(1+t)^{\frac{2 \min \left\{ r_\mathrm{H}\left(u_0\right), r_\mathrm{H}\left(u_1\right)-2\delta_2 \right\} +Q+4\delta_1-4\delta_2}{4\delta_1-4\delta_2}} \|\partial_t u(t, \cdot)\|_{L^2\left(\mathbf{H}_n\right)} \right].
\end{align*}
With $\omega:=\frac{Q}{Q+\min \left\{ r_\mathrm{H}\left(u_0\right), r_\mathrm{H}\left(u_1\right)-2\delta_2 \right\}+2\delta_2}$, it is easy to see that $1 \leq \omega < 2$. The application of Gagliardo-Nirenberg inequality with interpolation exponents $\theta_1(2 p) := \frac{Q(p-1)}{2 \alpha p}, \, \theta_2(\omega p) := \frac{Q(\omega p-2)}{2 \omega \alpha p}$ from Lemma \ref{Gagliardo} and the definition of the evolution space leads to
\begin{align} \label{ineq:5.3}
\left\| \left| u(\tau, \cdot)\right|^p\right\|_{L^2\left(\mathbf{H}_n\right)}=\left\| u(\tau, \cdot)\right\|_{L^{2 p}\left(\mathbf{H}_n\right)}^p \lesssim & (1+\tau)^{-\frac{Q +\min \left\{r_\mathrm{H}\left(u_0\right), r_\mathrm{H}\left(u_1\right)-2\delta_2 \right\}}{2\delta_1-2\delta_2}p+\frac{Q}{4\delta_1-4\delta_2}}\|u\|_{X_\mathrm{H}(T)}^p, \\
\label{ineq:5.4}
\left\| \left| u(\tau, \cdot)\right|^p\right\|_{L^\omega\left(\mathbf{H}_n\right) \cap L^2\left(\mathbf{H}_n\right)}=&\left\| u(\tau, \cdot)\right\|_{L^{\omega p}\left(\mathbf{H}_n\right)}^p + \left\| u(\tau, \cdot)\right\|_{L^{2 p}\left(\mathbf{H}_n\right)}^p \notag \\
\lesssim &(1+\tau)^{-\frac{Q +\min \left\{r_\mathrm{H}\left(u_0\right), r_\mathrm{H}\left(u_1\right)-2\delta_2 \right\}}{2\delta_1-2\delta_2}p+\frac{Q}{2 \omega\left(\delta_1-\delta_2\right)}}\|u\|_{X_\mathrm{H}(T)}^p,
\end{align}
provided that $p \in\left[\frac{2}{\omega}, \infty\right) \text { if } Q \leq 2 \alpha, \text { or } p \in\left[\frac{2}{\omega}, \frac{Q}{Q-2 \alpha}\right] \text { if } 2\alpha < Q \leq \frac{4 \alpha}{2-\omega}$. First, we will prove the inequality \eqref{ineq:5.1}. The estimates for solutions to \eqref{eq:1.1}, from Theorem \ref{theorem:1.2}, imply
\begin{align*}
\left\|u^{\mathrm{lin}}(t, \cdot)\right\|_{L^2\left(\mathbf{H}_n\right)} \lesssim &(1+t)^{-\frac{2 \min \left\{ r_\mathrm{H}\left(u_0\right), r_\mathrm{H}\left(u_1\right)-2\delta_2 \right\} +Q}{4\delta_1-4\delta_2}}\left\|\left(u_0, u_1\right)\right\|_{D^{\delta_1, 0}_\mathrm{H}}, \\
\left\|u^{\mathrm{lin}}(t, \cdot)\right\|_{\dot{H}^\alpha \left(\mathbf{H}_n\right)} \lesssim &(1+t)^{-\frac{2 \min \left\{ r_\mathrm{H}\left(u_0\right), r_\mathrm{H}\left(u_1\right)-2\delta_2 \right\} +Q+2\alpha}{4\delta_1-4\delta_2}}\left\|\left(u_0, u_1\right)\right\|_{D^{\delta_1, 0}_\mathrm{H}}, \\
\left\|\partial_t u^{\mathrm{lin}}(t,\cdot)\right\|_{L^2\left(\mathbf{H}_n\right)} \lesssim & \log(\mathrm{e}+t) (1+t)^{-\frac{2 \min \left\{ r_\mathrm{H}\left(u_0\right), r_\mathrm{H}\left(u_1\right)-2\delta_2 \right\} +Q+4\delta_1-4\delta_2}{4\delta_1-4\delta_2}} \left\|\left(u_0, u_1\right)\right\|_{D^{\delta_1, 0}_\mathrm{H}}.
\end{align*}
This means that the linear part fulfills
\begin{align*}
&(1+t)^{\frac{2 \min \left\{ r_\mathrm{H}\left(u_0\right), r_\mathrm{H}\left(u_1\right)-2\delta_2 \right\} +Q}{4\delta_1-4\delta_2}}\left\|u^{\mathrm{lin}}(t, \cdot)\right\|_{L^2\left(\mathbf{H}_n\right)} \\
&+(1+t)^{\frac{2 \min \left\{ r_\mathrm{H}\left(u_0\right), r_\mathrm{H}\left(u_1\right)-2\delta_2 \right\} +Q+2\alpha}{4\delta_1-4\delta_2}}\left\|u^{\mathrm{lin}}(t, \cdot)\right\|_{\dot{H}^\alpha \left(\mathbf{H}_n\right)} \\
&+ \log(\mathrm{e}+t)^{-1}(1+t)^{\frac{2 \min \left\{ r_\mathrm{H}\left(u_0\right), r_\mathrm{H}\left(u_1\right)-2\delta_2 \right\} +Q+4\delta_1-4\delta_2}{4\delta_1-4\delta_2}} \left\|\partial_t u^{\mathrm{lin}}(t,\cdot)\right\|_{L^2\left(\mathbf{H}_n\right)} \lesssim \left\|\left(u_0, u_1\right)\right\|_{D^{\delta_1, 0}_\mathrm{H}}.
\end{align*}
From the above we see that $u^{\mathrm{lin}} \in X_\mathrm{H}(T)$. It remains to prove that $\left\|u^{\mathrm{non}}\right\|_{X_\mathrm{H}(T)} \lesssim\|u\|_{X_\mathrm{H}(T)}^p$. We may estimate the $\dot{H}^\alpha \left(\mathbf{H}_n\right)$ norm of $u^{\mathrm{non}}(t, \cdot)$ by applying $\left(L^2\left(\mathbf{H}_n\right) \cap L^\omega\left(\mathbf{H}_n\right)\right)-L^2\left(\mathbf{H}_n\right)$ estimates in $\left[0, \frac{t}{2}\right]$ and $L^2\left(\mathbf{H}_n\right)-L^2\left(\mathbf{H}_n\right)$ estimates in $\left[\frac{t}{2}, t\right]$ from Proposition \ref{proposition:7.2} as follows:
\begin{align*}
\left\|u^{\mathrm{non}}(t, \cdot)\right\|_{\dot{H}^\alpha \left(\mathbf{H}_n\right)} \leq & \int_0^{\frac{t}{2}} (1+t-s)^{-\frac{Q}{2\delta_1-2\delta_2}\left(\frac{1}{\omega}-\frac{1}{2}\right)-\frac{\alpha}{2\delta_1-2\delta_2}+\frac{\delta_2}{\delta_1-\delta_2}}\left\| \left| u(s, \cdot)\right|^p\right\|_{L^{\omega}\left(\mathbf{H}_n\right) \cap L^2\left(\mathbf{H}_n\right)} \mathrm{d} s \\
&+ \int_{\frac{t}{2}}^t (1+t-s)^{-\frac{\alpha}{2\delta_1-2\delta_2}+\frac{\delta_2}{\delta_1-\delta_2}}\left\| \left| u(s, \cdot)\right|^p\right\|_{L^2\left(\mathbf{H}_n\right)} \mathrm{d} s.
\end{align*}
On the one hand, the first integral in the right hand part of the last inequality can be estimated as
\begin{align*}
& \int_0^{\frac{t}{2}} (1+t-s)^{-\frac{Q}{2\delta_1-2\delta_2}\left(\frac{1}{\omega}-\frac{1}{2}\right)-\frac{\alpha}{2\delta_1-2\delta_2}+\frac{\delta_2}{\delta_1-\delta_2}}\left\| \left| u(s, \cdot)\right|^p\right\|_{L^{\omega}\left(\mathbf{H}_n\right) \cap L^2\left(\mathbf{H}_n\right)} \mathrm{d} s \\
\lesssim & \|u\|_{X_\mathrm{H}(T)}^p(1+t)^{-\frac{Q}{2\delta_1-2\delta_2}\left(\frac{1}{\omega}-\frac{1}{2}\right)-\frac{\alpha}{2\delta_1-2\delta_2}+\frac{\delta_2}{\delta_1-\delta_2}} \int_0^{\frac{t}{2}} (1+s)^{-\frac{Q +\min \left\{r_\mathrm{H}\left(u_0\right), r_\mathrm{H}\left(u_1\right)-2\delta_2 \right\}}{2\delta_1-2\delta_2}p+\frac{Q}{2 \omega\left(\delta_1-\delta_2\right)}} \mathrm{d} s \\
\lesssim & \|u\|_{X_\mathrm{H}(T)}^p(1+t)^{-\frac{Q}{2\delta_1-2\delta_2}\left(\frac{1}{\omega}-\frac{1}{2}\right)-\frac{\alpha}{2\delta_1-2\delta_2}+\frac{\delta_2}{\delta_1-\delta_2}} = \|u\|_{X_\mathrm{H}(T)}^p(1+t)^{-\frac{2 \min \left\{ r_\mathrm{H}\left(u_0\right), r_\mathrm{H}\left(u_1\right)-2\delta_2 \right\} +Q+2\alpha}{4\delta_1-4\delta_2}},
\end{align*}
by \eqref{ineq:5.4}, $\|\cdot\|_{X_\mathrm{H}(\tau)} \lesssim\|\cdot\|_{X_\mathrm{H}(T)}$ for $0 \leq \tau \leq T$ and the definition of $\omega$. Since $p>p_{\mathrm{Fuji}}\left(\frac{Q-2 \omega \delta_2}{\omega \delta_1}\right)$, it follows that $-\frac{Q +\min \left\{r_\mathrm{H}\left(u_0\right), r_\mathrm{H}\left(u_1\right)-2\delta_2 \right\}}{2\delta_1-2\delta_2}p+\frac{Q}{2 \omega\left(\delta_1-\delta_2\right)} < -1$. On the other hand, for the second integral, using inequality \eqref{ineq:5.3}, $p>p_{\mathrm{Fuji}}\left(\frac{Q-2 \omega \delta_2}{\omega \delta_1}\right)$ and the definition of $\omega$, we have
\begin{align*}
& \int_{\frac{t}{2}}^t (1+t-s)^{-\frac{\alpha}{2\delta_1-2\delta_2}+\frac{\delta_2}{\delta_1-\delta_2}}\left\| \left| u(s, \cdot)\right|^p\right\|_{L^2\left(\mathbf{H}_n\right)} \mathrm{d} s \\
\lesssim & \|u\|_{X_\mathrm{H}(T)}^p(1+t)^{-\frac{Q +\min \left\{r_\mathrm{H}\left(u_0\right), r_\mathrm{H}\left(u_1\right)-2\delta_2 \right\}}{2\delta_1-2\delta_2}p+\frac{Q}{4\delta_1-4\delta_2}} \int_{\frac{t}{2}}^t (1+t-s)^{-\frac{\alpha}{2\delta_1-2\delta_2}+\frac{\delta_2}{\delta_1-\delta_2}} \mathrm{d} s \\ \lesssim & \|u\|_{X_\mathrm{H}(T)}^p(1+t)^{-\frac{Q +\min \left\{r_\mathrm{H}\left(u_0\right), r_\mathrm{H}\left(u_1\right)-2\delta_2 \right\}}{2\delta_1-2\delta_2}p+\frac{Q}{4\delta_1-4\delta_2}} (1+t)^{1-\frac{\alpha}{2\delta_1-2\delta_2}+\frac{\delta_2}{\delta_1-\delta_2}} \\
\lesssim & \|u\|_{X_\mathrm{H}(T)}^p(1+t)^{-\frac{2 \min \left\{ r_\mathrm{H}\left(u_0\right), r_\mathrm{H}\left(u_1\right)-2\delta_2 \right\} +Q+2\alpha}{4\delta_1-4\delta_2}}.
\end{align*}
Therefore, $\left\|u^{\mathrm{non}}(t, \cdot)\right\|_{\dot{H}^\alpha \left(\mathbf{H}_n\right)} \lesssim (1+t)^{-\frac{2 \min \left\{ r_\mathrm{H}\left(u_0\right), r_\mathrm{H}\left(u_1\right)-2\delta_2 \right\} +Q+2\alpha}{4\delta_1-4\delta_2}} \|u\|_{X_\mathrm{H}(T)}^p$. In the same way, we can derive
\begin{align*}
\left\|u^{\mathrm{non}}(t, \cdot)\right\|_{L^2\left(\mathbf{H}_n\right)} \lesssim & (1+t)^{-\frac{2 \min \left\{ r_\mathrm{H}\left(u_0\right), r_\mathrm{H}\left(u_1\right)-2\delta_2 \right\} +Q}{4\delta_1-4\delta_2}} \|u\|_{X_\mathrm{H}(T)}^p , \\
\left\| \partial_t u^{\mathrm{non}}(t, \cdot)\right\|_{L^2\left(\mathbf{H}_n\right)} \lesssim & \log(\mathrm{e}+t) (1+t)^{-\frac{2 \min \left\{ r_\mathrm{H}\left(u_0\right), r_\mathrm{H}\left(u_1\right)-2\delta_2 \right\} +Q+4\delta_1-4\delta_2}{4\delta_1-4\delta_2}}\|u\|_{X_\mathrm{H}(T)}^p.
\end{align*}
From the definition of the norm $X_\mathrm{H}(T)$ the inequality \eqref{ineq:5.1} is verified. Let us prove \eqref{ineq:5.2}. We see that
$$
\|N u-N v\|_{X_\mathrm{H}(T)}=\left\|\int_0^t \left( \left | u(s, \eta)\right|^p-\left| v(s, \eta)\right|^p\right) *_{(\eta)} E_1(t, s, \eta) \mathrm{d} s\right\|_{X_\mathrm{H}(T)} .
$$
Thanks to the estimates for the solutions from Proposition \ref{proposition:7.2}, we can estimate
\begin{align} \label{ineq:5.5}
& \left\|\left( \left| u(s, \eta)\right|^p-\left| v(s, \eta)\right|^p\right) *_{(\eta)} E_1(t, s, \eta) \right\|_{\dot{H}^\alpha \left(\mathbf{H}_n\right)} \notag\\
\lesssim &
\begin{cases}
(1+t-s)^{-\frac{Q}{2\delta_1-2\delta_2}\left(\frac{1}{\omega}-\frac{1}{2}\right)-\frac{\alpha-2\delta_2}{2\delta_1-2\delta_2}} \left\| \left| u(s, \eta)\right|^p-\left| v(s, \eta)\right|^p\right\|_{L^\omega\left(\mathbf{H}_n\right) \cap L^2\left(\mathbf{H}_n\right)} & \text { if } s \in \left[0, \frac{t}{2}\right] \\
(1+t-s)^{-\frac{\alpha-2\delta_2}{2\delta_1-2\delta_2}} \left\| \left| u(s, \eta)\right|^p-\left| v(s, \eta)\right|^p\right\|_{L^2\left(\mathbf{H}_n\right)} & \text { if } s \in \left[\frac{t}{2}, t\right].
\end{cases}
\end{align}
Since
$$
\left| \left| u(s, \eta)\right|^p-\left| v(s, \eta)\right|^p \right| \lesssim \left| u(s, \eta)- v(s, \eta)\right| \left(\left| u(s, \eta)|^{p-1}+| v(s, \eta)\right|^{p-1}\right),
$$
by H$\ddot{\text{o}}$lder's inequality, we obtain
\begin{align*}
\left\| \left| u(s, \cdot)\right|^p-\left| v(s, \cdot)\right|^p\right\|_{L^\omega\left(\mathbf{H}_n\right)} 
\lesssim & \left\| u(s, \cdot)- v(s, \cdot)\right\|_{L^{\omega p}\left(\mathbf{H}_n\right)}\left( \left\| u(s, \cdot)\right\|_{L^{\omega p}\left(\mathbf{H}_n\right)}^{p-1}+\left\| v(s, \cdot)\right\|_{L^{\omega p}\left(\mathbf{H}_n\right)}^{p-1}\right), \\
\left\| \left| u(s, \cdot)\right|^p-\left| v(s, \cdot)\right|^p\right\|_{L^2\left(\mathbf{H}_n\right)} 
\lesssim & \left\| u(s, \cdot)- v(s, \cdot)\right\|_{L^{2 p}\left(\mathbf{H}_n\right)}\left( \left\| u(s, \cdot)\right\|_{L^{2 p}\left(\mathbf{H}_n\right)}^{p-1}+\left\| v(s, \cdot)\right\|_{L^{2 p}\left(\mathbf{H}_n\right)}^{p-1}\right).
\end{align*}
Just as in proving \eqref{ineq:5.1}, we apply the fractional Gagliardo-Nirenberg inequality
from Proposition \ref{Gagliardo} to the terms $\left \| u(s, \cdot)- v(s, \cdot)\right\|_{L^h\left(\mathbf{H}_n\right)}, \left\| u(s, \cdot)\right\|_{L^h\left(\mathbf{H}_n\right)}$ and $\left\| v(s, \cdot)\right\|_{L^h\left(\mathbf{H}_n\right)}$,
with $h=\omega p$ and $h=2 p$. Plugging these estimates into \eqref{ineq:5.5}, similarly to estimating $\left\|u^{\mathrm{non}}(t, \cdot)\right\|_{\dot{H}^\alpha \left(\mathbf{H}_n\right)}$, we derive \eqref{ineq:5.2}. The proof of Theorem\ref{theorem:1.3} is completed.
\end{proof}

\begin{proof}[Proof of Theorem \ref{theorem:1.4}]
In this theorem, we define the solution space $X_\mathrm{H}(T)$ and its corresponding norm $\|u\|_{X_\mathrm{H}(T)}$ as in Theorem \ref{theorem:1.3}. Given that $\omega:=\frac{Q}{Q+\min \left\{ r_\mathrm{H}\left(u_0\right), r_\mathrm{H}\left(u_1\right)-2\delta_2 \right\}+2\delta_2}$, it is straightforward to conclude that  $1 < \omega < 2$. Since $\frac{2}{\omega}<p:=p_{\mathrm{Fuji}}\left(\frac{Q-2 \omega \delta_2}{\omega \delta_1}\right):=1+\frac{2 \omega \delta_1}{Q-2 \omega \delta_2}$, there exists $\omega_1 \in [1, 2]$ such that $\frac{2}{\omega}<\frac{2}{\omega_1} <p$. By applying the Gagliardo-Nirenberg inequality with interpolation exponents $\theta_1(2 p):=\frac{Q(p-1)}{2 \alpha p}, \theta_2(\omega_1 p):=\frac{Q(\omega_1 p-2)}{2 \omega_1 \alpha p}$ from Lemma \ref{Gagliardo} and the definition of the evolution space, we obtain the following inequalities:
\begin{align} \label{ineq:5.6}
\left\| \left| u(\tau, \cdot)\right|^p\right\|_{L^2\left(\mathbf{H}_n\right)}=\left\| u(\tau, \cdot)\right\|_{L^{2 p}\left(\mathbf{H}_n\right)}^p \lesssim & (1+\tau)^{-\frac{Q +\min \left\{r_\mathrm{H}\left(u_0\right), r_\mathrm{H}\left(u_1\right)-2\delta_2 \right\}}{2\delta_1-2\delta_2}p+\frac{Q}{4\delta_1-4\delta_2}}\|u\|_{X_\mathrm{H}(T)}^p, \\
\label{ineq:5.7}
\left\| \left| u(\tau, \cdot)\right|^p\right\|_{L^{\omega_1}\left(\mathbf{H}_n\right) \cap L^2\left(\mathbf{H}_n\right)}=&\left\| u(\tau, \cdot)\right\|_{L^{\omega_1 p}\left(\mathbf{H}_n\right)}^p + \left\| u(\tau, \cdot)\right\|_{L^{2 p}\left(\mathbf{H}_n\right)}^p \notag \\
\lesssim&(1+\tau)^{-\frac{Q +\min \left\{r_\mathrm{H}\left(u_0\right), r_\mathrm{H}\left(u_1\right)-2\delta_2 \right\}}{2\delta_1-2\delta_2}p+\frac{Q}{2 \omega_1\left(\delta_1-\delta_2\right)}}\|u\|_{X_\mathrm{H}(T)}^p,
\end{align}
under the condition that 
$p \in \left(\frac{2}{\omega}, \infty\right)$ for $Q \leq 2 \alpha$, or $ p \in\left(\tfrac{2}{\omega}, \frac{Q}{Q-2 \alpha}\right]$ when $2 \alpha < Q \leq \tfrac{4 \alpha}{2-\omega}$. Proceeding in exactly the same way as in Theorem \ref{theorem:1.3}, we also obtain $u^{\mathrm{lin}} \in X_\mathrm{H}(T)$. The next step is to establish that $\left\|u^{\mathrm{non}}\right\|_{X_\mathrm{H}(T)} \lesssim\|u\|_{X_\mathrm{H}(T)}^p$. We may estimate the $\dot{H}^\alpha \left(\mathbf{H}_n\right)$ norm of $u^{\mathrm{non}}(t, \cdot)$ by applying $\left(L^2\left(\mathbf{H}_n\right) \cap L^{\omega_1}\left(\mathbf{H}_n\right)\right)-L^2\left(\mathbf{H}_n\right)$ estimates in $\left[0, \frac{t}{2}\right]$ and $L^2\left(\mathbf{H}_n\right)-L^2\left(\mathbf{H}_n\right)$ estimates in $\left[\frac{t}{2}, t\right]$ from Proposition \ref{proposition:7.2} as follows:
\begin{align*}
\left\|u^{\mathrm{non}}(t, \cdot)\right\|_{\dot{H}^\alpha \left(\mathbf{H}_n\right)} \leq & \int_0^{\frac{t}{2}} (1+t-s)^{-\frac{Q}{2\delta_1-2\delta_2}\left(\frac{1}{\omega_1}-\frac{1}{2}\right)-\frac{\alpha}{2\delta_1-2\delta_2}+\frac{\delta_2}{\delta_1-\delta_2}}\left\| \left| u(s, \cdot)\right|^p\right\|_{L^{\omega_1}\left(\mathbf{H}_n\right) \cap L^2\left(\mathbf{H}_n\right)} d s \\
&+ \int_{\frac{t}{2}}^t (1+t-s)^{-\frac{\alpha}{2\delta_1-2\delta_2}+\frac{\delta_2}{\delta_1-\delta_2}}\left\| \left| u(s, \cdot)\right|^p\right\|_{L^2\left(\mathbf{H}_n\right)} d s .
\end{align*}
For the first integral, we can estimate it as follows:
\begin{align*}
& \int_0^{\frac{t}{2}} (1+t-s)^{-\frac{Q}{2\delta_1-2\delta_2}\left(\frac{1}{\omega_1}-\frac{1}{2}\right)-\frac{\alpha}{2\delta_1-2\delta_2}+\frac{\delta_2}{\delta_1-\delta_2}}\left\| \left| u(s, \cdot)\right|^p\right\|_{L^{\omega_1}\left(\mathbf{H}_n\right) \cap L^2\left(\mathbf{H}_n\right)} d s \\
\lesssim & \|u\|_{X_\mathrm{H}(T)}^p(1+t)^{-\frac{Q}{2\delta_1-2\delta_2}\left(\frac{1}{\omega_1}-\frac{1}{2}\right)-\frac{\alpha}{2\delta_1-2\delta_2}+\frac{\delta_2}{\delta_1-\delta_2}} \int_0^{\frac{t}{2}} (1+s)^{-\frac{Q +\min \left\{r_\mathrm{H}\left(u_0\right), r_\mathrm{H}\left(u_1\right)-2\delta_2 \right\}}{2\delta_1-2\delta_2}p+\frac{Q}{2 \omega_1\left(\delta_1-\delta_2\right)}} d s \\
\lesssim & \|u\|_{X_\mathrm{H}(T)}^p (1+t)^{-\frac{Q}{2\delta_1-2\delta_2}\left(\frac{1}{\omega_1}-\frac{1}{2}\right)-\frac{\alpha}{2\delta_1-2\delta_2}+\frac{\delta_2}{\delta_1-\delta_2}} \lesssim \|u\|_{X_\mathrm{H}(T)}^p (1+t)^{-\frac{Q}{2\delta_1-2\delta_2}\left(\frac{1}{\omega}-\frac{1}{2}\right)-\frac{\alpha}{2\delta_1-2\delta_2}+\frac{\delta_2}{\delta_1-\delta_2}} \\
= & \|u\|_{X_\mathrm{H}(T)}^p(1+t)^{-\frac{2 \min \left\{ r_\mathrm{H}\left(u_0\right), r_\mathrm{H}\left(u_1\right)-2\delta_2 \right\} +Q+2\alpha}{4\delta_1-4\delta_2}},
\end{align*}
where we used \eqref{ineq:5.7}, $\|\cdot\|_{X_\mathrm{H}(\tau)} \lesssim\|\cdot\|_{X_\mathrm{H}(T)}$ for $0 \leq \tau \leq T$ and the definition of $\omega$. Since $p:=p_{\mathrm{Fuji}}\left(\frac{Q-2 \omega \delta_2}{\omega \delta_1}\right):=1+\frac{2 \omega \delta_1}{Q-2 \omega \delta_2}>1+\frac{2 \omega_1 \delta_1}{Q-2 \omega \delta_2}$, it follows that $-\frac{Q +\min \left\{r_\mathrm{H}\left(u_0\right), r_\mathrm{H}\left(u_1\right)-2\delta_2 \right\}}{2\delta_1-2\delta_2}p+\frac{Q}{2 \omega_1\left(\delta_1-\delta_2\right)} < -1$. For the second integral, using \eqref{ineq:5.6}, $p:=p_{\mathrm{Fuji}}\left(\frac{Q-2 \omega \delta_2}{\omega \delta_1}\right)$ and the definition of $\omega$, we obtain
\begin{align*}
& \int_{\frac{t}{2}}^t (1+t-s)^{-\frac{\alpha}{2\delta_1-2\delta_2}+\frac{\delta_2}{\delta_1-\delta_2}}\left\| \left| u(s, \cdot)\right|^p\right\|_{L^2\left(\mathbf{H}_n\right)} d s \\
\lesssim & \|u\|_{X_\mathrm{H}(T)}^p(1+t)^{-\frac{Q +\min \left\{r_\mathrm{H}\left(u_0\right), r_\mathrm{H}\left(u_1\right)-2\delta_2 \right\}}{2\delta_1-2\delta_2}p+\frac{Q}{4\delta_1-4\delta_2}} \int_{\frac{t}{2}}^t (1+t-s)^{-\frac{\alpha}{2\delta_1-2\delta_2}+\frac{\delta_2}{\delta_1-\delta_2}} d s \\ \lesssim & \|u\|_{X_\mathrm{H}(T)}^p(1+t)^{-\frac{Q +\min \left\{r_\mathrm{H}\left(u_0\right), r_\mathrm{H}\left(u_1\right)-2\delta_2 \right\}}{2\delta_1-2\delta_2}p+\frac{Q}{4\delta_1-4\delta_2}} (1+t)^{1-\frac{\alpha}{2\delta_1-2\delta_2}+\frac{\delta_2}{\delta_1-\delta_2}} \\
\lesssim & \|u\|_{X_\mathrm{H}(T)}^p(1+t)^{-\frac{2 \min \left\{ r_\mathrm{H}\left(u_0\right), r_\mathrm{H}\left(u_1\right)-2\delta_2 \right\} +Q+2\alpha}{4\delta_1-4\delta_2}}.
\end{align*}
Thus, we conclude that $\left\|u^{\mathrm{non}}(t, \cdot)\right\|_{\dot{H}^\alpha \left(\mathbf{H}_n\right)} \lesssim (1+t)^{-\frac{2 \min \left\{ r_\mathrm{H}\left(u_0\right), r_\mathrm{H}\left(u_1\right)-2\delta_2 \right\} +Q+2\alpha}{4\delta_1-4\delta_2}} \|u\|_{X_\mathrm{H}(T)}^p$. In a similar manner, we obtain
\begin{align*}
\left\|u^{\mathrm{non}}(t, \cdot)\right\|_{L^2\left(\mathbf{H}_n\right)} \lesssim & (1+t)^{-\frac{2 \min \left\{ r_\mathrm{H}\left(u_0\right), r_\mathrm{H}\left(u_1\right)-2\delta_2 \right\} +Q}{4\delta_1-4\delta_2}} \|u\|_{X_\mathrm{H}(T)}^p , \\
\left\| \partial_t u^{\mathrm{non}}(t, \cdot)\right\|_{L^2\left(\mathbf{H}_n\right)} \lesssim & \log(\mathrm{e}+t) (1+t)^{-\frac{2 \min \left\{ r_\mathrm{H}\left(u_0\right), r_\mathrm{H}\left(u_1\right)-2\delta_2 \right\} +Q+4\delta_1-4\delta_2}{4\delta_1-4\delta_2}}\|u\|_{X_\mathrm{H}(T)}^p.
\end{align*}
Using the definition of the norm $X_\mathrm{H}(T)$, the inequality \eqref{ineq:5.1} is thus verified. The remaining steps are carried out exactly as in Theorem \ref{theorem:1.3}.
\end{proof}
\begin{remark}
\rm{
According to Theorems \ref{theorem:1.3} and \ref{theorem:1.4}, if $\left(u_0, u_1\right) \in H_{\sigma \psi(t, \cdot)}^1\left(\mathbf{H}_n\right) \times L_{\sigma \psi(t,)}^2\left(\mathbf{H}_n\right)$, with $\sigma>0, t \geq 0,
\psi(t, \eta) := \frac{|x|^2+|y|^2+4|\tau|}{8(1+t)}, \eta=(x, y, \tau) \in \mathbf{H}_n$, then $P_{r_\mathrm{H}}\left(u_0\right)_{+}, P_{r_\mathrm{H}}\left(u_1\right)_{+}<\infty$ for $r_\mathrm{H}\left(u_0\right) = \delta_1, r_\mathrm{H}\left(u_1\right) = 0$ (see Proposition \ref{proposition:3.5}). From \eqref{p_*} and \eqref{p_*_1}, it can be concluded that the problem \eqref{eq:1.0} allows for a unique global solution (in time) when $p > p_{\mathrm{Fuji}}\left(\frac{Q-2\delta_2}{\delta_1}\right):=1+\frac{2 \delta_1}{Q-2 \delta_2}$ (the equality case does not belong to the solution-existence region, because in this case $\min \left\{ r_\mathrm{H}\left(u_0\right), r_\mathrm{H}\left(u_1\right)-2\delta_2 \right\} = -2\delta_2$, which leads to a violation of condition \eqref{condition_omega_p_crit}). This is precisely the result regarding the critical exponent $p_{\mathrm{Fuji}}(Q):=1+\frac{2}{Q}$ as presented in \cite[Theorem 2.2]{Georgiev2020}, in the case when $\delta_1=1$ and $\delta_2=0$.
}
\end{remark}
\begin{remark}
\rm{
In Theorems \ref{theorem:1.3} and \ref{theorem:1.4}, if $\left(u_0, u_1\right) \in\left(H^{\delta_1}\left(\mathbf{H}_n\right) \cap \dot{H}^{-\gamma}\left(\mathbf{H}_n\right)\right) \times\left(L^2\left(\mathbf{H}_n\right) \cap \dot{H}^{-\gamma}\left(\mathbf{H}_n\right)\right)$, with $\gamma \in\left[0, \frac{Q}{2}\right)$, then $P_{r_\mathrm{H}}\left(u_0\right)_{+}, P_{r_\mathrm{H}}\left(u_1\right)_{+}<\infty$ for $r_\mathrm{H}\left(u_0\right) = \gamma-\frac{Q}{2}+\delta_1, r_\mathrm{H}\left(u_1\right) = \gamma-\frac{Q}{2}$ (see Proposition \ref{proposition:3.3}). From \eqref{p_*} and \eqref{p_*_1}, the problem \eqref{eq:1.0} admits a unique global (in time) solution when $p \geq p_{\mathrm{Fuji}}\left(\frac{Q+2\gamma-4 \delta_2}{2 \delta_1}\right):=1+\frac{4 \delta_1}{Q+2\gamma-4 \delta_2}$. This exactly corresponds to the result on the condition $p \geq p_{\mathrm{Fuji}}\left(\frac{Q+2 \gamma}{2}\right):=1+\frac{4}{Q+2\gamma}$ established in \cite[Theorem 1.2]{Dasgupta2024} and \cite[Theorem 2]{DAbbicco2025}, in the case when $\delta_1=1$ and $\delta_2=0$.
}
\end{remark}
\begin{remark}
\rm{
In Theorems \ref{theorem:1.3} and \ref{theorem:1.4}, if $\left(u_0, u_1\right) \in \left(H^{\delta_1}\left(\mathbf{H}_n\right) \cap L^1\left(\mathbf{H}_n\right)\right) \times \left(L^2\left(\mathbf{H}_n\right) \cap L^1\left(\mathbf{H}_n\right)\right)$, or $\left(u_0, u_1\right) \in\left(H^{\delta_1}\left(\mathbf{H}_n\right) \cap \dot{H}^{-\frac{Q}{2}}\left(\mathbf{H}_n\right)\right) \times\left(L^2\left(\mathbf{H}_n\right) \cap \dot{H}^{-\frac{Q}{2}}\left(\mathbf{H}_n\right)\right)$, or $\left(u_0, u_1\right) \in\left(H^{\delta_1}\left(\mathbf{H}_n\right) \cap \mathcal{Y}^0\left(\mathbf{H}_n\right)\right) \times\left(L^2\left(\mathbf{H}_n\right) \cap \mathcal{Y}^0\left(\mathbf{H}_n\right)\right)$, then $P_{r_\mathrm{H}}\left(u_0\right)_{+}, P_{r_\mathrm{H}}\left(u_1\right)_{+}<\infty$ for $r_\mathrm{H}\left(u_0\right) = \delta_1, r_\mathrm{H}\left(u_1\right) = 0$ (see Propositions \ref{proposition:3.2}, \ref{proposition:3.3} and \ref{proposition:3.6}). From \eqref{p_*} and \eqref{p_*_1}, the problem \eqref{eq:1.0} admits a unique global (in time) solution when $p > p_{\mathrm{Fuji}}\left(\frac{Q-2\delta_2}{\delta_1}\right):=1+\frac{2 \delta_1}{Q-2 \delta_2}$ (the equality case does not lie in the solution-existence region, because in these cases $\min \left\{ r_\mathrm{H}\left(u_0\right), r_\mathrm{H}\left(u_1\right)-2\delta_2 \right\} = -2\delta_2$, which leads to a violation of condition \eqref{condition_omega_p_crit}).
}
\end{remark}
\begin{remark}
\rm{
In Theorems \ref{theorem:1.3} and \ref{theorem:1.4}, if $\left(u_0, u_1\right) \in \left(H^{\delta_1}\left(\mathbf{H}_n\right) \cap L^m\left(\mathbf{H}_n\right)\right) \times \left(L^2\left(\mathbf{H}_n\right) \cap L^m\left(\mathbf{H}_n\right)\right)$, with $m \in (1,2]$, then $P_{r_\mathrm{H}}\left(u_0\right)_{+}, P_{r_\mathrm{H}}\left(u_1\right)_{+}<\infty$ for $r_\mathrm{H}\left(u_0\right) = -Q\left(1-\frac{1}{m}\right)+\delta_1, r_\mathrm{H}\left(u_1\right) = -Q\left(1-\frac{1}{m}\right)$ (see Proposition \ref{proposition:3.4}). From \eqref{p_*} and \eqref{p_*_1}, the problem \eqref{eq:1.0} admits a unique global solution (in time) when $p \geq p_{\mathrm{Fuji}}\left(\frac{Q-2 m \delta_2}{m \delta_1}\right):=1+\frac{2 m \delta_1}{Q-2 m \delta_2}$.
}
\end{remark}
\begin{remark}
\rm{
In Theorems \ref{theorem:1.3} and \ref{theorem:1.4}, if $\left(u_0, u_1\right) \in\left(H^{\delta_1}\left(\mathbf{H}_n\right) \cap \mathcal{Y}^q\left(\mathbf{H}_n\right)\right) \times\left(L^2\left(\mathbf{H}_n\right) \cap \mathcal{Y}^q\left(\mathbf{H}_n\right)\right)$, with $q\in \left(0,\frac{Q}{2} \right]$, then $P_{r_\mathrm{H}}\left(u_0\right)_{+}, P_{r_\mathrm{H}}\left(u_1\right)_{+}<\infty$ for $r_\mathrm{H}\left(u_0\right) = -q+\delta_1, r_\mathrm{H}\left(u_1\right) = -q$ (see Proposition \ref{proposition:3.6}). From \eqref{p_*} and \eqref{p_*_1}, the problem \eqref{eq:1.0} has a unique global (in time) solution when $p \geq p_{\mathrm{Fuji}}\left(\frac{Q-q-2\delta_2}{\delta_1}\right):=1 + \frac{2\delta_1}{Q-q-2\delta_2}$.
}
\end{remark}
\begin{remark} \label{remark_theorem_1.4.2}
\rm{
In Theorems \ref{theorem:1.3} and \ref{theorem:1.4}, if $\left(u_0, u_1\right) \in\left(H^{\delta_1}\left(\mathbf{H}_n\right) \cap \dot{H}_m^{-\gamma}\left(\mathbf{H}_n\right)\right) \times\left(L^2\left(\mathbf{H}_n\right) \cap \dot{H}_m^{-\gamma}\left(\mathbf{H}_n\right)\right)$, with $m \in (1, 2], \gamma \in\left[0, Q-\frac{Q}{m}\right)$, then $P_{r_\mathrm{H}}\left(u_0\right)_{+}, P_{r_\mathrm{H}}\left(u_1\right)_{+}<\infty$ for $r_\mathrm{H}\left(u_0\right) = \gamma-\frac{Q(m-1)}{m}+\delta_1, r_\mathrm{H}\left(u_1\right) = \gamma-\frac{Q(m-1)}{m}$ (see Proposition \ref{proposition:3.7}). From \eqref{p_*} and \eqref{p_*_1}, the problem \eqref{eq:1.0} has a unique global (in time) solution when $p \geq p_{\mathrm{Fuji}}\left(\frac{Q+m\gamma-2 m \delta_2}{m \delta_1}\right):=1 + \frac{2 m \delta_1}{Q+m\gamma-2 m \delta_2}$.
}
\end{remark}

\section{blow-up results} \label{section_6}
The next main result is concerned with indicating the sharpness of the exponent $p$ to \eqref{eq:1.0} in the case of $\left(u_0, u_1\right) \in\dot{H}_m^{-\gamma}\left(\mathbf{H}_n\right) \times\dot{H}_m^{-\gamma}\left(\mathbf{H}_n\right)$, with $m \in (1, 2]$ and $\gamma \in\left[0, Q-\frac{Q}{m}\right)$. For the sake of brevity, in the blow-up result below we restrict our attention to the case $\delta_1 \in (0, 1]$ and $\delta_2 \in \left[0, \frac{\delta_1}{2}\right]$.
\begin{theorem} \label{theorem:1.5}
(Blow-up). Let $\delta_1 \in (0, 1], \delta_2 \in \left[0, \frac{\delta_1}{2}\right]$ and $Q=2n+2$ be the homogeneous dimension on the Heisenberg group. Assume that the initial data $u_0=0$ and $u_1 \in \dot{H}_m^{-\gamma}\left(\mathbf{H}_n\right)$, with $m \in (1, 2], \gamma \in\left[0, Q-\frac{Q}{m}\right)$ satisfying the following relation:
\begin{equation}
\label{condition_u0u1}
u_1(\eta) \gtrsim\langle \eta\rangle_{\mathbf{H}_n}^{-Q\left(\frac{1}{m}+\frac{\gamma}{Q}\right)} \log (\mathrm{e}+|\eta|_{\mathbf{H}_n})^{-1} \quad \text{ for all } \eta \in \mathbf{H}_n,
\end{equation}
where $|\cdot|_{\mathbf{H}_n}$ be any homogeneous norm on the Heisenberg group $\mathbf{H}_n$, while we denote $\left(1+|\eta|_{\mathbf{H}_n}^2\right)^{\frac{1}{2}}$ by the Japanese bracket $\langle \eta\rangle_{\mathbf{H}_n}$ for $\eta \in \mathbf{H}_n$. Moreover, we suppose that the following conditions hold:
\begin{equation} \label{p_critical}
1 < p < p_{\mathrm{Fuji}}\left(\tfrac{Q+m\gamma-2 m \delta_2}{m \delta_1}\right):=1+\tfrac{2 m \delta_1}{Q+m\gamma-2 m \delta_2}.
\end{equation}
Then, there is no global (in time) weak solution to the Cauchy problem \eqref{eq:1.0}.    
\end{theorem}
\begin{remark}
\rm{
We will show that the set of all initial data $u_0=0, u_1 \in \dot{H}_m^{-\gamma}\left(\mathbf{H}_n\right)$, with $m \in (1, 2], \gamma \in\left[0, Q-\frac{Q}{m}\right)$, under the assumption stated in \eqref{condition_u0u1}, is non-empty. To this end, we define the set $\mathbb{D}_{Q, m, \gamma}$ as
$$
\mathbb{D}_{Q, m, \gamma}:=\left\{\left(u_0, u_1\right): u_0(\eta)=0, \, u_1(\eta) \gtrsim \langle \eta\rangle_{\mathbf{H}_n}^{-Q\left(\frac{1}{m}+\frac{\gamma}{Q}\right)}\left(\log \left(\mathrm{e}+|\eta|_{\mathbf{H}_n}\right)\right)^{-1}\right\} .
$$
Let $\frac{m Q}{Q+m \gamma}>1$. We choose $u_1(\eta)=C\langle \eta\rangle_{\mathbf{H}_n}^{-Q\left(\frac{1}{m}+\frac{\gamma}{Q}\right)}\left(\log \left(\mathrm{e}+|\eta|_{\mathbf{H}_n}\right)\right)^{-1}$. Then we get
\begin{align*}
\int_{\mathbb{R}^n} u_1(\eta)^{\frac{m Q}{Q+m \gamma}} \mathrm{d} \eta =&C^{\frac{m Q}{Q+m \gamma}} \int_{\mathbb{R}^n}\langle \eta\rangle_{\mathbf{H}_n}^{-Q}\left(\log \left(\mathrm{e}+|\eta|_{\mathbf{H}_n}\right)\right)^{-\frac{m Q}{Q+m \gamma}} \mathrm{d} \eta \\
\lesssim & \int_0^{\infty}\langle r\rangle^{-1}(\log (\mathrm{e}+r))^{-\frac{m Q}{Q+m \gamma}} \mathrm{d} r<\infty.
\end{align*}
Hence $u_1 \in L^{\frac{m Q}{Q+m \gamma}}\left(\mathbf{H}_n\right)$ for $\frac{m Q}{Q+m \gamma}>1$. By Lemma \ref{Hardy_Littlewood}, it follows that $L^{\frac{m Q}{Q+m \gamma}}\left(\mathbf{H}_n\right) \subset \dot{H}_m^{-\gamma}\left(\mathbf{H}_n\right)$ since $\frac{Q+m \gamma}{m Q}-\frac{1}{m}=\frac{\gamma}{Q}$ with $\gamma \in\left[0, Q-\frac{Q}{m}\right)$. Therefore, $\left(u_0, u_1\right) \in \mathbb{D}_{Q, m, \gamma} \cap\left(\dot{H}_m^{-\gamma}\left(\mathbf{H}_n\right) \times \dot{H}_m^{-\gamma}\left(\mathbf{H}_n\right)\right) \neq \emptyset$ for $\gamma \in\left[0, Q-\frac{Q}{m}\right)$.
}
\end{remark}
\begin{remark}
\rm{
It is clear that, if let $\left(u_0, u_1\right) \in\left(H^{\delta_1}\left(\mathbf{H}_n\right) \cap \dot{H}_m^{-\gamma}\left(\mathbf{H}_n\right)\right) \times\left(L^2\left(\mathbf{H}_n\right) \cap \dot{H}_m^{-\gamma}\left(\mathbf{H}_n\right)\right)$, with $m \in (1, 2], \gamma \in\left[0, Q-\frac{Q}{m}\right)$ in Theorems \ref{theorem:1.3} and \ref{theorem:1.4}, the blow-up result in Theorem \ref{theorem:1.5} confirms that the exponent $p_{\mathrm{Fuji}}\left(\frac{Q+m\gamma-2 m \delta_2}{m \delta_1}\right):= 1+\frac{2 m \delta_1}{Q+m\gamma-2 m \delta_2}$ is indeed critical (see Remark \ref{remark_theorem_1.4.2}).
}
\end{remark}
Next, we collect some preliminary knowledge about a modified test function method needed in our proofs.
\begin{definition} \cite{Palatucci2022}.
\label{definition:6.1}
Let $s \in(0,1)$. Let $X$ be a suitable set of functions defined on $\mathbf{H}_n$. Then, the fractional Laplacian $(-\Delta_{\mathrm{H}})^s$ in $\mathbf{H}_n$ is a non-local operator given by 
$$
(-\Delta_{\mathrm{H}})^s: v \in X \rightarrow(-\Delta_{\mathrm{H}})^s v(\eta) := C_{n, s} \mathrm{p.v.} \int_{\mathbf{H}_n} \tfrac{v(\eta)-v(\zeta)}{\left|\zeta^{-1} \circ \eta\right|_{\mathbf{H}_n}^{Q+2 s}} \mathrm{d} \zeta
$$
as long as the right-hand side exists, where $\mathrm{p.v.}$ stands for Cauchy's principal value, $\Gamma$ denotes the Gamma function and $C_{n, s} := 2^{n-1+3 s} \pi^{-n-1} \Gamma\left(\frac{n+s+1}{2}\right)^2$ is a normalization constant.
\end{definition}
\begin{lemma} \label{lemma:6.2}
Let $\langle \eta\rangle_{\mathbf{H}_n} := \left(1+|\eta|_{\mathbf{H}_n}^2\right)^{\frac{1}{2}}$ for $\eta \in \mathbf{H}_n$ and $q>0$. Then, the following estimate holds:
$$
\left|\Delta_{\mathrm{H}} \langle \eta\rangle_{\mathbf{H}_n}^{-q}\right| \lesssim\langle \eta\rangle_{\mathbf{H}_n}^{-q-2}.
$$
\end{lemma}
\begin{proof}
For a point $\eta=(x, y, \tau) \in \mathbf{H}_n$, let $r =\left(|x|^2+|y|^2\right)^\frac{1}{2}$ and $\rho=|\eta|_{\mathbf{H}_n}=\left(r^4+\tau^2\right)^\frac{1}{4}$. Let $\psi$ be a radial function, that is, $\psi$ depends only on $\rho$. Direct computation shows that
$$
\tfrac{\partial \psi}{\partial x_j} =\tfrac{\partial \psi}{\partial \rho} \tfrac{\partial \rho}{\partial x_j}=\rho^{-3} r^2 x_j \tfrac{\partial \psi}{\partial \rho} \quad \text{ and } \quad \tfrac{\partial \psi}{\partial y_j}=\tfrac{\partial \psi}{\partial \rho} \tfrac{\partial \rho}{\partial y_j}=\rho^{-3} r^2 y_j \tfrac{\partial \psi}{\partial \rho}.
$$
Then we deduce that
\begin{align*}
\tfrac{\partial^2 \psi}{\partial x_j^2}=&\tfrac{r^4 x_j^2}{\rho^6} \tfrac{\partial^2 \psi}{\partial \rho^2}+\tfrac{\rho^4 r^2+2 \rho^4 x_j^2-3 r^4 x_j^2}{\rho^7} \tfrac{\partial \psi}{\partial \rho}, \quad
\tfrac{\partial^2 \psi}{\partial y_j^2}=\tfrac{r^4 y_j^2}{\rho^6} \tfrac{\partial^2 \psi}{\partial \rho^2}+\tfrac{\rho^4 r^2+2 \rho^4 y_j^2-3 r^4 y_j^2}{\rho^7} \tfrac{\partial \psi}{\partial \rho}, \\
\tfrac{\partial^2 \psi}{\partial \tau^2}=&\tfrac{\tau^2}{4 \rho^6} \tfrac{\partial^2 \psi}{\partial \rho^2}+\tfrac{2 \rho^4-3 \tau^2}{4 \rho^7} \tfrac{\partial \psi}{\partial \rho}.
\end{align*}
Hence, by using that $\frac{1}{\rho^6}\left(\sum_{j=1}^n\left(r^4 x_j^2+r^4 y_j^2\right)+r^2 \tau^2\right)=\frac{r^2}{\rho^2}$ and
\begin{align*}
&\tfrac{1}{\rho^7} \left[(2 n+1) \rho^4 r^2+\left(2 \rho^4-3 r^4+\left(2 \rho^4-3 \tau^2\right)\right)\sum_{j=1}^n \left(x_j^2+y_j^2 \right) \right] \\
=&\tfrac{1}{\rho^7}\left[(2 n+1) \rho^4 r^2+2 \rho^4 r^2-3 r^6+r^2\left(2 \rho^4-3 \tau^2\right)\right]=\tfrac{Q r^2}{\rho^3},
\end{align*}
we conclude that $\Delta_{\mathrm{H}} \psi(\rho)=\frac{r^2}{\rho^2}\left(\frac{\mathrm{d}^2 \psi(\rho)}{\mathrm{d} \rho^2}+\frac{Q}{\rho} \frac{\mathrm{d} \psi(\rho)}{\mathrm{d} \rho}\right)$. By letting $\psi(\rho)=\left(1+\rho^2\right)^{-\frac{q}{2}}$, we get
$$
\Delta_{\mathrm{H}} \left(1+\rho^2\right)^{-\frac{q}{2}}=\tfrac{r^2}{\rho^2} \left(C_1\left(1+\rho^2\right)^{-\frac{q}{2}-1}+C_2 \rho^2 \left(1+\rho^2\right)^{-\frac{q}{2}-2}\right),
$$
where $C_i=C_i\left(q, Q\right)$ are some suitable constants for any $i=1, 2$. Therefore, based on $r^2 \lesssim \rho^2$, we derive the following conclusion: $\left|\Delta_{\mathrm{H}} \left(1+\rho^2\right)^{-\frac{q}{2}}\right| \lesssim \left(1+\rho^2\right)^{-\frac{q}{2}-1}$. This completes the proof.
\end{proof}
\begin{lemma} \label{lemma:6.3}
Let $\langle \eta\rangle_{\mathbf{H}_n} := \left(1+|\eta|_{\mathbf{H}_n}^2\right)^{\frac{1}{2}}$ for all $\eta \in \mathbf{H}_n, s \in(0,1)$ and $q>0$. Then, the following estimates hold for all $\eta \in \mathbf{H}_n$:
$$
\left|(-\Delta_{\mathrm{H}})^s\langle \eta\rangle_{\mathbf{H}_n}^{-q}\right| \lesssim \begin{cases}\langle \eta\rangle_{\mathbf{H}_n}^{-q-2 s} & \text { if } 0<q<Q, \\ \langle \eta\rangle_{\mathbf{H}_n}^{-Q-2 s} \log (e+|\eta|_{\mathbf{H}_n}) & \text { if } q=Q, \\ \langle \eta\rangle_{\mathbf{H}_n}^{-Q-2 s} & \text { if } q>Q.\end{cases}
$$
\end{lemma}
\begin{proof}
Denoting by $\psi=\psi(\eta) := \langle \eta\rangle_{\mathbf{H}_n}^{-q}$ we write $(-\Delta_{\mathrm{H}})^s\langle \eta\rangle_{\mathbf{H}_n}^{-q}=$ $(-\Delta_{\mathrm{H}})^s(\psi)(\eta)$. According to Definition \ref{definition:6.1} of fractional Laplacian as a singular integral operator, we have
$$
(-\Delta_{\mathrm{H}})^s(\psi)(\eta) := C_{n, s} \mathrm{p.v.} \int_{\mathbf{H}_n} \tfrac{\psi(\eta)-\psi(\zeta)}{\left|\zeta^{-1} \circ \eta\right|_{\mathbf{H}_n}^{Q+2 s}} \mathrm{d} \zeta.
$$
A standard change of variables leads to
\begin{align*}
&(-\Delta_{\mathrm{H}})^s(\psi)(\eta)= -\tfrac{C_{n, s}}{2} \mathrm{p.v.} \int_{\mathbf{H}_n} \tfrac{\psi(\eta \circ \zeta)+\psi(\eta \circ \zeta^{-1})-2 \psi(\eta)}{|\zeta|_{\mathbf{H}_n}^{Q+2 s}} \mathrm{d} \zeta \\
= & -\tfrac{C_{n, s}}{2} \lim _{\varepsilon \rightarrow 0^+} \int_{\varepsilon \leq|\zeta|_{\mathbf{H}_n} \leq 1} \tfrac{\psi(\eta \circ \zeta)+\psi(\eta \circ \zeta^{-1})-2 \psi(\eta)}{|\zeta|_{\mathbf{H}_n}^{Q+2 s}} \mathrm{d} \zeta - \tfrac{C_{n, s}}{2} \int_{|\zeta|_{\mathbf{H}_n} \geq 1} \tfrac{\psi(\eta \circ \zeta)+\psi(\eta \circ \zeta^{-1})-2 \psi(\eta)}{|\zeta|_{\mathbf{H}_n}^{Q+2 s}} \mathrm{d} \zeta .
\end{align*}
To deal with the first integral, after using Proposition \ref{proposition:2.1} for any $\zeta=\left(z, \tau \right) \in \mathbf{H}_n$, we arrive at
\begin{equation} \label{eq:6.3.1}
\psi\left(\eta \circ \zeta^{-1}\right)=P_2(\psi, \eta)\left(\eta \circ \zeta^{-1}\right)+\mathrm{o}\left(|\zeta|_{\mathbf{H}_n}^3\right) \text { as } |\zeta|_{\mathbf{H}_n} \rightarrow 0,
\end{equation}
where $P_2(\psi, \eta)$ is the Taylor polynomial of $\mathbf{H}_n$-degree 2 associated to $\psi$ and centered at $\eta$ presented in Section \ref{section_2.1}. Also, by the very definition of Taylor polynomial, it follows
$$
P_2(\psi, \eta)\left(\eta \circ \zeta^{-1}\right) =P_2(\psi(\eta \circ \cdot), 0)\left(\zeta^{-1}\right) =\psi(\eta)-\left(\nabla_{\mathbf{H}_n} \psi(\eta), \partial_\tau \psi(\eta)\right) \cdot \zeta+\tfrac{1}{2}\left\langle z, D_{\mathbf{H}_n}^{2, *} \psi(\eta) \cdot z\right\rangle.
$$
Thus, equality \eqref{eq:6.3.1} yields
$$
\psi\left(\eta \circ \zeta^{-1}\right)= \psi(\eta)-\left(\nabla_{\mathbf{H}_n} \psi(\eta), \partial_\tau \psi(\eta)\right) \cdot \zeta+\tfrac{1}{2}\left\langle z, D_{\mathbf{H}_n}^{2, *} \psi(\eta) \cdot z\right\rangle +\mathrm{o}\left(|\zeta|_{\mathbf{H}_n}^3\right) \text { as }|\zeta|_{\mathbf{H}_n} \rightarrow 0.
$$
Similarly, we also obtain
$$
\psi\left(\eta \circ \zeta\right)= \psi(\eta)+\left(\nabla_{\mathbf{H}_n} \psi(\eta), \partial_\tau \psi(\eta)\right) \cdot \zeta+\tfrac{1}{2}\left\langle z, D_{\mathbf{H}_n}^{2, *} \psi(\eta) \cdot z\right\rangle +\mathrm{o}\left(|\zeta|_{\mathbf{H}_n}^3\right) \text { as }|\zeta|_{\mathbf{H}_n} \rightarrow 0.
$$
Therefore, we deduce that $\frac{\left|\psi(\eta \circ \zeta)+\psi(\eta \circ \zeta^{-1})-2 \psi(\eta)\right|}{|\zeta|_{\mathbf{H}_n}^{Q+2 s}} \lesssim \frac{\left\|D_{\mathbf{H}_n}^{2, *} \psi(\eta)\right\|_{L^{\infty}\left(\mathbf{H}_n\right)}}{|\zeta|_{\mathbf{H}_n}^{Q+2 s-2}}$. Thanks to the above estimate and $s \in(0,1)$, we may remove the principal value of the integral at the origin to conclude
$$
(-\Delta_{\mathrm{H}})^s(\psi)(\eta)=-\tfrac{C_{n, s}}{2} \int_{\mathbf{H}_n} \tfrac{\psi(\eta \circ \zeta)+\psi(\eta \circ \zeta^{-1})-2 \psi(\eta)}{|\zeta|_{\mathbf{H}_n}^{Q+2 s}} \mathrm{d} \zeta.
$$
To prove the desired estimates, we shall divide our considerations into two subcases. In the first subcase $\{\eta:|\eta|_{\mathbf{H}_n} \leq 1\}$, we can proceed as follows:
\begin{align*}
\left|(-\Delta_{\mathrm{H}})^s(\psi)(\eta)\right| \lesssim & \int_{|\zeta|_{\mathbf{H}_n} \leq 1} \tfrac{|\psi(\eta \circ \zeta)+\psi(\eta \circ \zeta^{-1})-2 \psi(\eta)|}{|\zeta|_{\mathbf{H}_n}^{Q+2 s}} \mathrm{d} \zeta+\int_{|\zeta|_{\mathbf{H}_n} \geq 1} \tfrac{|\psi(\eta \circ \zeta)+\psi(\eta \circ \zeta^{-1})-2 \psi(\eta)|}{|\zeta|_{\mathbf{H}_n}^{Q+2 s}} \mathrm{d} \zeta \\
\lesssim & \left\|D_{\mathbf{H}_n}^{2, *} \psi(\eta)\right\|_{L^{\infty}\left(\mathbf{H}_n\right)} \int_{|\zeta|_{\mathbf{H}_n} \leq 1} \tfrac{1}{|\zeta|_{\mathbf{H}_n}^{Q+2 s-2}} \mathrm{d} \zeta+\|\psi\|_{L^{\infty}\left(\mathbf{H}_n\right)} \int_{|\zeta|_{\mathbf{H}_n} \geq 1} \tfrac{1}{|\zeta|_{\mathbf{H}_n}^{Q+2 s}} \mathrm{d} \zeta.
\end{align*}
Due to the boundedness of the above two integrals, it follows immediately
\begin{equation} \label{ineq:6.3.1.1}
\left|(-\Delta_{\mathrm{H}})^s(\psi)(\eta)\right| \lesssim 1 \text { for all } |\eta|_{\mathbf{H}_n} \leq 1.
\end{equation}
In order to deal with the second subcase $\{\eta \in \mathbf{H}_n: |\eta|_{\mathbf{H}_n} \geq 1\}$, we can re-write
\begin{align} \label{eq:6.3.2}
(-\Delta_{\mathrm{H}})^s(\psi)(\eta)= & -\tfrac{C_{n, s}}{2} \int_{|\zeta|_{\mathbf{H}_n} \geq 2|\eta|_{\mathbf{H}_n}} \tfrac{\psi(\eta \circ \zeta)+\psi(\eta \circ \zeta^{-1})-2 \psi(\eta)}{|\zeta|_{\mathbf{H}_n}^{Q+2 s}} \mathrm{d} \zeta \notag\\
& -\tfrac{C_{n, s}}{2} \int_{\frac{1}{2}|\eta|_{\mathbf{H}_n} \leq|\zeta|_{\mathbf{H}_n} \leq 2|\eta|_{\mathbf{H}_n}} \tfrac{\psi(\eta \circ \zeta)+\psi(\eta \circ \zeta^{-1})-2 \psi(\eta)}{|\zeta|_{\mathbf{H}_n}^{Q+2 s}} \mathrm{d} \zeta \notag\\
& -\tfrac{C_{n, s}}{2} \int_{|\zeta|_{\mathbf{H}_n} \leq \frac{1}{2}|\eta|_{\mathbf{H}_n}} \tfrac{\psi(\eta \circ \zeta)+\psi(\eta \circ \zeta^{-1})-2 \psi(\eta)}{|\zeta|_{\mathbf{H}_n}^{Q+2 s}} \mathrm{d} \zeta .
\end{align}
For the first integral, we notice that the relations $|\eta \circ \zeta|_{\mathbf{H}_n} \geq|\zeta|_{\mathbf{H}_n}-|\eta|_{\mathbf{H}_n} \geq|\eta|_{\mathbf{H}_n}$ and $|\eta \circ \zeta^{-1}|_{\mathbf{H}_n} \geq|\zeta|_{\mathbf{H}_n}-|\eta|_{\mathbf{H}_n} \geq|\eta|_{\mathbf{H}_n}$ hold for $|\zeta|_{\mathbf{H}_n} \geq 2|\eta|_{\mathbf{H}_n}$. Since $\psi$ is a decreasing function in the variable $|\eta|_{\mathbf{H}_n}$ and $|\eta|_{\mathbf{H}_n} \approx\langle \eta\rangle_{\mathbf{H}_n}$ for all $|\eta|_{\mathbf{H}_n} \geq 1$, we obtain the following estimate:
\begin{align} \label{ineq:6.3.3}
&\left|\int_{|\zeta|_{\mathbf{H}_n} \geq 2|\eta|_{\mathbf{H}_n}} \tfrac{\psi(\eta \circ \zeta)+\psi(\eta \circ \zeta^{-1})-2 \psi(\eta)}{|\zeta|_{\mathbf{H}_n}^{Q+2 s}} \mathrm{d} \zeta\right| \leq 4|\psi(\eta)| \int_{|\zeta|_{\mathbf{H}_n} \geq 2|\eta|_{\mathbf{H}_n}} \tfrac{1}{|\zeta|_{\mathbf{H}_n}^{Q+2 s}} \mathrm{d} \zeta \notag\\
\lesssim & \langle \eta\rangle^{-q} \int_{|\zeta|_{\mathbf{H}_n} \geq 2|\eta|_{\mathbf{H}_n}} \tfrac{1}{|\zeta|_{\mathbf{H}_n}^{Q+2 s}} \mathrm{d} |\zeta|_{\mathbf{H}_n} \lesssim \langle \eta\rangle_{\mathbf{H}_n}^{-q}|\eta|_{\mathbf{H}_n}^{-2 s} \lesssim\langle \eta\rangle_{\mathbf{H}_n}^{-q-2 s},
\end{align}
It is clear that $|\zeta|_{\mathbf{H}_n} \approx|\eta|_{\mathbf{H}_n}$ in the second integral domain. Moreover, it follows
\begin{align*}
& \left\{\zeta: \tfrac{1}{2}|\eta|_{\mathbf{H}_n} \leq|\zeta|_{\mathbf{H}_n} \leq 2|\eta|_{\mathbf{H}_n}\right\} \subset\{\zeta:|\eta \circ \zeta|_{\mathbf{H}_n} \leq 3|\eta|_{\mathbf{H}_n}\}, \\
& \left\{\zeta: \tfrac{1}{2}|\eta|_{\mathbf{H}_n} \leq|\zeta|_{\mathbf{H}_n} \leq 2|\eta|_{\mathbf{H}_n}\right\} \subset\{\zeta:|\eta \circ \zeta^{-1}|_{\mathbf{H}_n} \leq 3|\eta|_{\mathbf{H}_n}\} .
\end{align*}
For this reason, we arrive at
\begin{align} \label{ineq:6.3.4}
& \left|\int_{\frac{1}{2}|\eta|_{\mathbf{H}_n} \leq|\zeta|_{\mathbf{H}_n} \leq 2|\eta|_{\mathbf{H}_n}} \tfrac{\psi(\eta \circ \zeta)+\psi(\eta \circ \zeta^{-1})-2 \psi(\eta)}{|\zeta|_{\mathbf{H}_n}^{Q+2 s}} \mathrm{d} \zeta\right| \lesssim |\eta|_{\mathbf{H}_n}^{-Q-2 s}\int_{|\eta \circ \zeta|_{\mathbf{H}_n} \leq 3|\eta|_{\mathbf{H}_n}} \psi(\eta \circ \zeta) \mathrm{d} \zeta \\
&+|\eta|_{\mathbf{H}_n}^{-Q-2 s}\int_{|\eta \circ \zeta^{-1}|_{\mathbf{H}_n} \leq 3|\eta|_{\mathbf{H}_n}} \psi(\eta \circ \zeta^{-1}) \mathrm{d} \zeta +\psi(\eta) |\eta|_{\mathbf{H}_n}^{-Q-2 s} \int_{\frac{1}{2}|\eta|_{\mathbf{H}_n} \leq|\zeta|_{\mathbf{H}_n} \leq 2|\eta|_{\mathbf{H}_n}} 1 \mathrm{d} \zeta \notag\\
\lesssim & |\eta|_{\mathbf{H}_n}^{-Q-2 s}\left(\int_{|\eta \circ \zeta|_{\mathbf{H}_n} \leq 3|\eta|_{\mathbf{H}_n}} \psi(\eta \circ \zeta) \mathrm{d} \zeta+\langle \eta\rangle_{\mathbf{H}_n}^{-q}|\eta|_{\mathbf{H}_n}^Q\right),
\end{align}
where we used the relation $\int_{|\eta \circ \zeta|_{\mathbf{H}_n} \leq 3|\eta|_{\mathbf{H}_n}} \psi(\eta \circ \zeta) \mathrm{d} \zeta=\int_{|\eta \circ \zeta^{-1}|_{\mathbf{H}_n} \leq 3|\eta|_{\mathbf{H}_n}} \psi(\eta \circ \zeta^{-1}) \mathrm{d} \zeta$. By the change of variables $r=|\eta \circ \zeta|_{\mathbf{H}_n}$, we apply the inequality $1+r^2 \geq \frac{(1+r)^2}{2}$ to get
\begin{align} \label{ineq:6.3.5}
\int_{|\eta \circ \zeta|_{\mathbf{H}_n} \leq 3|\eta|_{\mathbf{H}_n}} \psi(\eta \circ \zeta) \mathrm{d} \zeta \lesssim & \int_{r \leq 3|\eta|_{\mathbf{H}_n}}\left(1+r^2\right)^{-\frac{q}{2}} r^{Q-1} \mathrm{d} r \lesssim \int_{r \leq 3|\eta|_{\mathbf{H}_n}}(1+r)^{Q-q-1} \mathrm{d} r \notag\\
\lesssim &
\begin{cases}
(1+3|\eta|_{\mathbf{H}_n})^{Q-q} & \text { if } 0<q<Q, \\
\log (e+3|\eta|_{\mathbf{H}_n}) & \text { if } q=Q, \\
1 & \text { if } q>Q.
\end{cases}
\end{align}
By using the fact that $|\eta|_{\mathbf{H}_n} \approx\langle \eta\rangle_{\mathbf{H}_n}$ for all $|\eta|_{\mathbf{H}_n} \geq 1$, we obtain from \eqref{ineq:6.3.4} and \eqref{ineq:6.3.5} that
\begin{align} \label{ineq:6.3.6}
\left|\int_{\frac{1}{2}|\eta|_{\mathbf{H}_n} \leq|\zeta|_{\mathbf{H}_n} \leq 2|\eta|_{\mathbf{H}_n}} \tfrac{\psi(\eta \circ \zeta)+\psi(\eta \circ \zeta^{-1})-2 \psi(\eta)}{|\zeta|_{\mathbf{H}_n}^{Q+2 s}} \mathrm{d} \zeta\right| \lesssim \begin{cases}\langle \eta\rangle_{\mathbf{H}_n}^{-q-2 s} & \text { if } 0<q<Q, \\ \langle \eta\rangle_{\mathbf{H}_n}^{-Q-2 s} \log (e+3|\eta|_{\mathbf{H}_n}) & \text { if } q=Q, \\ \langle \eta\rangle_{\mathbf{H}_n}^{-Q-2 s} & \text { if } q>Q .\end{cases}
\end{align}
For the third integral in \eqref{eq:6.3.2}, using again the second order Taylor expansion for $\psi$ we obtain
\begin{align} \label{ineq:6.3.7}
&\left|\int_{|\zeta|_{\mathbf{H}_n} \leq \frac{1}{2}|\eta|_{\mathbf{H}_n}} \tfrac{\psi(\eta \circ \zeta)+\psi(\eta \circ \zeta^{-1})-2 \psi(\eta)}{|\zeta|_{\mathbf{H}_n}^{Q+2 s}} \mathrm{d} \zeta\right| \notag\\
\leq & \int_{|\zeta|_{\mathbf{H}_n} \leq \frac{1}{2}|\eta|_{\mathbf{H}_n}} \tfrac{|\psi(\eta \circ \zeta)+\psi(\eta \circ \zeta^{-1})-2 \psi(\eta)|}{|\zeta|_{\mathbf{H}_n}^{Q+2 s}} \mathrm{d} \zeta \lesssim \int_{|\zeta|_{\mathbf{H}_n} \leq \frac{1}{2}|\eta|_{\mathbf{H}_n}} \tfrac{\max _{\theta \in[0,1]}\left|D_{\mathbf{H}_n}^{2, *} \psi \left(\eta \circ \left(\theta \odot \zeta \right)^{\pm 1}\right)\right|}{|\zeta|_{\mathbf{H}_n}^{Q+2 s-2}} \mathrm{d} \zeta \notag\\
\lesssim & \int_{|\zeta|_{\mathbf{H}_n} \leq \frac{1}{2}|\eta|_{\mathbf{H}_n}} \tfrac{\max _{\theta \in[0,1]}\left\langle \eta \circ \left(\theta \odot \zeta\right)^{\pm 1}\right\rangle_{\mathbf{H}_n}^{-q-2}}{|\zeta|_{\mathbf{H}_n}^{Q+2 s-2}} \mathrm{d} \zeta \lesssim\langle \eta\rangle_{\mathbf{H}_n}^{-q-2} \int_{|\zeta|_{\mathbf{H}_n} \leq \frac{1}{2}|\eta|_{\mathbf{H}_n}}|\zeta|_{\mathbf{H}_n}^{1-2 s} \mathrm{d} |\zeta|_{\mathbf{H}_n} \lesssim\langle \eta\rangle_{\mathbf{H}_n}^{-q-2 s}.
\end{align}
Here, we used the relations $\left|\eta \circ \left(\theta \odot \zeta\right)^{\pm 1}\right|_{\mathbf{H}_n} \geq|\eta|_{\mathbf{H}_n}-\theta|\zeta|_{\mathbf{H}_n} \geq|\eta|_{\mathbf{H}_n}-\frac{1}{2}|\eta|_{\mathbf{H}_n}=\frac{1}{2}|\eta|_{\mathbf{H}_n}$ and $\left|D_{\mathbf{H}_n}^{2, *} \psi \left(\eta \circ \left(\theta \odot \zeta \right)^{\pm 1}\right)\right| \lesssim \left\langle \eta \circ \left(\theta \odot \zeta\right)^{\pm 1}\right\rangle_{\mathbf{H}_n}^{-q-2}$. From \eqref{eq:6.3.2}, \eqref{ineq:6.3.3}, \eqref{ineq:6.3.6} and \eqref{ineq:6.3.7}, we arrive at the following estimates for all $|\eta|_{\mathbf{H}_n} \geq 1$:
\begin{equation} \label{ineq:6.3.8}
\left|(-\Delta_{\mathrm{H}})^s(\psi)(\eta)\right| \lesssim
\begin{cases}
\langle \eta\rangle_{\mathbf{H}_n}^{-q-2 s} & \text { if } 0<q<Q, \\
\langle \eta\rangle_{\mathbf{H}_n}^{-Q-2 s} \log (e+3|\eta|_{\mathbf{H}_n}) & \text { if } q=Q, \\
\langle \eta\rangle_{\mathbf{H}_n}^{-Q-2 s} & \text { if } q>Q.
\end{cases}
\end{equation}
Summarizing, combining \eqref{ineq:6.3.1.1} and \eqref{ineq:6.3.8}, the proof of Lemma \ref{lemma:6.3} is completed.
\end{proof}
\begin{lemma} \label{lemma:6.4}
Let $s \in(0,1)$. Let $\psi$ be a smooth function on $\mathbf{H}_n$ such that all second-order derivatives in the directions of $X_j, Y_j$ are bounded, i.e., $X_j X_k \psi, X_j Y_k \psi, Y_j Y_k \psi \in L^{\infty}\left(\mathbf{H}_n\right) \text{ for all } j, k=1, \ldots, n$. For any $R>0$, let $\psi_R$ be a function defined by $\psi_R(\eta) := \psi\left(R^{-1} \odot \eta\right)$ for all $\eta \in \mathbf{H}_n$, where $R^{-1} \odot \eta := \delta_{R^{-1}}(\eta)$, that is: $R^{-1} \odot \eta=\left(R^{-1} x, R^{-1} y, R^{-2} t\right)$. Then, $(-\Delta_{\mathrm{H}})^s\left(\psi_R\right)$ satisfies the following scaling properties for all $\eta \in \mathbf{H}_n$: $
(-\Delta_{\mathrm{H}})^s\left(\psi_R\right)(\eta)=R^{-2 s}\left((-\Delta_{\mathrm{H}})^s \psi\right)\left(R^{-1} \odot \eta\right)$.
\end{lemma}
\begin{proof}
Thanks to the assumption $X_j X_k \psi, X_j Y_k \psi, Y_j Y_k \psi \in L^{\infty}\left(\mathbf{H}_n\right) \text{ for all } j, k=1, \ldots, n.$, following the proof of Lemma \ref{lemma:6.3} we may remove the principal value of the integral at the origin to conclude
\begin{align*}
(-\Delta_{\mathrm{H}})^s \left(\psi_R\right)(\eta) =&\tfrac{C_{n, s}}{2} \int_{\mathbf{H}_n} \tfrac{\psi_R(\eta \circ \zeta)+\psi_R(\eta \circ \zeta^{-1})-2 \psi_R(\eta)}{|\zeta|_{\mathbf{H}_n}^{Q+2 s}} \mathrm{d} \zeta \\
=&-\tfrac{C_{n, s}}{2} R^{-2 s} \int_{\mathbf{H}_n} \tfrac{\psi\left(R^{-1} \odot \eta+R^{-1} \odot \zeta\right)+\psi\left(R^{-1} \odot \eta-R^{-1} \cdot \zeta\right)-2 \psi \odot \left(R^{-1} \eta\right)}{\left|R^{-1} \odot \zeta\right|_{\mathbf{H}_n}^{Q+2 s}} d\left(R^{-1} \odot \zeta\right) \\
=&R^{-2 s}\left((-\Delta_{\mathrm{H}})^s \psi\right)\left(R^{-1} \odot \eta\right),
\end{align*}
where we have used Definition \ref{definition:2.1}. This completes the proof.
\end{proof}
Before proceeding to prove Theorem \ref{theorem:1.5}, we briefly revisit the definition of a weak solution to \eqref{eq:1.0}.
\begin{definition}
Let $p>1$. A weak solution to the Cauchy problem \eqref{eq:1.0} in $[0, T) \times \mathbf{H}_n$ is defined as a function $u \in L_{\mathrm{loc}}^p\left([0, T), L^p\left(\mathbf{H}_n,\langle \eta \rangle_{\mathbf{H}_n}^{-q} d \eta \right)\right)$ that fulfills the condition
\begin{align*}
& \int_0^T \int_{\mathbf{H}_n}|u(t, \eta)|^p \phi(t, \eta) \mathrm{d} \eta \mathrm{d} t+\int_{\mathbf{H}_n}\left(\left(-\Delta_{\mathrm{H}}\right)^{\delta_2}u_0(\eta)+u_1(\eta)\right) \phi(0, \eta) \mathrm{d} \eta-\int_{\mathbf{H}_n} u_0(\eta) \partial_t \phi(0, \eta) \mathrm{d} \eta \\
=&\int_0^T \int_{\mathbf{H}_n} u(t, \eta)\left(\partial_t^2 \phi(t, \eta)+\left(-\Delta_{\mathrm{H}}\right)^{\delta_1} \phi(t, \eta)-\left(-\Delta_{\mathrm{H}}\right)^{\delta_2}\partial_t \phi(t, \eta)\right) \mathrm{d} \eta \mathrm{d} t
\end{align*}
hold, for any test function $\phi(t, \eta)$ of the form $\phi(t, \eta)=\chi(t) \varphi(\eta)$, with $\chi \in \mathcal{C}_0^{\infty}\left([0, T)\right)$ ($\chi \equiv 1$ in a neighborhood of $0$) and $\varphi(\eta) \in \mathcal{C}_q^{\infty}\left(\mathbf{H}_n\right)$. If $T=\infty$, the function $u$ is called a global (in time) weak solution to \eqref{eq:1.0}; otherwise, it is referred to as a local (in time) weak solution.
\end{definition}
Next, we focus on proving Theorem \ref{theorem:1.5}.
\begin{proof}[Proof of Theorem \ref{theorem:1.5}] 
We divide the proof of Theorem \ref{theorem:1.5} into several cases. 
\subsection{The case that the parameter $\delta_1=1$ and the parameter $\delta_2$ is fractional from $(0, 1)$} \label{section_7.1.1}
First, we introduce the functions $\varphi=\varphi\left(|\eta|_{\mathbf{H}_n}\right) := \langle \eta\rangle_{\mathbf{H}_n}^{-Q-2 \delta_2}$ and $\chi=\chi(t)$ with the following properties:
\begin{align} \label{ineq:7.2}
\bullet \quad & \chi \in \mathcal{C}_0^{\infty}([0, \infty)) \text{ and } \chi(t)= \begin{cases}1 & \text { if } 0 \leq t \leq \frac{1}{2}, \\ \text { decreasing } & \text { if } \frac{1}{2} \leq t \leq 1, \\ 0 & \text { if } t \geq 1, \end{cases} \notag\\
\bullet \quad & \chi^{-\frac{p^{\prime}}{p}}(t)\left(\left|\chi^{\prime}(t)\right|^{p^{\prime}}+\left|\chi^{\prime \prime}(t)\right|^{p^{\prime}}\right) \leq C \text{ for any } t \in \left[\tfrac{1}{2},1 \right],
\end{align}
where $p^{\prime}$ is the conjugate of $p>1$ and $C$ is a suitable positive constant.
Let $R$ be a large parameter in $[0, \infty)$. We define the test function $\phi_R(t, \eta)=\chi_R(t) \varphi_R(\eta)$, where $\chi_R(t) := \chi\left(R^{-2+2 \delta_2} t\right)$ and $\varphi_R(\eta) := \varphi\left(R^{-1} \odot \eta\right)$. We consider $\mathcal{I}_R := \int_0^{\infty} \int_{\mathbf{H}_n}|u(t, \eta)|^p \phi_R(t, \eta) \mathrm{d} \eta \mathrm{d} t=\int_0^{R^{2-2 \delta_2}} \int_{\mathbf{H}_n}|u(t, \eta)|^p \phi_R(t, \eta) \mathrm{d} \eta \mathrm{d} t$ and $\mathcal{I}_{R, t} := \int_{\frac{R^{2-2 \delta_2}}{2}}^{R^{2-2 \delta_2}} \int_{\mathbf{H}_n}|u(t, \eta)|^p \phi_R(t, \eta) \mathrm{d} \eta \mathrm{d} t$. Suppose that $u=u(t, \eta)$ is a global (in time) weak solution to \eqref{eq:1.0}, then
\begin{align} \label{ineq:testfunction}
0 \leq \mathcal{I}_R= & -\int_{\mathbf{H}_n} u_1(\eta) \varphi_R(\eta) \mathrm{d} \eta +\int_{\frac{R^{2-2 \delta_2}}{2}}^{R^{2-2 \delta_2}} \int_{\mathbf{H}_n} u(t, \eta) \partial_t^2 \chi_R(t) \varphi_R(\eta) \mathrm{d} \eta \mathrm{d} t \notag\\
& +\int_0^{\infty} \int_{\mathbf{H}_n} u(t, \eta) \chi_R(t)\left(-\Delta_{\mathrm{H}}\right) \varphi_R(\eta) \mathrm{d} \eta \mathrm{d} t -\int_{\frac{R^{2-2 \delta_2}}{2}}^{R^{2-2 \delta_2}} \int_{\mathbf{H}_n} u(t, \eta)\partial_t \chi_R(t)\left(-\Delta_{\mathrm{H}}\right)^{\delta_2} \varphi_R(\eta) \mathrm{d} \eta \mathrm{d} t \notag\\
:= &-\int_{\mathbf{H}_n} u_1(\eta) \varphi_R(\eta) \mathrm{d} \eta+\mathcal{I}_1+\mathcal{I}_2-\mathcal{I}_3.
\end{align}
Applying H$\ddot{\text{o}}$lder's inequality with $\frac{1}{p}+\frac{1}{p^{\prime}}=1$ we may estimate $\mathcal{I}_1$ as follows:
$$
\left|\mathcal{I}_1\right| \lesssim \mathcal{I}_{R, t}^{\frac{1}{p}}\left(\int_{\frac{R^{2-2 \delta_2}}{2}}^{R^{2-2 \delta_2}} \int_{\mathbf{H}_n} \chi_R^{-\frac{p^{\prime}}{p}}(t)\left|\partial_t^2 \chi_R(t)\right|^{p^{\prime}} \varphi_R(\eta) \mathrm{d} \eta \mathrm{d} t\right)^{\frac{1}{p^{\prime}}} .
$$
By the change of variables $\tilde{t} := R^{-2+2 \delta_2} t$ and $\tilde{\eta} := R^{-1} \odot \eta$, a straight-forward calculation gives
\begin{equation} \label{ineq:J_1}
\left|\mathcal{I}_1\right| \leq C \mathcal{I}_{R, t}^{\frac{1}{p}} R^{-4+4\delta_2+\frac{Q+2-2\delta_2}{p^{\prime}}}\left(\int_{\mathbf{H}_n}\langle \tilde{\eta}\rangle_{\mathbf{H}_n}^{-Q-2 \delta_2} \mathrm{d} \tilde{\eta}\right)^{\frac{1}{p^{\prime}}}.
\end{equation}
Here we used $\partial_t^2 \chi_R(t)=R^{-4+4\delta_2} \chi^{\prime \prime}\left(\tilde{t}\right)$ and the inequality \eqref{ineq:7.2}. Now let us turn to estimate $\mathcal{I}_2$ and $\mathcal{I}_3$. Applying H$\ddot{\text{o}}$lder's inequality again as we estimated $\mathcal{I}_1$ leads to
$$
\left|\mathcal{I}_2\right| \leq \mathcal{I}_R^{\frac{1}{p}}\left(\int_0^{R^{2-2\delta_2}} \int_{\mathbf{H}_n} \chi_R(t) \varphi_R^{-\frac{p^{\prime}}{p}}(\eta)\left|\left(-\Delta_{\mathrm{H}}\right) \varphi_R(\eta)\right|^{p^{\prime}} \mathrm{d} \eta \mathrm{d} t\right)^{\frac{1}{p^{\prime}}}.
$$
In order to control the above integral, we rely on Lemmas \ref{lemma:6.2}, \ref{lemma:6.3} and \ref{lemma:6.4}. Namely, by the change of
variables $\tilde{t} := R^{-2+2\delta_2} t$ and $\tilde{\eta} := R^{-1} \odot \eta$, and using the inequality \eqref{ineq:7.2}, we arrive at
\begin{align*}
\left|\mathcal{I}_2\right| \lesssim & \mathcal{I}_R^{\frac{1}{p}} R^{-2+\frac{Q+2-2\delta_2}{p^{\prime}}}\left(\int_{0}^{1} \int_{\mathbf{H}_n} \chi\left(\tilde{t}\right) \varphi^{-\frac{p^{\prime}}{p}}\left(\tilde{\eta}\right)\left|\left(-\Delta_{\mathrm{H}}\right)\varphi\left(\tilde{\eta}\right)\right|^{p^{\prime}} \mathrm{d} \tilde{\eta} \mathrm{d} \tilde{t} \right)^{\frac{1}{p^{\prime}}} \\
\lesssim & \mathcal{I}_R^{\frac{1}{p}} R^{-2+\frac{Q+2-2\delta_2}{p^{\prime}}}\left(\int_{\mathbf{H}_n} \varphi^{-\frac{p^{\prime}}{p}}\left(\tilde{\eta}\right)\left|\left(-\Delta_{\mathrm{H}}\right)\varphi\left(\tilde{\eta}\right)\right|^{p^{\prime}} \mathrm{d} \tilde{\eta}\right)^{\frac{1}{p^{\prime}}}.
\end{align*}
where we note that $\left(-\Delta_{\mathrm{H}}\right) \varphi_R(\eta)=R^{-2}\left(-\Delta_{\mathrm{H}}\right) \varphi\left(\tilde{\eta}\right)$, since Lemma \ref{lemma:6.4}. Using Lemma \ref{lemma:6.2} implies the following estimate:
\begin{equation} \label{ineq:J_2}
\left|\mathcal{I}_2\right| \leq C \mathcal{I}_R^{\frac{1}{p}} R^{-2+\frac{Q+2-2\delta_2}{p^{\prime}}}\left(\int_{\mathbf{H}_n}\langle \tilde{\eta}\rangle_{\mathbf{H}_n}^{-Q-2 \delta_2-2 p^{\prime}} \mathrm{d} \tilde{\eta}\right)^{\frac{1}{p^{\prime}}} .
\end{equation}
Applying H$\ddot{\text{o}}$lder's inequality again as in estimating $\mathcal{I}_1$ leads to
$$
\left|\mathcal{I}_3\right| \leq \mathcal{I}_{R, t}^{\frac{1}{p}}\left(\int_{\frac{R^{2-2\delta_2}}{2}}^{R^{2-2\delta_2}} \int_{\mathbf{H}_n} \chi_R^{-\frac{p^{\prime}}{p}}(t)\left|\partial_t \chi_R(t)\right|^{p^{\prime}} \varphi_R^{-\frac{p^{\prime}}{p}}(\eta)\left|\left(-\Delta_{\mathrm{H}}\right)^{\delta_2} \varphi_R(\eta)\right|^{p^{\prime}} \mathrm{d} \eta \mathrm{d} t\right)^{\frac{1}{p^{\prime}}}.
$$
Next, again by the change of variables $\tilde{t} := R^{-2+2\delta_2} t$ and $\tilde{\eta} := R^{-1} \odot \eta$ and using Lemma \ref{lemma:6.4}, we estimate $\mathcal{I}_3$ as follows:
\begin{align*}
\left|\mathcal{I}_3\right| \lesssim & \mathcal{I}_{R, t}^{\frac{1}{p}} R^{-2+\frac{Q+2-2\delta_2}{p^{\prime}}} \left(\int_{\frac{1}{2}}^{1} \int_{\mathbf{H}_n} \chi^{-\frac{p^{\prime}}{p}}\left(\tilde{t}\right)\left|\chi^{\prime}\left(\tilde{t}\right)\right|^{p^{\prime}} \varphi^{-\frac{p^{\prime}}{p}}\left(\tilde{\eta}\right)\left|\left(-\Delta_{\mathrm{H}}\right)^{\delta_2}\varphi\left(\tilde{\eta}\right)\right|^{p^{\prime}} \mathrm{d} \tilde{\eta} \mathrm{d} \tilde{t}\right)^{\frac{1}{p^{\prime}}} \\
\lesssim & \mathcal{I}_{R, t}^{\frac{1}{p}} R^{-2+\frac{Q+2-2\delta_2}{p^{\prime}}} \left(\int_{\mathbf{H}_n} \varphi^{-\frac{p^{\prime}}{p}}\left(\tilde{\eta}\right)\left|\left(-\Delta_{\mathrm{H}}\right)^{\delta_2}\varphi\left(\tilde{\eta}\right)\right|^{p^{\prime}} \mathrm{d} \tilde{\eta}\right)^{\frac{1}{p^{\prime}}}.
\end{align*}
Here we used $\partial_t \chi_R(t)=R^{-2+2\delta_2} \chi^{\prime}\left(\tilde{t}\right)$ and the inequality \eqref{ineq:7.2}. To deal with the last integral, we apply Lemma \ref{lemma:6.3} with $q=Q+2 \delta_2$ and $\gamma=\delta_2$, that is, $m=0$ and $s=\delta_2$, to get
\begin{equation} \label{ineq:J_3}
\left|\mathcal{I}_3\right| \leq C \mathcal{I}_{R, t}^{\frac{1}{p}} R^{-2+\frac{Q+2-2\delta_2}{p^{\prime}}} \left(\int_{\mathbf{H}_n}\langle \tilde{\eta}\rangle_{\mathbf{H}_n}^{-Q-2 \delta_2} \mathrm{d} \tilde{\eta}\right)^{\frac{1}{p^{\prime}}}.
\end{equation}
Next, under our assumption \eqref{condition_u0u1}, for sufficiently large $R$, we obtain
\begin{equation} \label{ineq:8.10}
\int_{\mathbf{H}^n} u_1(\eta) \varphi_R(\eta) \mathrm{d} \eta \geq C_0 \int_{|\eta|_{\mathbf{H}_n} \leq R}\langle \eta\rangle_{\mathbf{H}_n}^{-Q\left(\frac{1}{m}+\frac{\gamma}{Q}\right)} \log (\mathrm{e}+|\eta|_{\mathbf{H}_n})^{-1} \mathrm{d} \eta \geq C_0 R^{Q-\frac{Q}{m}-\gamma}(\log R)^{-1}.
\end{equation}
Combining the estimates from \eqref{ineq:testfunction}, \eqref{ineq:J_1}, \eqref{ineq:J_2}, \eqref{ineq:J_3} and \eqref{ineq:8.10}, we may arrive at for all $R>R_0$
\begin{align*}
\mathcal{I}_R+C_0 R^{Q-\frac{Q}{m}-\gamma}(\log R)^{-1} \leq & C \mathcal{I}_{R, t}^{\frac{1}{p}} R^{-4+4\delta_2+\frac{Q+2-2\delta_2}{p^{\prime}}} + C \mathcal{I}_R^{\frac{1}{p}} R^{-2+\frac{Q+2-2\delta_2}{p^{\prime}}} + C \mathcal{I}_{R, t}^{\frac{1}{p}} R^{-2+\frac{Q+2-2\delta_2}{p^{\prime}}}.
\end{align*}
Therefore, $\mathcal{I}_R+C_0 R^{Q-\frac{Q}{m}-\gamma}(\log R)^{-1} \leq C \mathcal{I}_R^{\frac{1}{p}} R^{-2+\frac{Q+2-2\delta_2}{p^{\prime}}} \text{ for all } R>R_0$. Now, the contradiction arises from the inequalities:
$$
0 \leq\left(1-\tfrac{1}{p}\right) \mathcal{I}_R \leq \tfrac{C R^{Q+2-2\delta_2-2 p^{\prime}}}{p^{\prime}}-C_0 R^{Q-\frac{Q}{m}-\gamma}(\log R)^{-1}<0 \quad \text{ for } R \gg 1,
$$
as long as $R^{Q+2-2\delta_2-2 p^{\prime}}<\frac{C_0 p^{\prime}}{C} R^{Q-\frac{Q}{m}-\gamma}(\log R)^{-1}$, that is always valid if $Q+2-2\delta_2-2 p^{\prime}<Q-\frac{Q}{m}-\gamma$ or equivalently, $p<p_{\mathrm{Fuji}}\left(\frac{Q+m\gamma-2 m \delta_2}{m}\right):=1+\frac{2 m}{Q+m\gamma-2 m \delta_2}$. This completes the proof of Theorem \ref{theorem:1.5} when parameter $\delta_1=1$ and the parameter $\delta_2$ is fractional from $(0, 1)$.

\subsection{The parameter $\delta_1$ is fractional from $(0, 1)$ and the parameter $\delta_2$ is fractional from $(0,1)$}
We apply the arguments analogous to the proof of Section \ref{section_7.1.1}. First, we denote $\delta^* := \min \left\{\delta_1, \delta_2\right\}$. It is obvious that $\delta^*$ is fractional from $(0,1)$. Let us introduce test functions $\chi=\chi(t)$ as in Section \ref{section_7.1.1} and $\varphi=\varphi(\eta) := \langle \eta\rangle_{\mathbf{H}_n}^{-Q-2 \delta^*}$. Then, by repeating the arguments in Section \ref{section_7.1.1} and applying Lemma \ref{lemma:6.3} instead of Lemma \ref{lemma:6.2}, we are able to conclude the statement that needs to be proved.

\subsection{The parameter $\delta_1=1$ and the parameter $\delta_2=0$}
The proof of this case can be found in \cite{Georgiev2020}. Let us define two bump functions $\alpha \in \mathcal{C}_0^{\infty}\left(\mathbb{R}^n\right)$ and $\beta \in \mathcal{C}_0^{\infty}(\mathbb{R})$ such that
$$
\alpha=1 \text { on } B_n\left(\tfrac{1}{2}\right) \text { with } \operatorname{supp} \alpha \subset B_n(1) \quad \text{ and } \quad \beta=1 \text { on }\left[-\tfrac{1}{4}, \tfrac{1}{4}\right] \text { with } \operatorname{supp} \beta \subset(-1,1).
$$
For a parameter $R>1$, we define the test function $\phi_R \in \mathcal{C}_0^{\infty}\left([0, \infty) \times \mathbb{R}^{2 n+1}\right)$ with variables separated as follows:
$$
\phi_R(t, x, y, \tau) := \beta\left(\tfrac{t}{R^2}\right) \alpha\left(\tfrac{x}{R}\right) \alpha\left(\tfrac{y}{R}\right) \beta\left(\tfrac{\tau}{R^2}\right), \quad(t, x, y, \tau) \in[0, \infty) \times \mathbb{R}^{2 n+1}.
$$
The remaining steps are carried out as in \cite{Georgiev2020}, from which the desired result follows. 
\end{proof}
\begin{remark}
\rm{
According to Theorem \ref{theorem:1.5}, if $\delta_1=1, \delta_2=0$ and $m=2$, then there is no global (in time) Sobolev solution $u \in \mathcal{C}\left([0, \infty), L^2\left(\mathbf{H}_n\right) \right)$ to the Cauchy problem \eqref{eq:1.0} when $1<p < p_{\mathrm{Fuji}}\left(\frac{Q+2\gamma}{2}\right):=1+\frac{4}{Q+2\gamma}$. We also observe that $p_{\mathrm{Fuji}}\left(\frac{Q+2\gamma}{2}\right):= 1+\frac{4}{Q+2\gamma}$ here coincides with the critical exponent in the case $\delta_1=1, \delta_2=0$ and $m=2$, as identified by Dasgupta, Kumar, Mondal and Ruzhansky in \cite[Theorem 1.5]{Dasgupta2024}.
}
\end{remark}

\end{document}